\documentclass[a4paper,notitlepage,twoside,reqno,12pt]{amsart}

\usepackage{bbm,pifont,latexsym}
\usepackage{anysize}  
\marginsize{2.3cm}{2.3cm}{2.2cm}{2.2cm}

\usepackage{dcolumn,indentfirst}
\usepackage[pdfpagemode=UseNone,colorlinks=true,linkcolor=black,citecolor=black,pdfstartview=FitH]{hyperref}
\usepackage{amsmath,amssymb,amscd,amsthm,amsfonts,mathrsfs}
\usepackage{color,graphicx,xcolor,graphics}
\usepackage{titlesec} 
\usepackage{titletoc} 
\titlecontents{section}[20pt]{\addvspace{2pt}\filright}
              {\contentspush{\thecontentslabel. }}
              {}{\titlerule*[8pt]{}\contentspage}

\usepackage{fancyhdr} 
\newtheorem{thm}{Theorem}[section]
\newtheorem{lema}[thm]{Lemma}
\newtheorem{cor}[thm]{Corollary}
\newtheorem{prop}[thm]{Proposition}

\theoremstyle{definition}

\newtheorem*{rmk}{Remark}

\newcommand{\DF}[2]{{\displaystyle\frac{#1}{#2}}}
\newcommand{\s}[0]{{\textup{\textbf{s}}}}
\newcommand{\I}{\mathbb{I}}
\newcommand{\D}{\mathbb{D}}
\newcommand{\T}{\mathbb{T}}
\newcommand{\R}{\mathbb{R}}

\newcommand{\C}{\mathbb{C}}
\newcommand{\HR}{\mathbb{H}}
\newcommand{\EC}{\widehat{\mathbb{C}}}

\newcommand{\Crit}{\textup{Crit}}

\newcommand{\diam}{\textup{diam}}
\newcommand{\dist}{\textup{dist}}

\newcommand{\dm}{d_{\max}}

\newcommand{\Jac}{\textup{Jac}}

\makeatletter\@addtoreset{equation}{section}\makeatother 

\titleformat{\section}{\centering\normalsize}{\textsc{\thesection.}}{0.5em}{\textsc}
\titleformat{\subsection}[runin]{\normalsize}{\thesubsection.}{0.3em}{\textbf}

\begin{document}

\author{WEIYUAN QIU}
\address{Weiyuan Qiu, School of Mathematical Sciences, Fudan University, Shanghai, 200433, P. R. China}
\email{wyqiu@fudan.edu.cn}

\author{FEI YANG}
\address{Fei Yang, Department of Mathematics, Nanjing University, Nanjing, 210093, P. R. China}
\email{yangfei\rule[-2pt]{0.2cm}{0.5pt}math@163.com}

\author{YONGCHENG YIN}
\address{Yongcheng Yin, Department of Mathematics, Zhejiang University, Hangzhou, 310027, P. R. China}
\email{yin@zju.edu.cn}

\title{QUASISYMMETRIC GEOMETRY OF THE CANTOR CIRCLES AS THE JULIA SETS OF RATIONAL MAPS}

\begin{abstract}
We give three families of parabolic rational maps and show that every Cantor set of circles as the Julia set of a non-hyperbolic rational map must be quasisymmetrically equivalent to the Julia set of one map in these families for suitable parameters. Combining a result obtained before, we give a complete classification of the Cantor circles Julia sets in the sense of quasisymmetric equivalence. Moreover, we study the regularity of the components of the Cantor circles Julia sets and establish a sufficient and necessary condition when a component of a Cantor circles Julia set is a quasicircle.
\end{abstract}

\subjclass[2010]{Primary 37F45; Secondary 37F20, 37F10}

\keywords{Julia sets; Cantor circles; quasisymmetrically equivalent}

\date{\today}



\maketitle


\section{Introduction}

Let $X$ and $Y$ be two metric spaces. If there exist a homeomorphism $f:X\rightarrow Y$ and a distortion control function $\eta:[0,\infty)\rightarrow [0,\infty)$ which is also a homeomorphism such that
\begin{equation}\label{defi-qs}
\frac{|f(x)-f(a)|}{|f(x)-f(b)|}\leq \eta\left(\frac{|x-a|}{|x-b|}\right)
\end{equation}
for every triple of distinct points $x,a,b\in X$, then $X$ and $Y$ are called \emph{quasisymetrically equivalent} to each other and $f$ is called a \textit{quasisymmetric} map between $X$ and $Y$ \cite[\S 10]{Hei}. A basic and widely open question in quasiconformal geometry is to determine whether two given homeomorphic spaces are quasisymmetrically equivalent to each other.

The quasisymmetric classification problem arises also in the hyperbolic spaces and word hyperbolic groups in the sense of Gromov \cite{BP,Hai,Kle}. For the definition of (word) hyperbolic groups, see \cite{Gro}. The Kapovich-Kleiner conjecture in the geometry group theory predicts that if a Gromov hyperbolic group $G$ has a boundary at infinity $\partial_\infty G$ that is a Sierpi\'{n}ski carpet, then $G$ should admit an action on a convex subset of hyperbolic 3-space $\mathbb{H}^3$ with non-empty totally geodesic boundary where the action is isometric, properly discontinuous, and cocompact \cite{KK}. This conjecture is equivalent to the following statement: if $\partial_\infty G$ is a Sierpi\'{n}ski carpet, then $\partial_\infty G$ is quasisymmetrically equivalent to a round Sierpi\'{n}ski carpet in the Riemann sphere $\EC$ (Recall that a Sierpi\'{n}ski carpet is called \textit{round} if each boundary component of this carpet is a round circle). Recently, Bonk gave a sufficient condition on the Sierpi\'{n}ski carpets such that they can quasisymmetrically equivalent to some round Sierpi\'{n}ski carpets \cite{Bon}.

For other quasisymmetrically inequivalent fractal sets, one can see also \cite{Bou} for the the examples of quasisymmetrically inequivalent spaces modelled on the universal Menger curve. The study of the topological properties of the Julia sets of rational maps is an important problem in complex dynamics. However, for the study of the quasisymmetric properties on the Julia sets of rational maps, very few results appeared in the literatures. Recently, Bonk, Lyubich and Merenkov studied the quasisymmetries of the Julia sets which are Sierpi\'{n}ski carpets \cite{BLM}. In \cite[Theorem 1]{HP}, Ha\"{i}ssinsky and Pilgrim gave the first example of homeomorphic but quasisymmetrically inequivalent hyperbolic Julia sets. They proved that there exist quasisymmetrically inequivalent Cantor circles as the Julia sets of rational maps.

In this paper, we will give a more comprehensive study on the quasiconformal geometry of Cantor circles that arise as the Julia sets of rational maps, including the quasisymmetric classification under the dynamics and the study of the regularity of the components of the Cantor circles Julia sets, etc.

Recall that a subset of the Riemann sphere $\overline{\mathbb{C}}$ is called a \textit{Cantor set of circles} (or \textit{Cantor circles} in short) if it consists of uncountably many closed Jordan curves which is homeomorphic to $\mathcal{C}\times \mathbb{S}^1$, where $\mathcal{C}$ is the middle third Cantor set and $\mathbb{S}^1$ is the unit circle. McMullen is the first one who found a rational map with this type Julia set (see \cite[\S 7]{McM}). He proved that the Julia set of $f(z)=z^2+\lambda/z^3$ is a Cantor set of circles if $\lambda\neq 0$ is small enough.
Henceforth, many authors have focused on the family
\begin{equation}\label{McM-family}
f_{\lambda}(z)=z^k+\lambda/z^l,
\end{equation}
which is commonly referred as the \emph{McMullen maps}, where $k,l\geq 2$ (see \cite{DLU,Ste,QWY} and the references therein). It is known that when $1/k+1/l<1$ and $\lambda\neq 0$ is small enough, then the Julia set of $f_\lambda$ is a Cantor set of circles (see \cite[\S 7]{McM} and \cite[\S 3]{DLU}).

As an answer to the natural question that to find more rational maps whose Julia sets are Cantor circles, the following Theorem \ref{thm-QYY} was proved in \cite{QYY}.

\begin{thm}[{\cite{QYY}}]\label{thm-QYY}
For each $p\in\{0,1\}$ and $n\geq 2$ positive integers $d_1,\cdots,d_n$ satisfying $\sum_{i=1}^{n}(1/d_i)<1$, there exist suitable parameters $a_1,\cdots,a_{n-1}$ such the Julia set of
\begin{equation}\label{family-QYY}
f_{p,d_1,\cdots,d_n}(z)=z^{(-1)^{n-p} d_1}\prod_{i=1}^{n-1}(z^{d_i+d_{i+1}}-a_i^{d_i+d_{i+1}})^{(-1)^{n-i-p}}
\end{equation}
is a Cantor set of circles. Moreover, any rational map whose Julia set is a Cantor set of circles must be topologically conjugate to $f_{p,d_1,\cdots,d_n}$ for some $p$ and $d_1,\cdots,d_n$ on their corresponding Julia sets with suitable parameters $a_1,\cdots,a_{n-1}$.
\end{thm}

From the topological point of view, all Cantor circles are the same since they are all topologically equivalent (homeomorphic) to the `standard' Cantor circles $\mathcal{C}\times \mathbb{S}^1$. Theorem \ref{thm-QYY} gives a complete topological classification of the Cantor circles as the Julia sets of rational maps under the dynamical behaviors. To obtain much richer structure of the set of all Cantor sets of circles, we can look at the Cantor circles Julia sets equipped with metric from the point of view of quasisymmetric geometry. Therefore, the question to give a classification of the Cantor circles as the Julia sets of rational maps in the sense of quasisymmetric equivalence arises naturally.

A rational map is \textit{hyperbolic} if all critical points are attracted by attracting periodic orbits. It is known that the maps in the family \eqref{family-QYY} are hyperbolic and the Julia components are quasicircles if the parameters $a_1,\cdots,a_{n-1}$ are chosen properly such the Julia set of $f_{p,d_1,\cdots,d_n}$ is a Cantor set of circles. However, there maybe exist cusps on the components of the Cantor circles Julia sets since the periodic Fatou components of the corresponding rational maps can be parabolic. According to \eqref{defi-qs}, an orientation preserving homeomorphism between a quasicircle and a Jordan curve with cusps cannot be quasisymmetric since the former has bounded \textit{turning} but the latter does not have \cite[Theorem 8.6]{LV}.

In this paper, we first give the specific expressions of three families of parabolic rational maps and prove that the Julia sets of them are Cantor circles if the parameters are chosen properly. Then, we prove that every Cantor circles as the Julia set of a non-hyperbolic rational map must be quasisymmetrically equivalent to the Julia set of one map in these three families (see Theorem \ref{this-is-all-resta}). Combining Theorem \ref{thm-QYY}, we give a complete quasisymmetric classification of the Cantor circles Julia sets under the dynamical behaviors (see Theorem \ref{this-is-all-intro} and Table \ref{Tab-classif}).

Let $d_1,\cdots,d_n$ be $n\geq 2$ positive integers such that $\sum_{i=1}^{n}(1/d_i)<1$. We define
\begin{equation}\label{family-para}
P_{d_1,\cdots,d_n}(z)=A_n\,\frac{d_1 z^{(-1)^{n-1} d_n}}{(d_1-1)z^{d_1}+1}
\prod_{i=1}^{n-1}(z^{d_i+d_{i+1}}-a_i^{d_i+d_{i+1}})^{(-1)^{i-1}}+B_n,
\end{equation}
where
\begin{equation*}\label{A-B-n}
 A_n=\frac{1}{1+C_n}\prod_{i=1}^{n-1}(1-a_i^{d_i+d_{i+1}})^{(-1)^i},B_n=\frac{C_n}{1+C_n}\text{~and~}
 C_n=\sum_{i=1}^{n-1}\frac{(-1)^{i-1}(d_i+d_{i+1})\,a_i^{d_i+d_{i+1}}}{1-a_i^{d_i+d_{i+1}}}
\end{equation*}
are three numbers depending only on the $n-1$ small \textit{complex} numbers $a_1,\cdots,a_{n-1}$ which satisfies $1\gg |a_1|\gg\cdots\gg |a_{n-1}|>0$. Here the notation `$a\gg b>0$' means `$a/b$ is a very big number. The choice of $A_n, B_n$ and $C_n$ here can guarantee that $1$ is a parabolic fixed point of $P_{d_1,\cdots,d_n}$ with multiplier $1$ (see Lemma \ref{para-fixed}(1)).

\begin{thm}\label{parameter-parabolic}
For each given $n\geq 2$ positive integers $d_1,\cdots,d_n$ satisfying $\sum_{i=1}^{n}(1/d_i)<1$, there exist suitable parameters $a_1,\cdots,a_{n-1}$, such that $P_{d_1,\cdots,d_n}$ has exactly one parabolic fixed point and the Julia set of $P_{d_1,\cdots,d_n}$ is a Cantor set of circles.
\end{thm}

Note that the parameters $a_1,\cdots,a_{n-1}$ are all very small since $1\gg |a_1|\gg\cdots\gg |a_{n-1}|>0$. This means that $C_n\approx 0$, $B_n\approx 0$ and $A_n\approx 1$ for fixed $d_1,\cdots,d_n$. For any $n\geq 2$, the rational map $P_{d_1,\cdots,d_n}$ can be seen as a small perturbation of the parabolic rational map
\begin{equation*}
h_{d_1}(z) = \frac{d_1 z^{d_1}}{(d_1-1)z^{d_1}+1}=\frac{1}{z}\circ \frac{z^{d_1}+d_1-1}{d_1}\circ\frac{1}{z}.
\end{equation*}
The map $h_{d_1}$ has a fixed parabolic basin containing $\infty$ whose boundary contains the parabolic fixed point $1$ and its Julia set is a Jordan curve with infinitely many cusps (Lemmas \ref{P_n_limit} and \ref{lema-h-n-h-mn}). If $n\geq 2$ is even, the Fatou component of $P_{d_1,\cdots, d_n}$ containing the origin is the preimage of the parabolic Fatou component containing the infinity. If $n\geq 3$ is odd, then the Fatou component of $P_{d_1,\cdots, d_n}$ containing the origin is an attracting Fatou component (see \S \ref{sec-para-1-fixed} for details).

For example, let $n=2$, $d_1=d_2=3$ and $a_1=0.25$ (As stated above, the parameter $a_1$ is required small enough. We choose $a_1=0.25$ here for generating more clear pictures of Julia sets). By a straightforward calculation, we have $A_2\approx 0.99878079$ and $B_2\approx 0.00146306$. The Julia set of
\begin{equation*}
P_{3,3}(z)=\frac{3 A_2 (z^6-a_1^6)}{z^3(2z^3+1)}+B_2,
\end{equation*}
is a Cantor set of circles and $P_{3,3}$ has a unique periodic Fatou component which is parabolic and contains $\infty$ (see left picture in Figure \ref{Fig_Cantor-cicle} and compare the left picture in Figure \ref{Fig_Cantor-cicle_local}).

Let $n=3$, $d_1=d_2=d_3=4$, $a_1=0.1$ and $a_2=0.01$. Then $A_3\approx 1-7\times 10^{-8}$ and $B_3\approx 8\times 10^{-8}$. The Julia set of
\begin{equation*}
P_{4,4,4}(z)=\frac{4 A_3 z^4(z^8-a_1^8)}{(3z^4+1)(z^8-a_2^8)}+B_3,
\end{equation*}
is a Cantor set of circles. Moreover, $P_{4,4,4}$ has an unbounded parabolic fixed Fatou component and a bounded attracting Fatou component containing the origin (see right picture in Figure \ref{Fig_Cantor-cicle} and compare the right picture in Figure \ref{Fig_Cantor-cicle_local}).

\begin{figure}[!htpb]
  \setlength{\unitlength}{1mm}
  \centering
  \includegraphics[width=70mm]{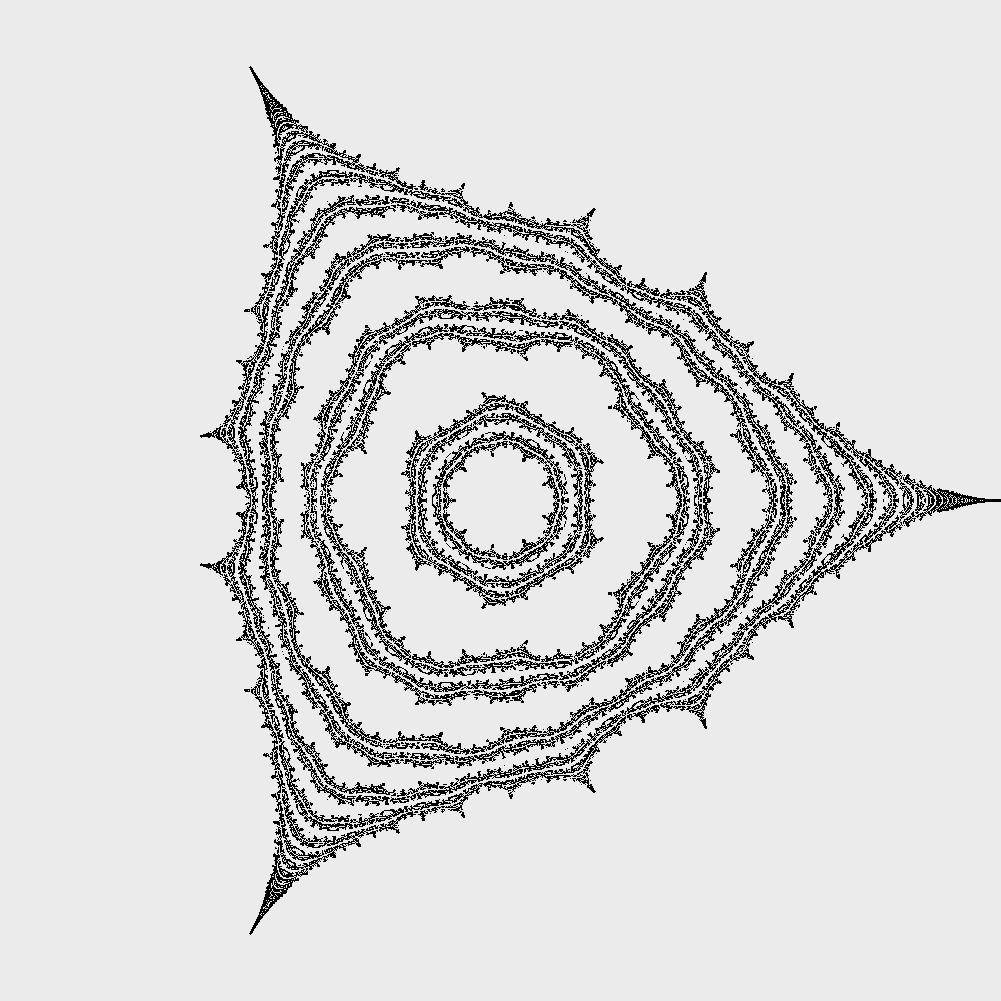}
  \includegraphics[width=70mm]{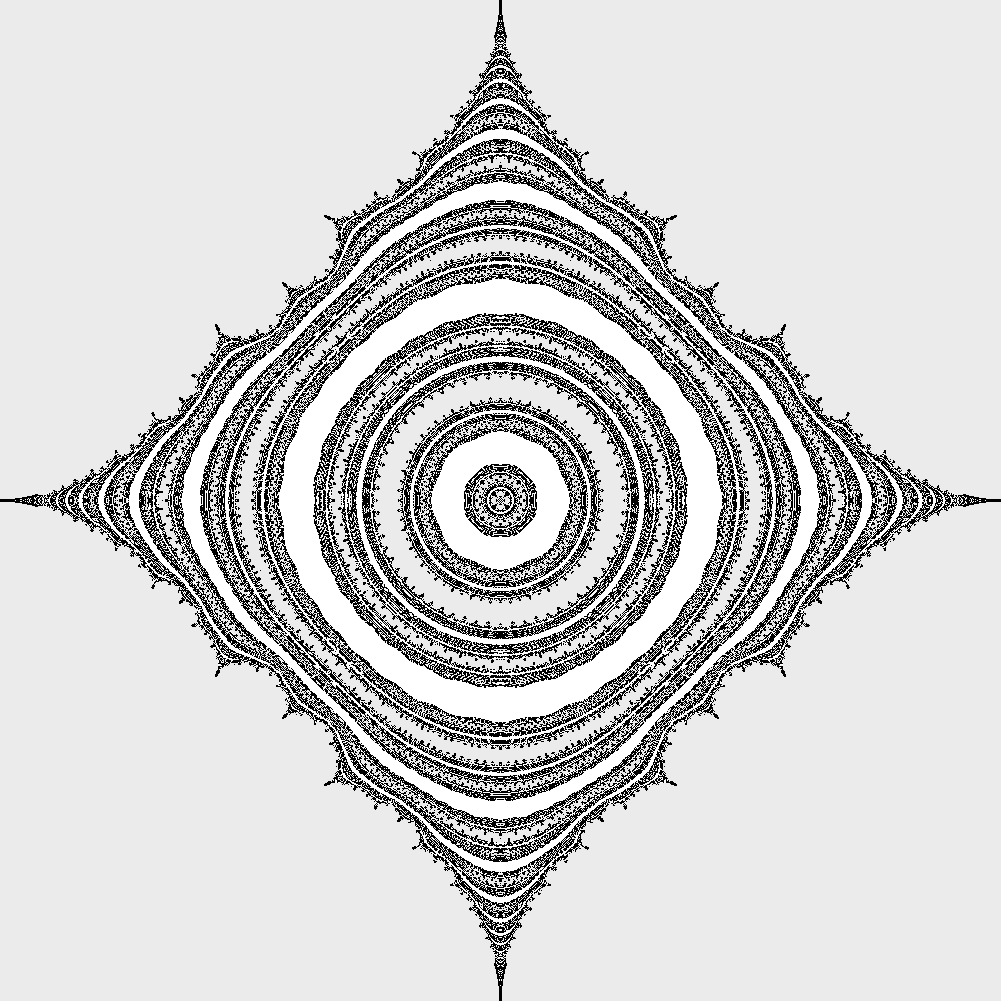}
  \caption{The Julia sets of $P_{3,3}$ and $P_{4,4,4}$ with suitable parameters (from left to right). The gray parts in both pictures denote the Fatou components which are eventually iterated to the unbounded  parabolic fixed Fatou components while the white parts in the right picture denote the Fatou components which are eventually iterated to the attracting Fatou component containing the origin (compare Figure \ref{Fig_Cantor-cicle_local}). Figure ranges: both are $[-1,1]^2$.}
  \label{Fig_Cantor-cicle}
\end{figure}

Note that $P_{d_1,\cdots,d_n}$ with odd $n\geq 3$ has a bounded attracting Fatou component containing the origin, we want to get a new map from the original one such that the attracting  Fatou component becomes parabolic. Define
\begin{equation}\label{family-para-fixed-intro}
Q_{d_1,\cdots,d_n}(z)=\frac{d_1 z^{d_n}}{(d_1-1)X_n z^{d_1}+Y_n z+Z_n}
\prod_{i=1}^{n-1}(z^{d_i+d_{i+1}}-b_i^{d_i+d_{i+1}})^{(-1)^{i-1}}+W_n,
\end{equation}
where $b_1,\cdots,b_{n-1}$ are $n-1$ small \textit{positive real} numbers satisfying $1\gg b_1\gg\cdots\gg b_{n-1}>0$ and $X_n\approx 1$, $Y_n\approx 0$, $Z_n\approx 1$, $W_n\approx 0$ are four real numbers depending only on $b_1,\cdots,b_{n-1}$ (see Lemmas \ref{lema-para-fixed-II} and \ref{zero-solution}).

\begin{thm}\label{parameter-parabolic-two-fixed}
For each odd $n\geq 3$ and positive integers $d_1,\cdots,d_n$ satisfying $\sum_{i=1}^{n}(1/d_i)<1$, there exist suitable parameters $b_1,\cdots,b_{n-1}$, such that $Q_{d_1,\cdots,d_n}$ has two fixed simply connected parabolic Fatou components and the Julia set of $Q_{d_1,\cdots,d_n}$ is a Cantor set of circles.
\end{thm}

Finally, we construct a family of parabolic rational maps whose Julia set are Cantor circles but their two simply connected Fatou components are periodic with period two. Define
\begin{equation}\label{family-para-R_n-intro}
R_{d_1,\cdots,d_n}(z)=\frac{S_n}{z^{d_n}} \prod_{i=1}^{n-1}(z^{d_i+d_{i+1}}-c_i^{d_i+d_{i+1}})^{(-1)^i}+T_n,
\end{equation}
where $c_1,\cdots,c_{n-1}$ are $n-1$ small \textit{positive real} numbers satisfying $1\gg c_1\gg\cdots\gg c_{n-1}>0$ and $S_n\approx 0$, $T_n\approx 0$ are numbers depending only on $c_1,\cdots,c_{n-1}$ (see Lemma \ref{lema-para-periodic-2}).

\begin{thm}\label{parameter-parabolic-two-periodic}
For each odd $n\geq 3$ and positive integers $d_1,\cdots,d_n$ satisfying $\sum_{i=1}^{n}(1/d_i)<1$, there exist suitable parameters $c_1,\cdots,c_{n-1}$, such that $R_{d_1,\cdots,d_n}$ has two simply connected periodic parabolic Fatou components with period two and the Julia set of $R_{d_1,\cdots,d_n}$ is a Cantor set of circles.
\end{thm}

Actually, $R_{d_1,\cdots,d_n}$ can be regarded as a rational map which comes from $f_{0,d_1,\cdots,d_n}$ for odd $n$ in some sense since they are actually topologically conjugate on their Julia sets for suitable parameters (see Table \ref{Tab-classif}). We will give the specific examples of $Q_{d_1,\cdots,d_n}$ and $R_{d_1,\cdots,d_n}$ with Cantor circles Julia sets in \S\ref{sec-2-para-fixed} and \S\ref{sec-para-period=2} respectively (see also Figures \ref{Fig_C-C-F} and \ref{Fig_C-C-F-periodic}).

According to Cui's result \cite{Cui}, each attracting Fatou component of a geometrically finite rational map $f$ can be `replaced' by a parabolic basin of a new map $g$ and the two rational maps $f$ and $g$ are topologically conjugate on their corresponding Julia sets. Therefore, by Theorem \ref{thm-QYY}, one can firmly believe that there exist the four kinds of parabolic rational maps with Cantor circles Julia sets (see Table \ref{Tab-classif}). However, the main work in the first part of this paper is to find their specific expressions.

Let $f$ be a rational map whose Julia set $J(f)$ is a Cantor set of circles. Then $f$ has exactly two simply connected Fatou components $D_0$ and $D_\infty$. By Riemann-Hurwitz's formula and Sullivan's eventually periodic theorem, we have $f(D_0\cup D_\infty)\subset D_0\cup D_\infty$.
Moreover, there exists no critical points in $J(f)$ since each Julia component is a Jordan closed curve (see \cite[Lemma 3.1]{QYY}). This means that each periodic Fatou component of $f$ must be attracting or parabolic. In fact, we have following quasisymmetric uniformization theorem of the Cantor circles as the Julia sets of rational maps.

\begin{thm}\label{this-is-all-intro}
Let $f$ be a rational map whose Julia set is a Cantor set of circles. Then there exist $n\geq 2$ positive integers $d_1,\cdots,d_n$ satisfying $\sum_{i=1}^{n}(1/d_i)<1$ such that $f$ is conjugate to either $f_{p,d_1,\cdots,d_n}$, $P_{d_1,\cdots,d_n}$, $Q_{d_1,\cdots,d_n}$ or $R_{d_1,\cdots,d_n}$ on their corresponding Julia sets for suitable parameters by a quasiconformal mapping.
\end{thm}

According to \cite[Theorem 11.14]{Hei}, an orientation preserving homeomorphism between the Riemann sphere is quasisymmetric if and only if it is quasiconformal. This means that Theorem \ref{this-is-all-intro} gives a complete dynamical classification of the Cantor circles Julia sets in the quasisymmetric sense and we have following Table \ref{Tab-classif}.

\begin{table}[htpb]
\renewcommand{\arraystretch}{1.1}
\begin{center}
\begin{tabular}{|c|c|c|c|c|c|}
 \hline
 Classifications  & $f(D_0)$     & $f(D_\infty)$   & $D_0$       & $D_\infty$   & models                                          \\ \hline\hline
 Hyperbolic I      & $D_\infty$   & $D_\infty$      & ---             & attracting    & $f_{1,d_1,\cdots,d_n}$, even $n$     \\ \hline
 Hyperbolic II     & $D_0$         & $D_\infty$      & attracting   & attracting    & $f_{1,d_1,\cdots,d_n}$, odd $n$       \\ \hline
 Hyperbolic III    & $D_\infty$   & $D_0$            & attracting   & attracting    & $f_{0,d_1,\cdots,d_n}$, odd $n$       \\ \hline\hline
 Parabolic I       & $D_\infty$   & $D_\infty$      & ---             & parabolic     & $P_{d_1,\cdots,d_n}$, even $n$         \\ \hline
 Parabolic II      & $D_0$         & $D_\infty$     & attracting    & parabolic     & $P_{d_1,\cdots,d_n}$, odd $n$          \\ \hline
 Parabolic III     & $D_0$         & $D_\infty$     & parabolic     & parabolic     & $Q_{d_1,\cdots,d_n}$, odd $n$          \\ \hline
 Parabolic IV     & $D_\infty$   & $D_0$           & parabolic     & parabolic      & $R_{d_1,\cdots,d_n}$, odd $n$          \\ \hline
\end{tabular}
\vskip0.2cm
 \caption{The quasisymmetric classifications of the Cantor circles as the Julia sets of rational maps. Every Cantor circles Julia set has a quasisymmetric model in this table.}
 \label{Tab-classif}
\end{center}
\end{table}

\vskip-0.5cm
Note that the definition of quasisymmetric equivalence has no relation to the dynamics. However, Theorem \ref{this-is-all-intro} provides a \textit{dynamical} quasisymmetric equivalence between Cantor circles Julia sets, i.e. the homeomorphism in \eqref{defi-qs} not only maps one Cantor circles Julia set to another, but also conjugates the corresponding dynamics. Unfortunately, we still don't know whether the Julia sets in the different families of the parabolic rational maps must be quasisymmetrically inequivalent without considering the dynamics and the difference of the degrees.

For the regularity of the Julia components of the Cantor circles Julia sets, we have following Theorem \ref{thm-regularity-intro}.

\begin{thm}\label{thm-regularity-intro}
Let $f$ be a rational map with Cantor circles Julia set and $J_0$ a Julia component of $f$. Then $J_0$ is a quasicircle if and only if the closure of the forward orbit of $J_0$ is disjoint with the boundaries of the parabolic periodic Fatou components.
\end{thm}

This theorem gives a sufficient and necessary condition to decide when a Julia component in Cantor circles Julia set is a quasicircle. Moreover, the quasicircle components are dense in the parabolic Cantor circles Julia set (see \S \ref{sec-regularity} and Figure \ref{Fig_Cantor-quasicircle}). As an immediate corollary of Theorem \ref{thm-regularity-intro}, we have following result.

\begin{cor}\label{cor-regularity-hyper}
Let $f$ be a hypebolic rational map whose Julia set is a Cantor set of circles. Then every Julia component of $f$ is a quasicircle.
\end{cor}

This paper is organized as follows: In \S \ref{sec-para-1-fixed}, we give the suitable parameters $a_1,\cdots,a_{n-1}$ in \eqref{family-para} and prove Theorem \ref{parameter-parabolic} after doing some estimations, including locating the positions of the critical orbits and controlling the size of the parabolic basins. In \S\S \ref{sec-2-para-fixed} and \ref{sec-para-period=2}, we show that the Julia sets of $Q_{d_1,\cdots,d_n}$ and $R_{d_1,\cdots,d_n}$ are Cantor circles if the parameters $b_1,\cdots,b_{n-1}$ and $c_1,\cdots,c_{n-1}$ are chosen properly and prove Theorems \ref{parameter-parabolic-two-fixed} and \ref{parameter-parabolic-two-periodic}. We will give the detailed proof of the existence of the four quantities $X_n,Y_n,Z_n$ and $W_n$ appeared in \eqref{family-para-fixed-intro} by using Newton's method. In \S \ref{sec-quasi-unifor}, we give the proof of quasisymmetric uniformization theorem of the Cantor circles as the Julia sets of rational maps and hence prove Theorem \ref{this-is-all-intro}. In \S \ref{sec-regularity}, we will prove Theorem \ref{thm-regularity-intro} and analyze the regularity of the Julia components of the parabolic Cantor circles Julia sets by dividing the arguments into four cases.

\vskip0.2cm
\noindent\textit{Acknowledgements.} The first author was supported by the NSFC under grant No.\,11271074 and the National Research Foundation for the Doctoral Program of Higher Education of China under grant No.\,20130071110029, the second author was supported by the Natural Science Foundation of Jiangsu Province under grant No. BK20140587 and the NSFC under grant No.\,11401298, and the third author was supported by the NSFC under grant No.\,11231009.

\section{Cantor circles with one parabolic fixed point}\label{sec-para-1-fixed}

In this section, we construct a family of non-hyperbolic rational maps whose Julia sets are Cantor circles such that each one of them has exactly one parabolic fixed point. Let $n\geq 2$ and $d_1,\cdots,d_n$ be $n$ positive integers such that $\sum_{i=1}^{n}(1/d_i)<1$ and $P_{d_1,\cdots,d_n}$ the family of rational maps defined in \eqref{family-para}. We use $\dm:=\max\{d_1,\cdots,d_n\}\geq 3$ to denote the maximal number among $d_1,\cdots,d_n$ and set
\begin{equation}\label{range-s}
|a_1|=(\dm^2 s)^{1/d_1} \text{~ and ~} |a_i|=s^{1/d_i}|a_{i-1}| \text{~for~} 2\leq i\leq n-1,
\end{equation}
where $s>0$ is small enough. Note that $a_1,\cdots,a_{n-1}$ are $n-1$ small complex parameters satisfying $1\gg |a_1|\gg \cdots\gg |a_{n-1}|>0$ if the parameter $s>0$ is sufficiently small.

Let $A(s)\geq 0$ and $B(s)\geq 0$ be two quantities depending only on $s>0$. We denote $A(s)\preceq B(s)$ if there exists a constant $C_0:=C_0(d_1,\cdots,d_n)\geq 0$ depending only on $d_1,\cdots,d_n$ such that $A(s)\leq C_0 B(s)$ when $s$ is small enough. Moreover, we use $A(s)\asymp B(s)$ to denote the two quantities in the same order of $s$ if there exists a constant $C_1:=C_1(d_1,\cdots,d_n)\geq 1$ depending only on $d_1,\cdots,d_n$ such that
\begin{equation*}
C_1^{-1}B(s)\leq A(s)\leq C_1 B(s)
\end{equation*}
when $s>0$ is small enough. We first give some estimations which will be useful in the later discussions.

\begin{lema}\label{very-useful-est}
$(1)$ For $1\leq i\leq n-1$,  then $|a_i|\asymp s^{\sum_{j=1}^i(1/d_j)}$.

$(2)$ If $1\leq j< i\leq n-1$, then $|a_i/a_j|^{d_i+d_{i+1}}\asymp s^{\alpha_1}$ and $|a_i/a_j|^{d_j+d_{j+1}}\asymp s^{\alpha_2}$, where $\alpha_1=\alpha_1(i,j)\geq 1+2/\dm$ and $\alpha_2=\alpha_2(i,j)\geq 1+2/\dm$.

$(3)$ Let $m\geq 2$ be an integer, $a\in\mathbb{C}\setminus\{0\}$ and $0<\varepsilon<1/2$. If $|z^m-a^m|\leq \varepsilon |a|^m$, then $|z-ae^{2\pi \textup{i} j/m}|< \varepsilon |a|$ for some $1\leq j\leq m$.
\end{lema}

\begin{proof}
(1) Since $|a_1|\asymp s^{1/d_1}$ and $|a_i|=s^{1/d_i}|a_{i-1}|$ for $2\leq i\leq n-1$ by \eqref{range-s}, we have
\begin{equation*}
|a_i|=s^{1/d_i}|a_{i-1}|=s^{(1/d_i)+(1/d_{i-1})}|a_{i-2}|=\cdots=s^{\sum_{j=2}^i(1/d_j)}|a_1| \asymp s^{\sum_{j=1}^i(1/d_j)}
\end{equation*}
for $i\geq 2$. Note that $|a_1|\asymp s^{1/d_1}$, the proof of (1) is complete.

(2) By (1), it can be seen that $|a_i/a_j|^{d_i+d_{i+1}}\asymp s^{\alpha_1}$, where
\begin{equation}\label{alpha_1}
\alpha_1:= (d_i+d_{i+1})\sum_{k=j+1}^i \frac{1}{d_k}
=1+ d_i\sum_{k=j+1}^{i-1} \frac{1}{d_k}+ d_{i+1}\sum_{k=j+1}^i \frac{1}{d_k}\geq 1+\frac{2}{\dm}
\end{equation}
since $2\leq d_k\leq \dm$ for $1\leq k\leq n$. Similarly, $|a_i/a_j|^{d_j+d_{j+1}}\asymp s^{\alpha_2}$, where
\begin{equation}\label{alpha_2}
\alpha_2:= (d_j+d_{j+1})\sum_{k=j+1}^i \frac{1}{d_k}
=1+ d_{j+1}\sum_{k=j+2}^{i} \frac{1}{d_k}+ d_j\sum_{k=j+1}^i \frac{1}{d_k}\geq 1+\frac{2}{\dm}.
\end{equation}

(3) Let $z^m=a^m(1+re^{i\theta})$ for $0\leq r\leq \varepsilon$ and $0\leq\theta<2\pi$, then $z=ae^{2\pi \textup{i}{j}/m}(1+re^{i\theta})^{1/m}$ for some $1\leq j\leq m$ and we have
\begin{equation*}
|z-ae^{2\pi \textup{i}{j}/m}|=|(1+re^{i\theta})^{1/m}-1|\cdot|a| \leq ((1+\varepsilon)^{1/m}-1)\cdot|a| < \varepsilon |a|
\end{equation*}
if $m\geq 2$. The proof is complete.
\end{proof}

For simplicity, we use $P_n$ to denote $P_{d_1,\cdots,d_n}$ for fixed integers $d_1,\cdots,d_n$ in the rest part of this section. For $1\leq i\leq n-1$, we denote $D_i:=d_i+d_{i+1}$. Then $5\leq D_i\leq 2\dm$. We now prove that $P_n$ is always parabolic with a parabolic fixed point $1$ and then use the order of the parameter $s$ to control the sizes of the terms $A_n$, $B_n$, $C_n$ which appeared in \eqref{family-para}.

\begin{lema}\label{para-fixed}
$(1)$ The map $P_n$ is parabolic. In particular, $P_n(1)=1$ and $P_n'(1)=1$.

$(2)$ If the parameter $s>0$ is small enough, then
\begin{equation*}
 |C_n|\preceq  s^{\beta},~~ |B_n|\preceq  s^{\beta} \text{~and~} |A_n-1| \preceq  s^{\beta},
\end{equation*}
where $\beta=\beta(d_1,\cdots,d_n)\geq 1+2/\dm$.
\end{lema}

\begin{proof}
(1) According to a straightforward calculation by following the expressions of $A_n$, $B_n$ and $C_n$ in the introduction, one can check that $P_n(1)=1$ directly. Note that
\begin{equation}\label{solu-crit-parabolic}
\frac{zP_n'(z)}{P_n(z)-B_n}
= \sum_{i=1}^{n-1}\frac{(-1)^{i-1}D_i z^{D_i}}{z^{D_i}-a_i^{D_i}}+(-1)^{n-1} d_n-\frac{(d_1-1)d_1z^{d_1}}{(d_1-1)z^{d_1}+1}.
\end{equation}
We have
\begin{equation*}\label{A-B-equation-1}
\begin{split}
   \frac{P_n'(1)}{P_n(1)-B_n}
= &~ \frac{P_n'(1)}{1-B_n} = \sum_{i=1}^{n-1}\frac{(-1)^{i-1}D_i\,a_i^{D_i}}{1-a_i^{D_i}}+\sum_{i=1}^{n-1}(-1)^{i-1}D_i +(-1)^{n-1} d_n -(d_1-1)\\
= &~ \sum_{i=1}^{n-1}\frac{(-1)^{i-1}D_i\,a_i^{D_i}}{1-a_i^{D_i}}+1=C_n+1.
\end{split}
\end{equation*}
Therefore, by the expressions of $B_n$ and $C_n$ in \eqref{family-para}, we have
\begin{equation*}\label{A-B-equation-2}
P_n'(1)=(1-B_n)(C_n+1)=1.
\end{equation*}
This means that $1$ is a parabolic fixed point of $P_n$ with multiplier $1$.

(2) By Lemma \ref{very-useful-est}(1), for $1\leq i\leq n-1$, we have
\begin{equation*}
|a_i|^{d_i+d_{i+1}}\asymp s^{\beta_i},
\end{equation*}
where
\begin{equation*}
\beta_i=(d_i+d_{i+1})\sum_{j=1}^i \frac{1}{d_j}
=1+d_i\sum_{j=1}^{i-1} \frac{1}{d_j}+d_{i+1}\sum_{j=1}^{i} \frac{1}{d_j}\geq 1+\frac{2}{\dm}.
\end{equation*}
Set $\beta:=\min_{1\leq i\leq n-1}\beta_i$. Then $|C_n|\preceq  s^{\beta}$, and hence $|B_n|\preceq  s^{\beta}$ and $|A_n-1| \preceq  s^{\beta}$.
\end{proof}

Let us explain the ideas behind the construction of $P_n$. For integer $d_1\geq 2$, we start with the unicritical parabolic polynomial $g_{d_1}(z)=(z^{d_1}+d_1-1)/d_1$. The holomorphic conjugacy of $g_{d_1}$ under $\psi(z)=1/z$ is
\begin{equation}\label{defi-h}
h_{d_1}(z):=\psi\circ g_{d_1}\circ\psi^{-1}(z)=\frac{d_1 z^{d_1}}{(d_1-1)z^{d_1}+1}.
\end{equation}
Then $\infty$ is a critical point of $h_{d_1}$ with multiplicity $d_1-1$ which is attracted to the parabolic fixed point $1$. In order to obtain a rational map whose Julia set is a Cantor set of circles, we want to replace the term $z^{d_1}$ in the numerator of $h_{d_1}$ by another term and add some new terms, which can guarantee that $1$ is always a parabolic fixed point of the new map after the substitution. The following lemma indicates that the rational map $P_n$ can be served as a small perturbation of $h_{d_1}$.

\begin{lema}\label{P_n_limit}
The rational map $P_n$ converges to $h_{d_1}$ locally uniformly on $\EC\setminus\{0\}$ as the parameter $s>0$ tends to zero.
\end{lema}

\begin{proof}
Let $A_n$ and $B_n$ be the two numbers in \eqref{family-para} which depend on the $n-1$ numbers $a_1,\cdots,a_{n-1}$. If $n\geq 2$ is even, then
\begin{equation}\label{family-para-even}
P_n(z)=\frac{A_n d_1 }{(d_1-1)z^{d_1}+1}\,\frac{z^{d_1+d_2}-a_1^{d_1+d_2}}{z^{d_2+d_3}-a_2^{d_2+d_3}}\cdots\frac{z^{d_{n-1}+d_n}-a_{n-1}^{d_{n-1}+d_n}}{z^{d_n}}+B_n.
\end{equation}
If $n\geq 3$ is odd, then
\begin{equation}\label{family-para-odd}
P_n(z)=\frac{A_n d_1 z^{d_n}}{(d_1-1)z^{d_1}+1}\,\frac{z^{d_1+d_2}-a_1^{d_1+d_2}}{z^{d_2+d_3}-a_2^{d_2+d_3}}
\cdots\frac{z^{d_{n-2}+d_{n-1}}-a_{n-2}^{d_{n-2}+d_{n-1}}}{z^{d_{n-1}+d_n}-a_{n-1}^{d_{n-1}+d_n}}+B_n.
\end{equation}
By Lemma \ref{very-useful-est}(1) and Lemma \ref{para-fixed}(2), it follows that $A_n$ tends to $1$, $a_1,\cdots,a_{n-1}$ and $B_n$ all tend to $0$ as $s>0$ tends to $0$. By the expressions of \eqref{family-para-even} and \eqref{family-para-odd}, this means that $P_n$ converges to $h_{d_1}$ locally uniformly on $\EC\setminus\{0\}$ as $s>0$ tends to zero.
\end{proof}

We now study the limit behaviors of the immediate parabolic basins of a sequence of parabolic rational maps which is locally uniformly convergent.
Let $f(z)=z+z^2+a_3 z^3+\cdots$ be a holomorphic germ defined in a domain $\Omega\subset\C$ such that the immediate parabolic basin $U$ of the origin of $f$ is compactly contained in $\Omega$. Note that $U$ is not necessarily simply connected. Suppose that $f_k(z)=z+a_{2,k} z^2+a_{3,k} z^3+\cdots$ is a sequence of holomorphic  germs defined in $\Omega$ which converges to $f$ locally uniformly, where $a_{2,k}\neq 0$ and $k\in\mathbb{N}$. We use $U_k\subset\Omega$ to denote the immediate parabolic basin of the origin of $f_k$.

\begin{lema}\label{lema-approx}
Let $K$ be a compact set in $U$. Then $K$ is contained in $U_k$ if $k$ is large enough. If $K$ is compact in $V=\{z:z\in U\text{~and~}\textup{Re}(z)<0\}\cup\{0\}$. Then $K\setminus\{0\}$ is contained in $U_k$ if $k$ is large enough.
\end{lema}

\begin{proof}
Define $w=\xi(z)=-1/z$ and denote $\widehat{\Omega}=\{\xi(z):z\in\Omega\setminus\{0\}\}$ and $\widehat{U}=\{\xi(z):z\in U\}$. Then $\widehat{\Omega}\supset\widehat{U}\supset\HR_R$ for some $R>0$, where $\HR_R=\{w:\textup{Re}(w)\geq R\}$ is a right half-plane \cite[\S 10]{Mil}. The maps $F=\xi\circ f\circ\xi^{-1}$ and $F_k=\xi\circ f_k\circ\xi^{-1}$ are defined on $\widehat{\Omega}$. By parabolic linearization theorem, there exists a conformal embedding $\alpha:\HR_R\to\C$ such that $\alpha(F(w))=\alpha(w)+1$ and $\textup{Re}(F(w))>\textup{Re}(w)+1/2$ if $R>0$ is large enough \cite[Theorem 10.9 and Lemma 10.10]{Mil}.

Since $f_k$ converges to $f$ locally uniformly on $\Omega$, it follows that $F_k$ converges to $F$ uniformly on $\HR_R$ for large $R>0$ since $0\in\Omega$. On the other hand, note that the map $f$ and each $f_k$ have a parabolic fixed point at the origin with multiplier $1$. Hence, there exists an integer $N_1>0$ such that $\textup{Re}(F_k(w))>\textup{Re}(w)+1/2$ for $w\in\HR_R$ if $k\geq N_1$. Therefore, $F_k(\HR_R)\subset\HR_{R+1/2}$ if $k\geq N_1$. Let $\mathcal{P}=\xi(\HR_R)$ be a parabolic petal in $U$. Then $f_k(\mathcal{P})\subset\mathcal{P}$ if $k\geq N_1$. This means that $\mathcal{P}$ is contained in the immediate parabolic basin of $0$ of $f_k$ if $k\geq N_1$.

Let $K$ be a compact set in $U$. Then there exists an integer $m\geq 0$ such that $F^{\circ m}(\xi(K))\subset\HR_R$. Equivalently, $f^{\circ m}(K)\subset\mathcal{P}$. Note that $f_k^{\circ m}$ converges to $f^{\circ m}$ uniformly on $K$ as $k\to\infty$. It follows that $f_k^{\circ m}(K)\subset\mathcal{P}$ if $k\geq N_2$ for some $N_2\geq 1$. Let $N=\max\{N_1,N_2\}$. Then we know that $K$ is contained in the parabolic basin of $0$ of $f_k$ if $k\geq N$. For the compact set $K$, there exists a continuous closed curve $\gamma\subset U$ connecting $K$ and $\mathcal{P}$. Note that $\gamma$ is also compact in $U$. This means that $\gamma$ is contained in the parabolic basin of $0$ of $f_k$. Above all, we know that $K$ is contained in the immediate parabolic basin of $0$ of $f_k$.

Suppose that $K$ is a compact set in $V$. Note that $V$ is the union of $0$ and an open subset of $U$. In this case $K$ is allowed to `touch' the parabolic fixed point $0$. By the definition of $V$, we have $\widehat{K}=\xi(K\setminus\{0\})\subset \HR_0\cap \widehat{U}$. Note that $\widehat{K}$ can be unbounded and hence is not compact in $\widehat{U}$. Suppose that $w\in\widehat{U}$. There exists a large $M>0$ such that if $|\textup{Im}(w)|>M$ or $w\in\HR_R$, then $F^{\circ n}(w)$ is well defined and $\textup{Re}(F^{\circ n}(w))>\textup{Re}(w)+n/2$ for every $n\geq 1$ (see \cite[\S 6.5]{Bea}). Then we can write the set $\widehat{K}$ as $K_1\cup K_2$ such that $K_1\subset\HR_0$ is compact and $K_2\subset\{w:|\textup{Im}(w)|>M\}\cap\HR_0$. If $w\in K_1$, then there exists an integer $m_1\geq 0$ such that $F^{\circ m_1}(K_1)\subset\HR_R$ since $K_1$ is compact and contained in $\widehat{U}$. If $w\in K_2$, there exists also an integer $m_2\geq 0$ such that $F^{\circ m_2}(K_2)\subset\HR_R$ since $\textup{Re}(F^{\circ n}(w))>\textup{Re}(w)+n/2$ for every $n\geq 1$. Hence, $F^{\circ l}(\widehat{K})\subset\HR_R$ if $l\geq m$, where $m= \max\{m_1,m_2\}$. Equivalently, $f^{\circ l}(K)\subset\mathcal{P}\cup\{0\}$.

Note that $f_k^{\circ l}$ converges to $f^{\circ l}$ uniformly on $K$ as $k\to\infty$. It follows that $f_k^{\circ l}(K)\subset\mathcal{P}\cup\{0\}$ if $k$ is large enough. This means that $K\setminus\{0\}$ is contained in the parabolic basin of $0$ of $f_k$ if $k$ is large enough. Similar argument can be shown that $K\setminus\{0\}$ is contained in the immediate parabolic basin of $0$ of $f_k$. The proof is complete.
\end{proof}

As a remark, we would like to point out that the immediate parabolic basin $U_k$ of the origin of $f_k$ needs not to converge to $U$ in the Hausdorff topology. For example, consider the quadratic rational map $f_c(z)=(z^2+cz)/(z+1)$
with a parabolic fixed point at $\infty$ with multiplier $1$. If $c=-1$, then $f_c$ has another parabolic fixed point $0$ with multiplier $-1$. Just do a suitable perturbation on $c$, the Julia set of new rational map is a Cantor set with a parabolic fixed point at $\infty$ (by parabolic implosion). This means that the immediate parabolic basin of $\infty$ of $f_c$ is not continuous at $c=-1$ in the Hausdorff topology.

Let $\mathbb{D}(a,r):=\{z\in\C:|z-a|<r\}$ be the Euclidean disk centered at $a$ with radius $r$. For $0\leq r\leq 1$, we denote $D_r:=\D(1-r,r)\subset\D$. Note that $1\in\partial D_r$.

\begin{lema}\label{lema-g-n-g-mn}
Let $m,n\geq 2$ be two integers. Then
\begin{equation}\label{g_m-g-mn}
g_n(z)=\frac{z^n+n-1}{n} \text{~and~} g_{m,n}(z)=\frac{(z^{m}+(mn-1))^{n}}{(mn)^{n}}
\end{equation}
are both parabolic polynomials containing a parabolic fixed point $1$ with multiplier $1$. For any $0\leq r\leq 1$, we have $g_n(\overline{D}_r)\subset D_r\cup\{1\}$ and $g_{m,n}(\overline{D}_r)\subset D_r\cup\{1\}$. In particular, the immediate parabolic basins of $1$ of $g_n$ and $g_{m,n}$ both contain the unit disk $\D$.
\end{lema}

\begin{proof}
It is easy to see that $1$ is the parabolic fixed point of $g_n$ and $g_{m,n}$ with multiplier $1$ and the attracting axis of the unique parabolic petal is on the left of $1$. If $z\in\overline{\D}$, then $|g_n(z)-(n-1)/n|=|z|^n/n\leq 1/n$. This means that $g_n(\overline{\D})\subset \overline{D}_{1/n}\cup\{1\}$. In particular, $g_n(\overline{D}_r)\subset \overline{D}_{1/n}\cup\{1\}\subset D_r\cup\{1\}$ if $1/n< r\leq 1$. Now suppose $0\leq r\leq 1/n$. Let $z=1-r+\rho e^{i\theta}\in \overline{D}_r$, where $\rho\in[0,r]$ and $\theta\in[0,2\pi)$. We have
\begin{equation}\label{est-parabolic-basin}
\begin{split}
      &~ \left|g_n(z)-\frac{(1-r)^n+n-1}{n}\right|=\frac{|(1-r+\rho e^{i\theta})^n-(1-r)^n|}{n} \\
\leq &~ \frac{1}{n}\left|\sum_{k=1}^nC_n^k\rho^ke^{ik\theta}(1-r)^{n-k}\right|\leq\frac{1}{n}\sum_{k=1}^nC_n^k\rho^k(1-r)^{n-k}\leq\frac{1-(1-r)^n}{n}.
\end{split}
\end{equation}
This means that $g_n(\overline{D}_r)\subset \overline{D}_s$, where $s=(1-(1-r)^n)/n\leq r$ since $0\leq r\leq 1/n$. Note that \eqref{est-parabolic-basin} can get the equality sign ``$=$" if and only if $\theta=0$ and $\rho=r$. This means that $g_n(\overline{D}_r)\subset D_r\cup\{1\}$ for $0\leq r\leq 1/n$. In summary, we have $g_n(\overline{D}_r)\subset D_r\cup\{1\}$ for $0\leq r\leq 1$. In particular, $\D$ is contained in the immediate parabolic basin of $1$ of $g_n$.

The conclusions for $g_{m,n}$ can be proved completely similar to that of $g_n$. We omit the details here. See Figure \ref{Fig_parabolic-polynomial} for the pictures of the Julia sets of $g_4$ and $g_{4,4}$.
\end{proof}

For $r\geq 1$, we denote $D'_r:=\EC\setminus \overline{\D}(1-r,r)\subset\EC\setminus\overline{\D}$. Note that $1\in\partial D_r'$.

\begin{lema}\label{lema-h-n-h-mn}
Let $m,n\geq 2$ be two integers. Then
\begin{equation}\label{new-parabolic-first}
h_n(z)=\frac{n z^n}{1+(n-1)z^n} \text{~and~} h_{m,n}(z)=\frac{(mn)^{n}z^{mn}}{(1+(mn-1)z^{m})^{n}}
\end{equation}
are both parabolic rational maps containing a parabolic fixed point $1$ with multiplier $1$. For any $r\geq 1$, we have $h_n(\overline{D_r'})\subset D_r'\cup\{1\}$ and $h_{m,n}(\overline{D_r'})\subset D_r'\cup\{1\}$. In particular, the immediate parabolic basin of $1$ contains the outside of the closed unit disk $D_1'=\EC\setminus\overline{\D}$.
\end{lema}

\begin{proof}
By a straightforward calculation, we have $h_n=\psi\circ g_n\circ \psi^{-1}$ and $h_{m,n}=\psi\circ g_{m,n}\circ \psi^{-1}$, where $\psi(z)=1/z$ and $g_n$ and $g_{m,n}$ are defined in \eqref{g_m-g-mn}. For $r\geq 1$, we have $\psi(D_r')=D_s$, where $s=r/(2r-1)\leq 1$. By Lemma \ref{lema-g-n-g-mn}, it is easy to see Lemma \ref{lema-h-n-h-mn} holds. See Figure \ref{Fig_parabolic-polynomial} for the pictures of the Julia sets of $h_4$ and $h_{4,4}$.
\end{proof}

\begin{figure}[!htpb]
  \setlength{\unitlength}{1mm}
  \centering
  \includegraphics[width=65mm]{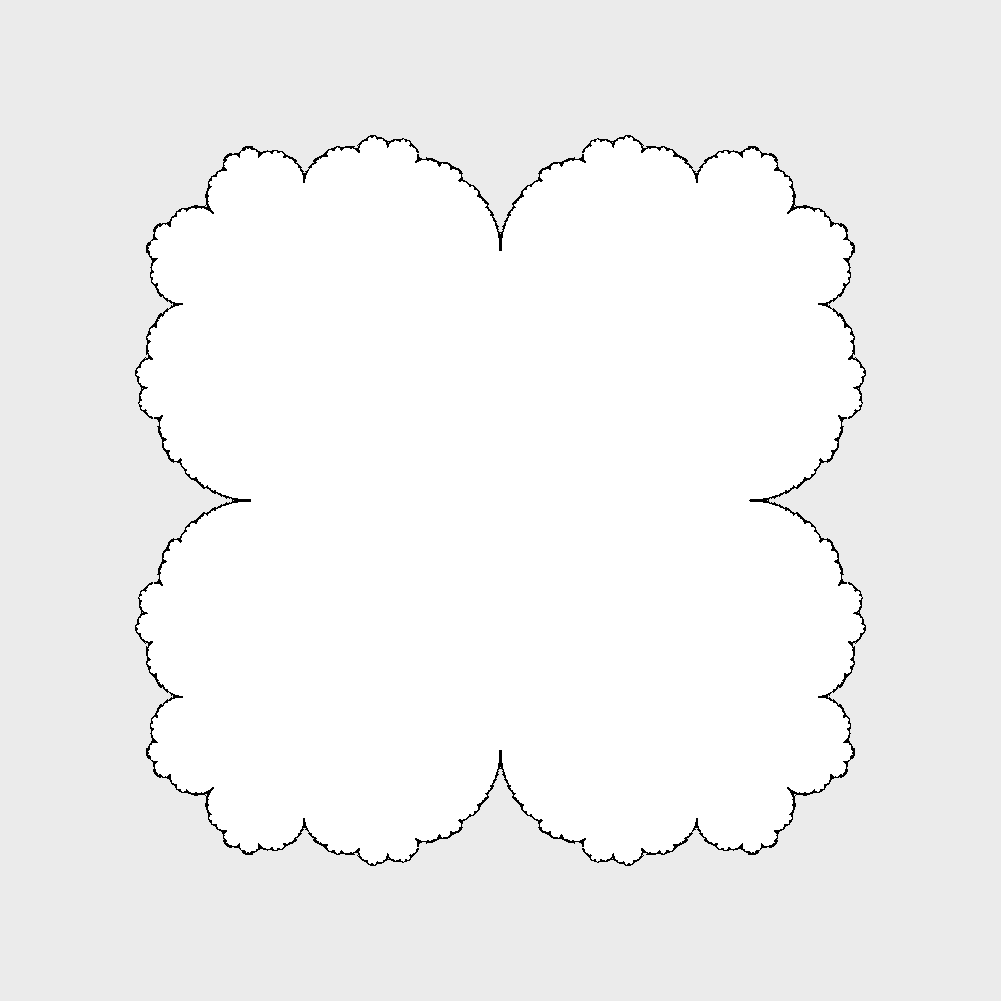}
  \includegraphics[width=65mm]{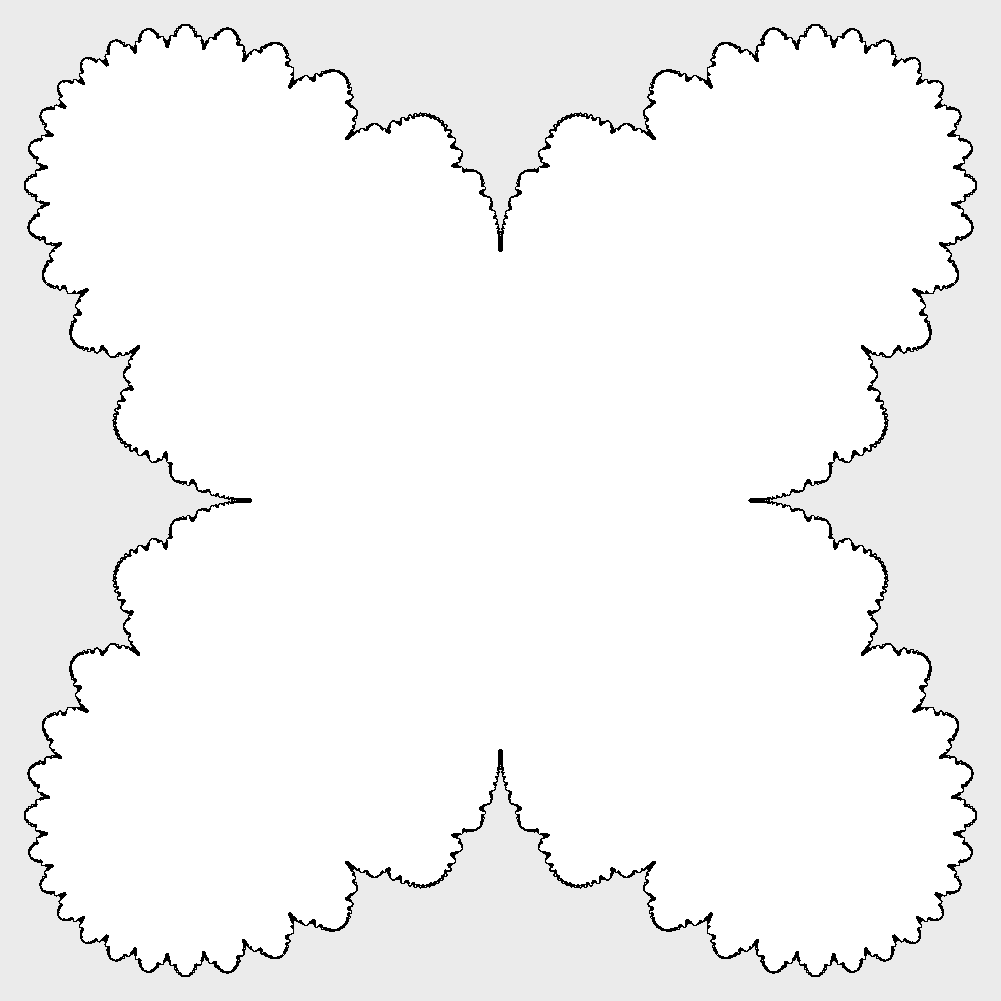}
  \includegraphics[width=65mm]{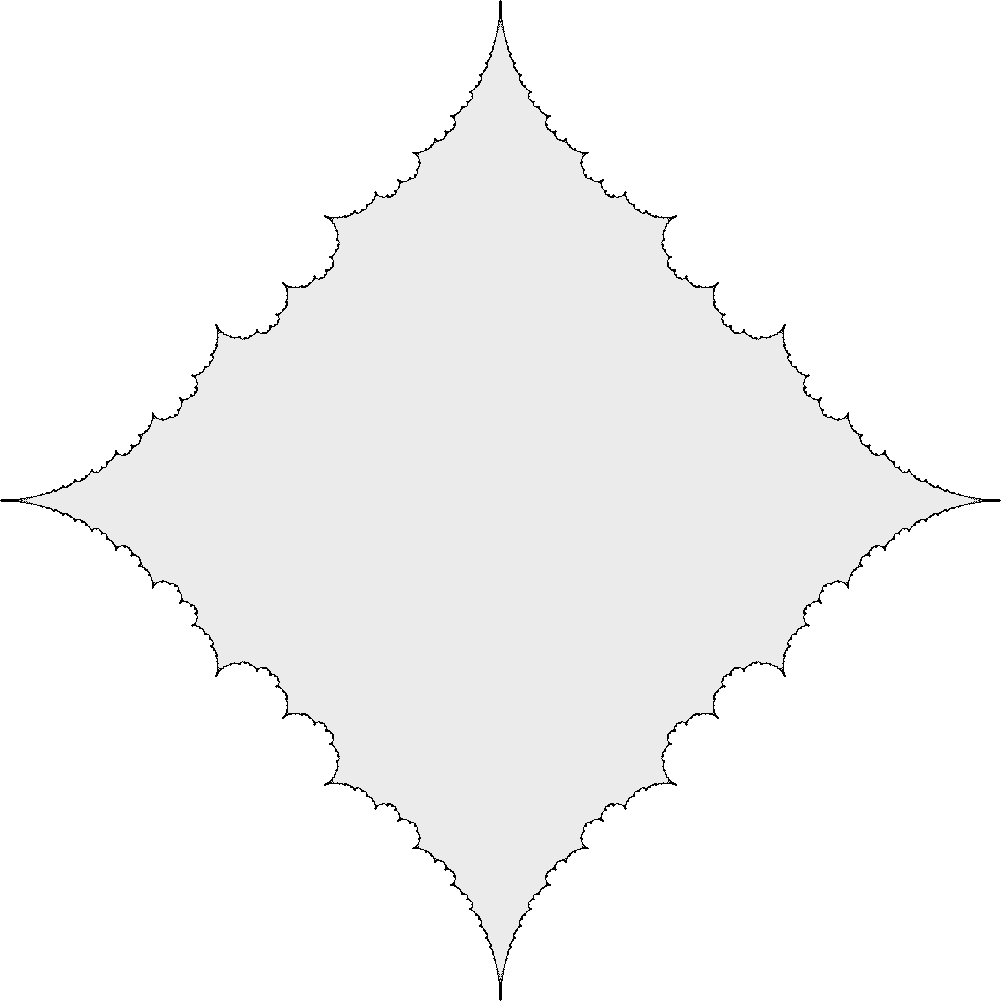}
  \includegraphics[width=65mm]{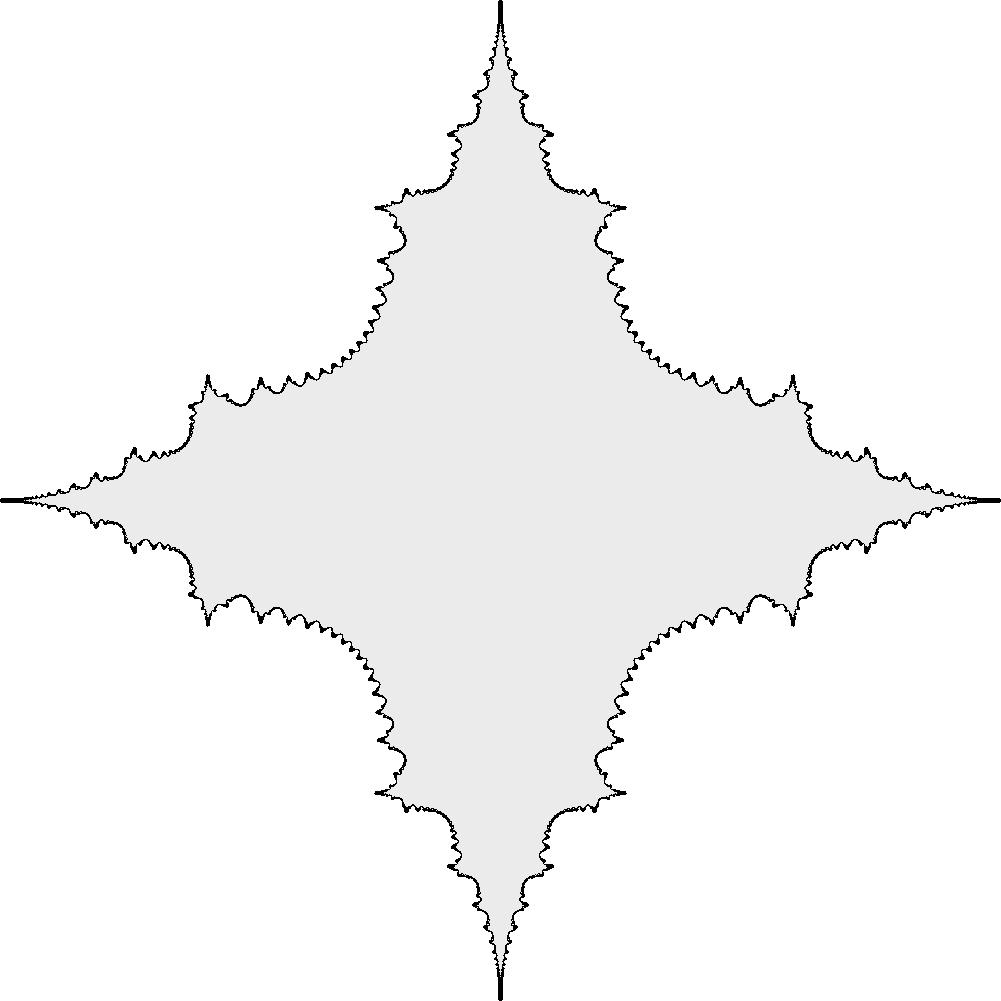}
  \caption{The Julia sets of $g_4$ and $g_{4,4}$ (up two), $h_4$ and $h_{4,4}$ (below two). The immediate parabolic basins of $1$ of $g_4$ and $g_{4,4}$ contain the unit disk and the immediate parabolic basins of $1$ of $h_4$ and $h_{4,4}$ contain the outside of the closed unit disk. Figure ranges:  up two are $[-2,2]^2$ and below two are $[-1,1]^2$.}
  \label{Fig_parabolic-polynomial}
\end{figure}

The following corollary shows that the rational map $P_n$ defined  in \eqref{family-para} has a very similar parabolic basin as that of $h_{d_1}$ if the parameter $s$ in $P_n$ is sufficiently small.

\begin{cor}\label{key-lemma-complex}
For any $r>1$, the open disk $D_r'=\EC\setminus \overline{\D}(1-r,r)$ lies in the immediate parabolic basin of $1$ of $P_n$ and $P_n(\overline{D_r'})\subset D_r'\cup\{1\}$ if the parameter $s>0$ is small enough.
\end{cor}

\begin{proof}
By Lemma \ref{P_n_limit}, $P_n$ converges to the parabolic rational map $h_{d_1}$ locally uniformly in $\EC\setminus \{0\}$ as the parameter $s>0$ tends to zero. Let $\varphi(z)=\lambda(1/z-1)$, where $\lambda=(d_1-1)/2$. Then $\varphi\circ P_n\circ\varphi^{-1}$ converges to
\begin{equation}
\widehat{h}_{d_1}(z):=\varphi\circ h_{d_1}\circ \varphi^{-1}(z)=\sum_{k=1}^{d_1}\frac{\lambda C_{d_1}^k}{d_1}\left(\frac{z}{\lambda}\right)^k=z+z^2+\frac{2(d_1-2)}{3(d_1-1)^2}z^3+\cdots
\end{equation}
locally uniformly on $\C$. The disk $\D(-\lambda,\lambda)$ is contained in the immediate parabolic basin of $0$ of $\widehat{h}_{d_1}$ since $h_{d_1}(\EC\setminus\overline{\D})\subset\EC\setminus\overline{\D}$. Since $\D\subset\D(1-r,r)$ for $r>1$, we have $\varphi(D_r')\subset\D(-\lambda,\lambda)$. Note that $\varphi(D_r')$ is a round disk whose boundary contains $0$. Then the closed disk $\overline{\varphi(D_r')}$ satisfies the condition in Lemma \ref{lema-approx} and $\overline{\varphi(D_r')}\setminus\{0\}$ is contained in the immediate parabolic basin of $0$ of $\varphi\circ P_n\circ \varphi^{-1}$ if the parameter $s$ in $P_n$ is small enough. This means that the open disk $D_r'$ lies in the immediate parabolic basin of $1$ of $P_n$ if $s$ is small enough.

By Lemma \ref{lema-h-n-h-mn}, we have $h_{d_1}(\overline{D_r'})\subset D_r'\cup\{1\}$. Let $\xi(z)=-1/z$. Then $\xi(\varphi(D_r'))=\mathbb{H}_{s}$ for some $s\in\R$ and $\xi\circ\widehat{h}_{d_1}\circ\xi^{-1}(\overline{\mathbb{H}}_s)\subset \mathbb{H}_s$. This means that $\xi\circ\varphi\circ P_n\circ\varphi^{-1}\circ\xi^{-1}(\overline{\mathbb{H}}_s)\subset \mathbb{H}_s$ since $P_n$ converges to $h_{d_1}$ locally uniformly on $\EC\setminus\{0\}$ as $s$ tends to zero by Lemma \ref{P_n_limit}. Therefore, we have $P_n(\overline{D_r'})\subset D_r'\cup\{1\}$. The proof is complete.
\end{proof}

In order to understand the dynamical behaviors of $P_n$, we now locate the critical points of $P_n$ (Lemma \ref{crit-close-Parameter}) and study the corresponding critical orbits (Lemma \ref{lemma-want}). It is easy to see $0$ and $\infty$ are critical points of $P_n$ with multiplicity $d_n-1$ and $d_1-1$ respectively, and the degree of $P_n$ is $\sum_{i=1}^{n}d_i$. Recall that $D_i:=d_i+d_{i+1}$ for $1\leq i\leq n-1$. Besides $0$ and $\infty$, by \eqref{solu-crit-parabolic}, the remaining $\sum_{i=1}^{n-1}D_i$ critical points of $P_n$ are the solutions of following equation:
\begin{equation}\label{equ-solu-F_n}
\sum_{i=1}^{n-1}\frac{(-1)^{i-1}D_i z^{D_i}}{z^{D_i}-a_i^{D_i}}+(-1)^{n-1} d_n-\frac{(d_1-1)d_1z^{d_1}}{(d_1-1)z^{d_1}+1}=0.
\end{equation}
For $1\leq i\leq n-1$, define $r_i:=\sqrt[D_i]{d_{i+1}/d_{i}}$ and let
\begin{equation}\label{defi-w-i-j}
\widetilde{\Crit}_i:=\{\widetilde{w}_{i,j}=r_i a_i e^{\pi \textup{i}(2j-1)/D_i}:1\leq j\leq D_i\}
\end{equation}
be the collection of $D_i$ points lying on the circle $\mathbb{T}_{r_i|a_i|}=\{z:|z|=r_i|a_i|\}$ uniformly.
The following Lemma \ref{crit-close-Parameter} shows that the positions of the remaining $\sum_{i=1}^{n-1}D_i$  critical points of $P_n$ are very `close' to $\bigcup_{i=1}^{n-1} \widetilde{\Crit}_i$.

\begin{lema}\label{crit-close-Parameter}
For any $\widetilde{w}_{i,j}\in\widetilde{\Crit}_i$, there exists $w_{i,j}$, which is a solution of \eqref{equ-solu-F_n}, such that $|w_{i,j}-\widetilde{w}_{i,j}|<s^{1/2}|a_i|$ if $s$ is small enough, where $1\leq i\leq n-1$ and $1\leq j\leq D_i$.
\end{lema}

\begin{proof}
By a direct calculation, the equation \eqref{equ-solu-F_n} is equivalent to
\begin{equation}\label{solu-crit-111}
(-1)^{i-1}\left(\frac{D_i z^{D_i}}{z^{D_i}-a_i^{D_i}}-d_{i+1}\right)+G_i(z)=0,
\end{equation}
where
\begin{equation}\label{G_n}
G_{i}(z)=\sum_{1\leq j\leq n-1,\,j\neq i}\frac{(-1)^{j-1}D_j z^{D_j}}{z^{D_j}-a_j^{D_j}}+(-1)^{i-1}d_{i+1}+(-1)^{n-1} d_n -\frac{(d_1-1)d_1z^{d_1}}{(d_1-1)z^{d_1}+1}.
\end{equation}
Multiplying both sides of \eqref{solu-crit-111} by $(z^{D_i}-a_i^{D_i})/d_i$, where $1\leq i\leq n-1$, we have
\begin{equation}\label{solu-crit-3}
(-1)^{i-1}(z^{D_i}+a_i^{D_i}d_{i+1}/d_i)+(z^{D_i}-a_i^{D_i})\,G_{i}(z)/d_i=0.
\end{equation}

For each $1\leq i\leq n-1$, define $\Omega_i:=\{z:|z^{D_i}+a_i^{D_i}d_{i+1}/d_{i}|\leq s^{1/2}|a_i|^{D_i}\}$.
By Lemma \ref{very-useful-est}(1), if $z\in\Omega_i$, we have
\begin{equation}\label{estim-0}
|z|^{d_1}\asymp |a_i|^{d_1}\asymp s^{\sum_{j=1}^i(d_1/d_j)}\preceq s.
\end{equation}
For $1\leq j<i$ and $z\in\Omega_i$, by Lemma \ref{very-useful-est}(2), we have
\begin{equation}\label{estim-1}
|z/a_j|^{D_j}\asymp |a_i/a_j|^{D_j}\asymp s^{\alpha_2}\leq s^{1+2/\dm},
\end{equation}
where $\alpha_2$ is defined in \eqref{alpha_2}. Similarly, if $i<j\leq n-1$ and $z\in\Omega_i$,  by Lemma \ref{very-useful-est}(2) and \eqref{alpha_1}, we have
\begin{equation}\label{estim-2}
|a_j/z|^{D_j}\asymp |a_j/a_i|^{D_j}\asymp s^{\alpha_1}\leq s^{1+2/\dm}.
\end{equation}

Since $D_i=d_i+d_{i+1}$ for $1\leq i\leq n-1$, we have
\begin{equation}\label{estim-3}
\sum_{i< j\leq n-1}(-1)^{j-1}D_j+(-1)^{i-1}d_{i+1}+(-1)^{n-1} d_n=0.
\end{equation}
From \eqref{G_n}, \eqref{estim-0}, \eqref{estim-1}, (\ref{estim-2}) and \eqref{estim-3}, if $z\in\Omega_i$, we have
\begin{equation*}\label{bound}
|G_{i}(z)|=   \left|\sum_{1\leq j<i}\frac{(-1)^{j}D_j (z/a_j)^{D_j}}{1-(z/a_j)^{D_j}}+
                   \sum_{i< j\leq n-1}\frac{(-1)^{j-1}D_j(a_j/z)^{D_j}}{1-(a_j/z)^{D_j}}
                  -\frac{(d_1-1)d_1z^{d_1}}{(d_1-1)z^{d_1}+1}\right| \preceq s.
\end{equation*}
This means that if $z\in\Omega_i$, we have
\begin{equation}\label{estim-4}
|z^{D_i}-a_i^{D_i}|\cdot|\,G_{i}(z)|/d_i \preceq|a_i|^{D_i}s.
\end{equation}

By \eqref{solu-crit-3}, \eqref{estim-4} and the definition of $\Omega_i$, there exist $D_i$ solutions $w_{i,j}$ of \eqref{equ-solu-F_n} such that $w_{i,j}\in\Omega_i$ by Rouch\'{e}'s Theorem, where $1\leq j\leq D_i$. By Lemma \ref{very-useful-est}(3) and \eqref{defi-w-i-j}, we have $|w_{i,j}-\widetilde{w}_{i,j}|<s^{1/2}|a_i|$ if $s$ is small enough. The proof is complete.
\end{proof}

For $1\leq i\leq n-1$, we use $\Crit_i:=\{w_{i,j}: 1\leq j\leq D_i\}$ to denote the collection of $D_i$ critical points of $P_n$ which is near the circle $\mathbb{T}_{r_i|a_i|}$.
A connected set $X\subset \overline{\mathbb{C}}$ is said \textit{separates} $0$ from $\infty$ if $0$ and $\infty$ lie in the two different components of $\overline{\mathbb{C}}\setminus X$ respectively. Let $X$ and $Y$ be two disjoint sets that both separate $0$ from $\infty$ respectively. We denote $X\prec Y$ if $X$ is contained in the component of $\overline{\mathbb{C}}\setminus Y$ which contains $0$.

\begin{lema}\label{lemma-want}
If $s>0$ is small enough, there exist two simply connected domains $U_0$ and $U_\infty$ containing $0$ and $\infty$, respectively, and $n-1$ annuli $A_1,\cdots, A_{n-1}$ satisfying $A_j\prec A_i$ for $1\leq i<j\leq n-1$, such that

$(1)$ $\mathbb{T}_{|a_i|}\cup \mathbb{T}_{r_i|a_i|} \cup \Crit_i\subset\subset A_i$ for $1\leq i\leq n-1$;

$(2)$ $P_n(\overline{A}_i)\subset U_0$ for odd $i$ and $P_n(\overline{A}_i)\subset U_\infty$ for even $i$;

$(3)$ $P_n(\overline{U}_0)\subset U_0$ for odd $n$ and $P_n(\overline{U}_0)\subset U_\infty$ for even $n$; and

$(4)$ $P_n(\overline{U}_\infty)\subset U_\infty\cup\{1\}$.
\end{lema}

\begin{proof}
(1) For every $1\leq i\leq n-1$, define the annulus
\begin{equation*}
A_i:=\{z:(\min\{r_i,1\}-2s^{1/2})|a_i|<|z|<(\max\{r_i,1\}+2s^{1/2})|a_i|\},
\end{equation*}
where $r_i=\sqrt[D_i]{d_{i+1}/d_{i}}$ and $s>0$ is small enough. Obviously, $\mathbb{T}_{|a_i|}\cup \mathbb{T}_{r_i|a_i|} \cup \Crit_i\subset\subset A_i$. Moreover, $\overline{A}_i\cap \overline{A}_j=\emptyset$ and $A_j\prec A_i$ if $1\leq i<j\leq n-1$ since $0<|a_j|\ll|a_i|$.

(2) By the definition of $r_i$, we have
\begin{equation*}
(2/\dm)^\frac{1}{D_i}\leq \min\{r_i,1\}\leq \max\{r_i,1\}\leq (\dm/2)^\frac{1}{D_i}.
\end{equation*}
This means that if $z\in \overline{A}_i$ and $s$ is small enough, then
\begin{equation}\label{esti-other}
|z/a_i|^{d_{i}}+|a_i/z|^{d_{i+1}}<\dm/2+1<\dm
\end{equation}
since $\dm\geq 3$ and at least one of $|z/a_i|$ and $|z/a_i|$ is less or equal to $1$.
By the definition of $P_n$, define
\begin{equation}\label{Phi-n}
\Phi_n(z): = \frac{P_n(z)-B_n}{A_n}\cdot\frac{(d_1-1)z^{d_1}+1}{d_1}.
\end{equation}
If $z\in \overline{A}_i$, by \eqref{esti-other}, we have
\begin{equation}\label{abs-f-n-yang}
\begin{split}
&|\Phi_n(z)| =
   |z|^{(-1)^{n-1} d_n}\,|z^{D_i}-a_i^{D_i}|^{(-1)^{i-1}}
        \,\prod_{j=1}^{i-1}|a_j|^{(-1)^{j-1}D_j}
           \prod_{j=i+1}^{n-1}|z|^{(-1)^{j-1}D_j}\cdot H_i(z)\\
= &
\left\{                         
\begin{array}{ll}               
\dm^2s\, |(z/a_i)^{d_i}-(a_i/z)^{d_{i+1}}| \ H_i(z) < \dm^2(\dm/2+1) s H_i(z)        &~\text{if}~i~\text{is odd}, \\
\dm^2\, |(z/a_i)^{d_i}-(a_i/z)^{d_{i+1}}|^{-1} \ H_i(z) >\dm H_i(z)  &~\text{if}~i~\text{is even},  
\end{array}                     
\right.                         
\end{split}
\end{equation}
where
\begin{equation*}\label{Q-i-yang}
H_i(z)=\prod_{j=1}^{i-1}\left|1-({z}/{a_j})^{D_j}\right|^{(-1)^{j-1}}
       \prod_{j=i+1}^{n-1}\left|1-({a_j}/{z})^{D_j}\right|^{(-1)^{j-1}}.
\end{equation*}
If $z\in \overline{A}_i$ for $1\leq i\leq n-1$, by Lemma \ref{very-useful-est}(2), we have \eqref{estim-1} and \eqref{estim-2} since $|z|\asymp|a_i|$. This means that
\begin{equation}\label{Q-i-esti-1-yang}
|H_i(z)-1|\preceq s^{1+2/\dm}.
\end{equation}

By Lemma \ref{para-fixed}(2), we have $|A_n-1| \preceq  s^{1+2/\dm}$, $|B_n| \preceq  s^{1+2/\dm}$ and $|z|^{d_1}\asymp|a_i|^{d_1}\preceq s$ by \eqref{estim-0}. For odd $i$ and sufficiently small $s>0$, if $z\in \overline{A}_i$, by \eqref{Phi-n}--\eqref{Q-i-esti-1-yang}, we have
\begin{equation}\label{P_n-est-0}
 |P_n(z)|\leq  \frac{d_1|A_n|\cdot |\Phi_n(z)|}{1-(d_1-1)|z|^{d_1}}+|B_n| <\dm^4 s.
\end{equation}
Similarly, for even $i$ and sufficiently small $s>0$, if $z\in \overline{A}_i$, we have
\begin{equation}\label{P_n-est-1}
 |P_n(z)|\geq  \frac{d_1|A_n|\cdot |\Phi_n(z)|}{1+(d_1-1)|z|^{d_1}}-|B_n| >\dm\geq 3.
\end{equation}

Let $U_0:=\D(0,\dm^4 s)$ be the disk centered at the origin with radius $\dm^4 s$ and $U_\infty:=D_2'=\EC\setminus \overline{{\D(-1,2)}}$.
Then we have $P_n(\overline{A}_i)\subset U_0$ for odd $i$ and $P_n(\overline{A}_i)\subset U_\infty$ for even $i$.

(3) Let $n\geq 3$ be an odd number and assume that $|z|\leq \dm^4 s$. By Lemma \ref{very-useful-est}(1) and \eqref{Phi-n}, we have
\begin{equation}\label{bound-lower-in-disk-even}
|\Phi_n(z)| =\frac{\dm^2|z|^{d_n}}{|a_{n-1}|^{d_n}} \prod_{i=1}^{n-1}\left|1-\frac{z^{D_i}}{a_i^{D_i}}\right|^{(-1)^{i-1}}
      \asymp \frac{|z|^{d_n}}{|a_{n-1}|^{d_n}}\preceq s^\ell,
\end{equation}
where $\ell=d_n(1-\sum_{i=1}^{n-1}(1/d_i))>1$.
Let $\widetilde{\ell}=\min\{\ell,1+2/\dm\}>1$. By Lemma \ref{para-fixed}(2) and \eqref{Phi-n}, we have
\begin{equation*}
|P_n(z)|\leq \frac{d_1|A_n|\cdot |\Phi_n(z)|}{1-(d_1-1)|z|^{d_1}}+|B_n| \preceq s^{\widetilde{\ell}} \text{~and hence~} |P_n(z)|<\dm^4 s
\end{equation*}
if $s$ is small enough. This means that $P_n(\overline{U}_0)\subset U_0$ for odd $n$.

If $n$ is even, $s$ is small enough and $|z|\leq \dm^4 s$, we have
\begin{equation*}\label{bound-lower-in-disk-odd}
|\Phi_n(z)|=\frac{|a_{n-1}|^{d_n}\dm^2 s}{|z|^{d_n}}\,\prod_{i=1}^{n-1}\left|1-\frac{z^{D_i}}{a_i^{D_i}}\right|^{(-1)^{i-1}} \asymp \frac{|a_{n-1}|^{d_n} s}{|z|^{d_n}}
 \succeq  s^{1-\ell}.
\end{equation*}
This means that if $|z|\leq \dm^4 s$ and $s>0$ is small enough, then $|\Phi_n(z)|>\dm$. Hence, we have
\begin{equation*}
|P_n(z)|\geq   \frac{d_1|A_n|\cdot |\Phi_n(z)|}{1+(d_1-1)|z|^{d_1}}-|B_n| >\dm.
\end{equation*}
Therefore, $P_n(\overline{U}_0)\subset \EC\setminus\overline{\D(0,\dm)}\subset U_\infty$ for even $n$.
This ends the proof of (3).

(4) Since $U_\infty=D_2'=\EC\setminus \overline{{\D(-1,2)}}$, it follows by Corollary \ref{key-lemma-complex} that $P_n(\overline{U}_\infty)\subset U_\infty\cup\{1\}$. The proof is complete.
\end{proof}

\begin{thm}\label{parameter-parabolic-resta}
Let $|a_1|=(\dm^2 s)^{1/d_1}$ and $|a_i|=s^{1/d_i}|a_{i-1}|$ be the numbers defined in \eqref{range-s} for $2\leq i\leq n-1$. If the parameter $s>0$ is small enough, then the Julia set of $P_n$ is a Cantor set of circles with a parabolic fixed point $1$.
\end{thm}

\begin{proof}
Let $U_0=\D(0,\dm^4 s)$ and $U_\infty=\EC\setminus\overline{\D(-1,2)}$ be the simply connected domains defined in Lemma \ref{lemma-want}. By Lemma \ref{lemma-want} (1) and (2), there exists a component $U_i$ of $P_n^{-1}(U_0\cup U_\infty)$ which contains the the annulus $A_i$, the round circle $\T_{|a_i|}$ and the critical points $\Crit_i$, where $1\leq i\leq n-1$. Moreover, $U_i\cap U_{i+1}=\emptyset$ since $P_n(U_i)\cap P_n(U_{i+1})\subset U_0\cap U_\infty=\emptyset$ by Lemma \ref{lemma-want} (2). This means that $U_i\cap U_j=\emptyset$ for different $i,j$ and in particular, $U_{n-1}\prec \cdots\prec U_1$ since $A_j\prec A_i$ for $1\leq i<j\leq n-1$. By Lemma \ref{lemma-want} (3) and (4), there exist two components $\widetilde{U}_0$ and $\widetilde{U}_\infty$ of $P_n^{-1}(U_0\cup U_\infty)$ such that they contain $U_0$ and $U_\infty$ respectively. Then $\widetilde{U}_0\cap U_{n-1}=\emptyset$ and $\widetilde{U}_\infty\cap U_1=\emptyset$ since
\begin{equation*}
P_n(\widetilde{U}_0)\cap P_n(U_{n-1})=P_n(\widetilde{U}_\infty)\cap P_n(U_1)=U_0\cap U_\infty=\emptyset.
\end{equation*}
Therefore, $\widetilde{U}_0$ lies in the component of $\EC\setminus \overline{U}_{n-1}$ which contains $0$ and $\widetilde{U}_\infty$ lies in the component of $\EC\setminus \overline{U}_1$ which contains $\infty$ (see Figure \ref{fig_P_n}).

\begin{figure}[!htpb]
  \setlength{\unitlength}{1mm}
  \centering
  \includegraphics[width=140mm]{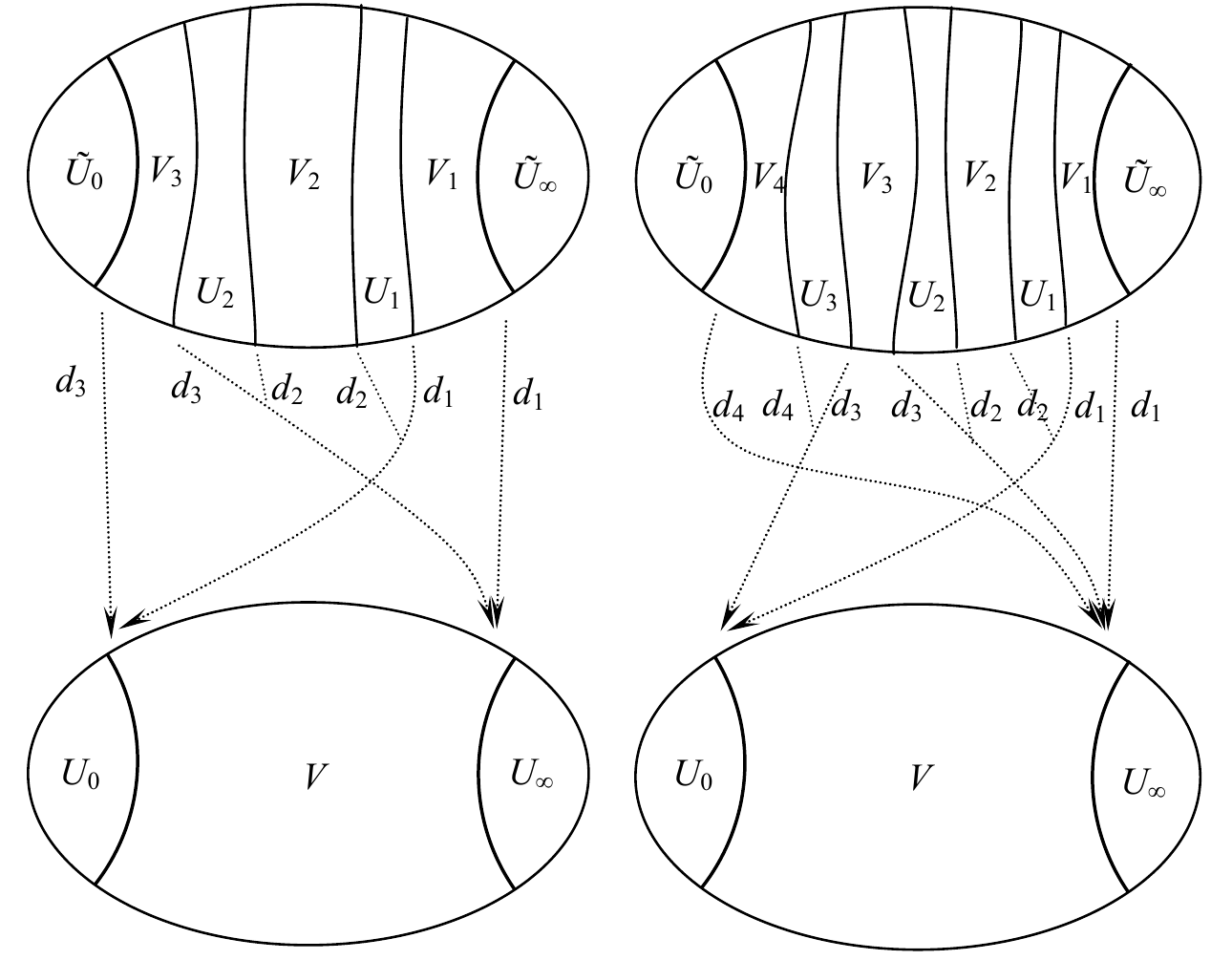}
  \caption{Sketch illustrating of the mapping relations of $P_3=P_{d_1,d_2,d_3}$ and $P_4=P_{d_1,d_2,d_3,d_4}$.}
  \label{fig_P_n}
\end{figure}

For $1\leq i\leq n-1$, we claim that $P_n:U_i\rightarrow U_0$ with odd $i$ and $P_n:U_i\rightarrow U_\infty$ with even $i$, respectively, are branched coverings with degree $D_i$. Recall that $B_n$ is a number defined in \eqref{family-para} satisfying $|B_n|\preceq s^{1+2/\dm}$ by Lemma \ref{para-fixed}. Then we have $B_n\in U_0$ if $s$ is small enough. Let $\textup{ZP}_i=\{e^{2\pi \textup{i} j/D_i}a_i:1\leq j\leq D_i\}\subset A_i$ be the collection of $D_i$ points lying on the circle $\T_{|a_i|}$ uniformly, where $1\leq i\leq n-1$. If $i$ is odd, then $P_n:U_i\rightarrow U_0$ is a branched covering with degree \textit{at least} $D_i$ since
\begin{equation*}
\textup{ZP}_i\subset P_n^{-1}(B_n) \cap U_i \subset P_n^{-1}(U_0)\cap U_i=U_i.
\end{equation*}
Similarly, $P_n:U_i\rightarrow U_\infty$ is a branched covering with degree \textit{at least} $D_i$ for even $i$ since $P_n^{-1}(\infty)\cap U_i$ also contains at least $D_i$ different points $\textup{ZP}_i$.

Since $P_n(\infty)=\infty$ and the local degree at $\infty$ is $d_1$, it follows that $P_n:\widetilde{U}_\infty\to U_\infty$ is a branched covering with degree at least $d_1$. If $n=2m+1\geq 3$ is odd, then $P_n^{-1}(U_\infty)$ contains $W_1:=\widetilde{U}_\infty\cup\bigcup_{k=1}^m U_{2k}$. The degree of the restriction of $P_n$ on $W_1$ is at least $d_1+\sum_{k=1}^m D_{2k}=\sum_{i=1}^n d_i$, which is equal to the degree of $P_n$. This means that $P_n^{-1}(U_\infty)=W_1$, $\deg(P_n:\widetilde{U}_\infty\to U_\infty)=d_1$ and $\deg(P_n:U_{2k}\to U_\infty)=D_{2k}$ for $1\leq k\leq m$. On the other hand, $P_n^{-1}(U_0)$ contains $W_2:=\widetilde{U}_0\cup\bigcup_{k=1}^m U_{2k-1}$. Since $P_n(0)=B_n\in U_0$ and the local degree at $0$ is $d_n$, it follows that $P_n:\widetilde{U}_0\to U_0$ is a branched covering with degree at least $d_n$. Then the degree of the restriction of $P_n$ on $W_2$ is at least $d_n+\sum_{k=1}^m D_{2k-1}=\sum_{i=1}^n d_i=\deg(P_n)$. This means that $P_n^{-1}(U_0)=W_2$, $\deg(P_n:\widetilde{U}_0\to U_0)=d_n$ and $\deg(P_n:U_{2k-1}\to U_\infty)=D_{2k-1}$ for $1\leq k\leq m$. This ends the proof of the claim in the case of odd $n$. If $n$ is even, the argument is completely similar and we omit the details here (see right picture in Figure \ref{fig_P_n}).

Note that for each $1\leq i\leq n-1$, the open set $U_i$ contains at least $D_i$ critical points, and $\widetilde{U}_0$, $\widetilde{U}_\infty$ contain at least $d_n-1$ and $d_1-1$ critical points respectively. This means that $U_i$, $\widetilde{U}_0$ and $\widetilde{U}_\infty$, respectively, contain exactly $D_i=d_i+d_{i+1}$, $d_n-1$ and $d_1-1$ critical points since $P_n$ contains exactly $2\sum_{i=1}^n{d_i}-2$ critical points. By Riemann-Hurwitz formula, $U_i$ has Euler characteristic zero and hence must be an annulus. The opens sets $\widetilde{U}_0$ and $\widetilde{U}_\infty$ both have Euler characteristic one and hence must be two topological disks. Moreover, $U_i$ is an annulus separating $0$ from $\infty$ for every $1\leq i\leq n-1$ since $U_i$ is disjoint with $\widetilde{U}_0\cup \widetilde{U}_\infty$ and $U_i$ contains the annulus $A_i$ which is defined in Lemma \ref{lemma-want}.

Define a closed annulus $V:=\EC\setminus (U_0\cup U_\infty)$. Then $P_n^{-1}(V)$ consists of $n$ annuli $V_1,\cdots,V_n$ which satisfies $V_n\prec V_{n-1}\prec \cdots\prec V_1$ since $V$ is disjoint with the critical orbits of $P_n$. For $1\leq j\leq n$, each $P_n:V_j\to V$ is a covering map with degree $d_j$. The Julia set of $P_n$ is $J(P_n)=\bigcap_{k\geq 0}P_n^{-k}(V)$. Let $g_i:V\rightarrow V_j$ be the inverse branch of $P_n:V_j\to V$ for $1\leq j\leq n$. Then each component of $J(P_n)$ can be written as $J_{j_1 j_2\cdots j_k\cdots}=\bigcap_{k=1}^\infty g_{j_k}\cdots g_{j_2}\circ g_{j_1}(V)$, where $(j_1,j_2,\cdots,j_k,\cdots)$ is an infinite sequence satisfying $1\leq j_k\leq n$. By the construction, each component $J_{j_1 j_2\cdots j_k\cdots}$ is a compact set separating $0$ and $\infty$. Since $P_n$ is geometrically finite, it follows that the Julia component $J_{j_1 j_2\cdots j_k\cdots}$ is locally connected (see \cite{TY} and \cite{PT}).

Let $E=J'\cup J''$, where $ J'=J_{n,n,\cdots,n,\cdots}=\partial D_0$ and $J''=J_{1,1,\cdots,1,\cdots}=\partial D_\infty$ and $D_0$, $D_\infty$ are the Fatou components containing $U_0$ and $U_\infty$ respectively. Let $A=\EC\setminus(\overline{D}_0\cup \overline{D}_\infty)$ be the annulus determined uniquely by $J'$ and $J''$. Then each component $A'$ of $P_n^{-1}(A)$ satisfies $A'\subset A$ and the identity $\textup{id}:A'\hookrightarrow A$ is not homotopic to a constant map. We now divide the arguments into two cases. The first case: if the forward orbit of the Julia component $J_{j_1 j_2\cdots j_k\cdots}$ is contained in $A$, then $J_{j_1 j_2\cdots j_k\cdots}$ is a Jordan curve by \cite[Lemma 2.4, Case 2]{PT} and \cite[Proposition Case 2]{PT} (Note that \cite[Lemma 2.4]{PT} holds for geometrically finite rational maps (see \cite[$\S$9]{PT}). The second case: $J_{j_1 j_2\cdots j_k\cdots}$ is eventually iterated onto $J'$ or $J''$ by $P_n$. Note that $P_n:(V_1\cup \widetilde{U}_\infty)^o\to (V\cup U_\infty)^o$ is a polynomial-like map with degree $d_1$, where the superscript `$^o$' means the interior of a set. By Douady and Hubbard's Straighten Theorem \cite[p.\,296]{DH}, $P_n:(V_1\cup \widetilde{U}_\infty)^o\to (V\cup U_\infty)^o$ is quasiconformally equivalent to a parabolic polynomial $\widetilde{P}_n$ with degree $d_1$. Note that $\infty$ is the unique critical point of $P_n$ in $(V_1\cup \widetilde{U}_\infty)^o$. This means that $\widetilde{P}_n$ is conformally conjugated to $g_{d_1}(z)=(z^{d_1}+d_1-1)/d_1$. It is well known that the Julia set of $g_{d_1}$ is a Jordan curve. Hence the Julia component $J''$ is a Jordan curve. Similarly, one can prove that the Julia component $J'$ is a Jordan curve since either $P_n(J')=J''$ ($n$ is even) or $P_n:(\widetilde{U}_0 \cup V_n )^o\to (U_0 \cup V )^o$ is a polynomial-like map with degree $d_n$ which is quasiconformally conjugated to $z\mapsto z^{d_n}$ ($n$ is odd). This means that $J_{j_1 j_2\cdots j_k\cdots}$ is a Jordan curve since it was eventually iterated onto $J'$ or $J''$ by $P_n$ and there are no critical points on $J(P_n)$.

Up to now, we have proved that all the Julia components of $P_n$ are Jordan curves. On the other hand, the dynamics on the Julia components $\{J_{j_1 j_2\cdots j_k\cdots}:1\leq j_k\leq n\}$  is isomorphic to the one-sided shift on the space of $n$ symbols $\Sigma_{n}:=\{1,\cdots,n\}^{\mathbb{N}}$. In particular, the Julia set $J(P_n)$ is homeomorphic to $\Sigma_{n}\times\mathbb{S}^1$, which is a Cantor set of circles, as desired. This ends the proof of Theorem \ref{parameter-parabolic-resta} and hence Theorem \ref{parameter-parabolic}.
\end{proof}

\begin{figure}[!htpb]
  \setlength{\unitlength}{1mm}
  \centering
  \includegraphics[width=70mm]{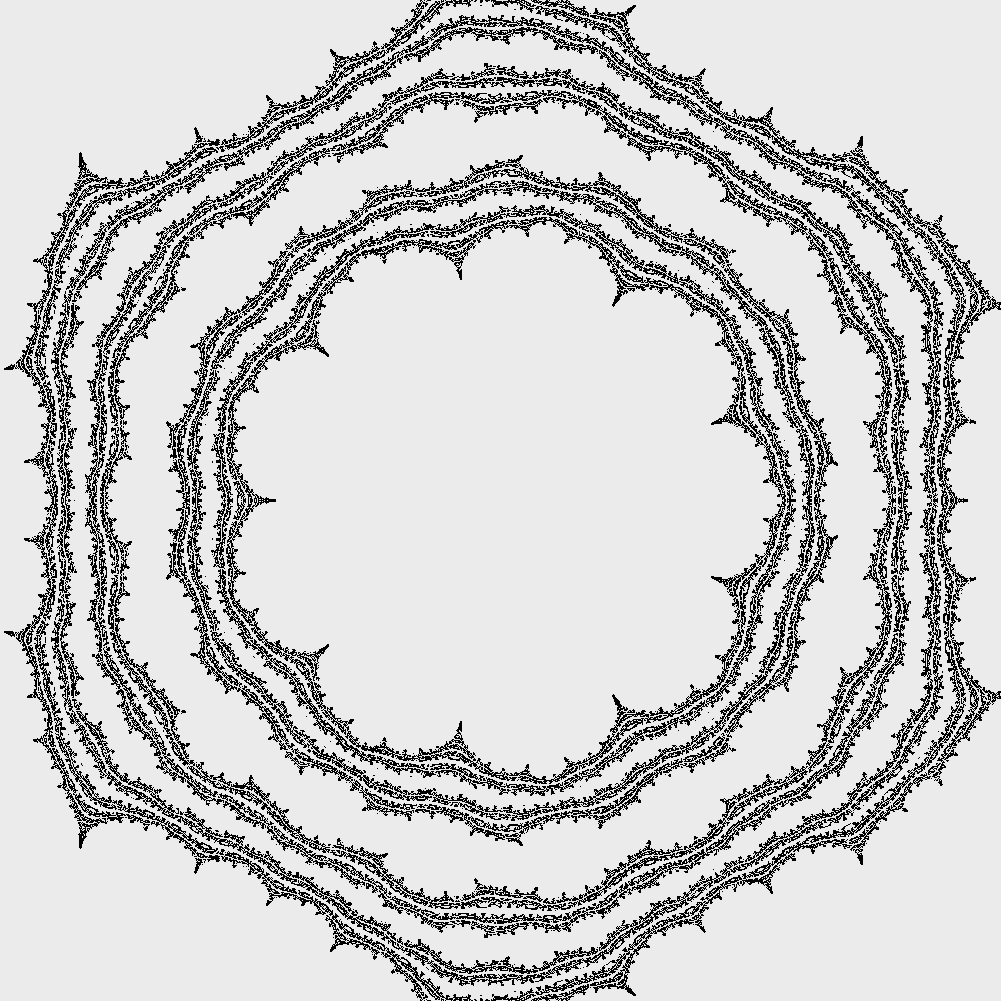}
  \includegraphics[width=70mm]{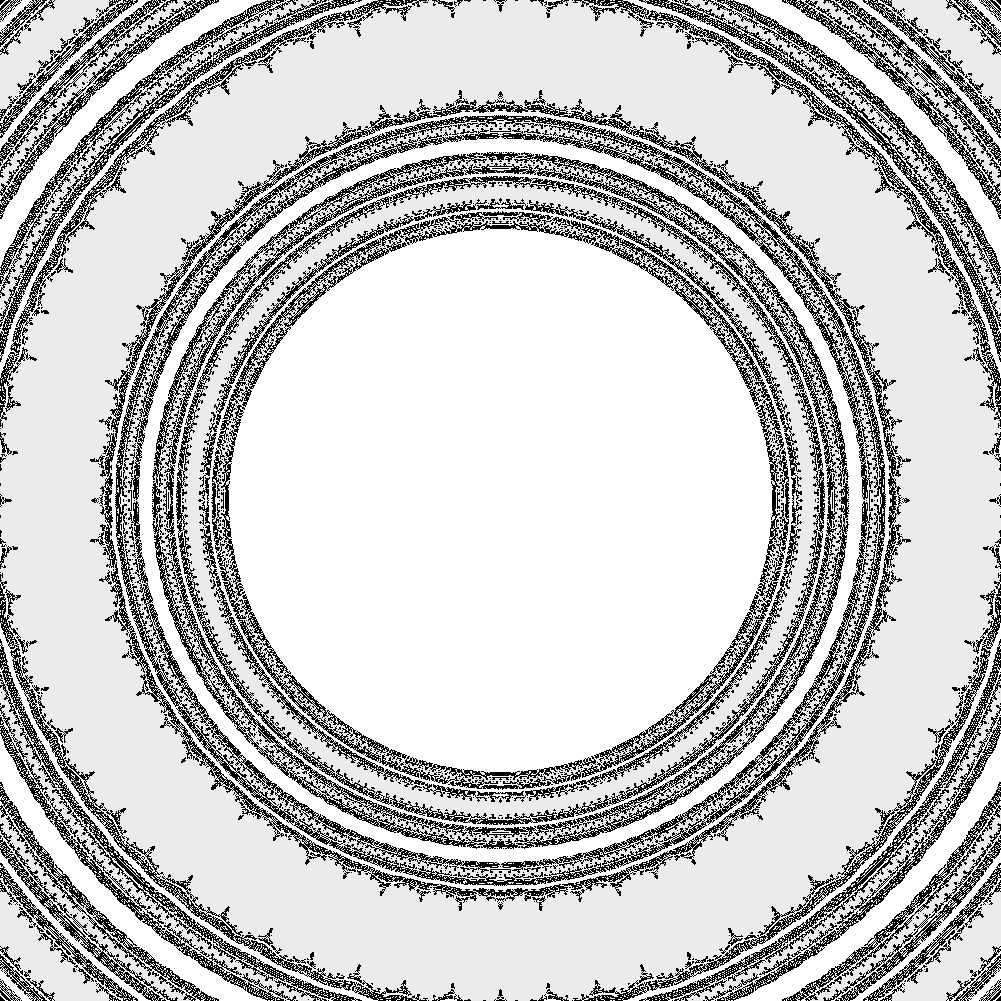}
  \caption{The zoom of the Julia sets of $P_{3,3}$ and $P_{4,4,4}$ near the Fatou components containing the origin with the same parameters as defined in the introduction. Figure ranges: $[-0.2,0.2]^2$ and $[-2.5\times 10^{-3},2.5\times 10^{-3}]^2$.}
  \label{Fig_Cantor-cicle_local}
\end{figure}

\section{Cantor circles with two parabolic fixed points}\label{sec-2-para-fixed}

In this section, we construct a family of non-hyperbolic rational maps with Cantor circles Julia sets such that each one of them has two parabolic fixed points lying on two fixed simply connected parabolic basins. The construction is more difficult than before since we have to control two fixed parabolic Fatou components.

Before we provide the specific formulas of the rational maps which satisfy the above conditions, we recommend the reader to take a look at the left picture in Figure \ref{fig_parabolic_III-IV} in which a rough indication of the dynamics of such rational maps has been made. Let $n\geq 3$ be an odd number and $d_1,\cdots,d_n$ be $n$ positive integers such that $\sum_{i=1}^{n}(1/d_i)<1$. We define
\begin{equation}\label{family-para-para-fixed}
Q_{d_1,\cdots,d_n}(z)=\frac{d_1 z^{d_n}}{(d_1-1)X_n z^{d_1}+Y_n z+Z_n}
\prod_{i=1}^{n-1}(z^{d_i+d_{i+1}}-b_i^{d_i+d_{i+1}})^{(-1)^{i-1}}+W_n,
\end{equation}
where $b_1,\cdots,b_{n-1}$ are $n-1$ small positive real numbers and $X_n,Y_n,Z_n,W_n$ are numbers depending only on $b_1,\cdots,b_{n-1}$ which will be determined later. Our main work in this section is to find suitable $X_n,Y_n,Z_n$ and $W_n$ such the Julia set of each $Q_{d_1,\cdots,d_n}$ is a Cantor set of circles containing two parabolic fixed points. The difficult we will face first is to locate the position of the two parabolic fixed points. As in \S\ref{sec-para-1-fixed}, we can arrange one of the parabolic fixed points is $1$. For another parabolic fixed point, we do it by trial-and-error method (see the remark in the last part of this section).

As in \S\ref{sec-para-1-fixed}, denote the maximal number among $d_1,\cdots,d_n$ by $\dm\geq 3$. Let
\begin{equation}\label{tau}
\tau:=(d_1 d_n \dm^{2(d_1-d_n)/d_1})^{1/\sum_{i=1}^{n-1}(d_n/d_i)}>0
\end{equation}
be a constant depending only on $d_1,\cdots,d_n$ and define $b_1,\cdots,b_{n-1}$ as
\begin{equation}\label{range-s-fixed}
b_1=(\dm^2 \tau s)^{1/d_1} \text{~ and ~} b_i=(\tau s)^{1/d_i}\,b_{i-1} \text{~for~} 2\leq i\leq n-1,
\end{equation}
where the parameter $s>0$ is small enough such that $1\gg b_1\gg b_2\gg\cdots\gg b_{n-1}>0$.

For simplicity, we use $Q_n$ to denote $Q_{d_1,\cdots,d_n}$ as defined in \eqref{family-para-para-fixed}. Moreover, define $D_i:=d_i+d_{i+1}$ as in \S\ref{sec-para-1-fixed}, where $1\leq i\leq n-1$. Note that $\sum_{i=1}^{n}(1/d_i)<1$, we have
\begin{equation}\label{my-nu}
\nu:=\frac{d_n}{d_n-1}\sum_{i=1}^{n-1}\frac{1}{d_i}<1.
\end{equation}

\begin{lema}\label{lema-useful-fixed-2}
For $1\leq i\leq n-1$,  we have $s^{D_i\nu}/b_i^{D_i}\asymp s^{\alpha_3}$, where $\alpha_3=\alpha_3(i,d_1,\cdots,d_n)\geq (1+2/\dm)\nu$.
\end{lema}

\begin{proof}
By \eqref{range-s-fixed}, we have $s^{D_i\nu}/b_i^{D_i}=(s^\nu/b_i)^{D_i}\asymp s^{\alpha_3}$, where
\begin{equation*}
\alpha_3 =(\frac{d_n}{d_n-1}\sum_{j=1}^{n-1}\frac{1}{d_j}-\sum_{j=1}^i\frac{1}{d_j})(d_i+d_{i+1})
= \frac{d_i+d_{i+1}}{d_n-1}\sum_{j=1}^{n-1}\frac{1}{d_j}+(d_i+d_{i+1})\sum_{j=i+1}^{n-1}\frac{1}{d_j}.
\end{equation*}
Note that $2\leq d_j\leq \dm$ for $1\leq j\leq n$. If $i=n-1$, we have
\begin{equation*}
\alpha_3 =\frac{d_{n-1}+d_n}{d_n-1}\sum_{j=1}^{n-1}\frac{1}{d_j}=(1+\frac{d_{n-1}}{d_n})\nu\geq (1+\frac{2}{\dm})\nu.
\end{equation*}
If $1\leq i<n-1$, we have
\begin{equation*}
\alpha_3 =\frac{d_i+d_{i+1}}{d_n}\nu+1+d_{i+1}\sum_{j=i+2}^{n-1} \frac{1}{d_j}+ d_i\sum_{j=i+1}^{n-1} \frac{1}{d_i}
\geq \frac{5}{\dm}\nu+1+\frac{2}{\dm}>(1+\frac{2}{\dm})\nu.
\end{equation*}
The proof is complete.
\end{proof}

\begin{lema}\label{lema-para-fixed-II}
There exist suitable $X_n\approx 1$, $Y_n\approx 0$, $Z_n\approx 1$ and $W_n\approx 0$ in \eqref{family-para-para-fixed} such that $Q_n$ has two parabolic fixed points $1$ and $s^\nu$ and both with multiplier $1$, where
\begin{equation}\label{X_n-Y_n}
\begin{split}
\lim_{s\to 0^+}\frac{X_n-1}{s^\nu}=\frac{d_1(d_1-3)(d_n-1)}{(d_1-1)^2 d_n},
&~~ \lim_{s\to 0^+}\frac{Y_n}{s^\nu}=\frac{2d_1(d_n-1)}{(d_1-1) d_n},\\
\lim_{s\to 0^+}\frac{Z_n-1}{s^\nu}=0\text{~and}
&~~ \lim_{s\to 0^+}\frac{W_n}{s^\nu}=\frac{d_n-1}{d_n}.
\end{split}
\end{equation}
\end{lema}

\begin{proof}
We find $X_n,Y_n,Z_n$ and $W_n$ by solving four equations. By $Q_n(1)=1$, we have
\begin{equation}\label{Q-n-equ-1}
\frac{d_1\rho_1}{(d_1-1)X_n+Y_n+Z_n}+W_n=1, \text{~where~}\rho_1=\prod_{i=1}^{n-1}(1-b_i^{D_i})^{(-1)^{i-1}}.
\end{equation}
By $Q_n(s^\nu)=s^\nu$, we have
\begin{equation}\label{Q-n-equ-2}
\frac{Q_n(s^\nu)}{s^\nu}=\frac{\rho_2}{(d_1-1)s^{d_1\nu}X_n+s^\nu Y_n+Z_n}+\frac{W_n}{s^\nu}=1,
\end{equation}
where
\begin{equation}\label{rho-2-defi}
\rho_2=d_1 s^{(d_n-1)\nu}\prod_{i=1}^{n-1}(s^{D_i\nu}-b_i^{D_i})^{(-1)^{i-1}}.
\end{equation}
By the definitions of $\tau$ and $\nu$ in \eqref{tau} and \eqref{my-nu}, it follows that $\tau^{(d_n-1)\nu}=d_1 d_n\dm^{\, 2(d_1-d_n)/d_1}$.
Since $b_1=(\dm^2 \tau s)^{1/d_1}$ and $b_i=(\tau s)^{1/d_i}\,b_{i-1}$ for $2\leq i\leq n-1$, we have
\begin{equation}\label{b_n-prod-simplify}
\prod_{i=1}^{n-1}b_i^{(-1)^{i-1}D_i}=\frac{\dm^2}{b_{n-1}^{d_n}}=\frac{\dm^{\, 2(d_1-d_n)/d_1}}{(\tau s)^{d_n((1/d_1)+\cdots+(1/d_{n-1}))}}=\frac{\dm^{\, 2(d_1-d_n)/d_1}}{(\tau s)^{(d_n-1)\nu}}=\frac{1}{d_1 d_n s^{(d_n-1)\nu}}.
\end{equation}
Note that $n\geq 3$ is odd, then by Lemma \ref{lema-useful-fixed-2}, \eqref{rho-2-defi} and \eqref{b_n-prod-simplify}, we have
\begin{equation}\label{rho_2}
\left|\rho_2-\frac{1}{d_n}\right|=\frac{1}{d_n}\left|\prod_{i=1}^{n-1}\left(1-\frac{s^{D_i\nu}}{b_i^{D_i}}\right)^{(-1)^{i-1}}-1\right|
\preceq s^{(1+2/\dm)\nu}.
\end{equation}

In order to calculate the derivation of $Q_n$, we consider
\begin{equation*}\label{solu-crit-parabolic-1}
\frac{zQ_n'(z)}{Q_n(z)-W_n}= \sum_{i=1}^{n-1}\frac{(-1)^{i-1}D_iz^{D_i}}{z^{D_i}-b_i^{D_i}}+d_n
-\frac{(d_1-1)d_1 X_n z^{d_1}+Y_n z}{(d_1-1)X_n z^{d_1}+Y_n z+Z_n}.
\end{equation*}
Then $Q_n'(1)=1$ and $Q_n'(s^\nu)=1$ are equivalent to the following two equations:
\begin{equation}\label{Q-n-equ-3}
\frac{1}{1-W_n}=\rho_3+\frac{(d_1-1)Y_n+d_1 Z_n}{(d_1-1)X_n+Y_n+Z_n},
\end{equation}
and
\begin{equation}\label{Q-n-equ-4}
\frac{1}{1-W_n/s^\nu}=\rho_4+d_n-\frac{(d_1-1) d_1 s^{d_1\nu}X_n+s^\nu Y_n}{(d_1-1)s^{d_1\nu}X_n+s^\nu Y_n+Z_n},
\end{equation}
where
\begin{equation}\label{rho-3-4-defi}
\rho_3=\sum_{i=1}^{n-1}\frac{(-1)^{i-1}D_i\,b_i^{D_i}}{1-b_i^{D_i}} \text{~~and~~} \rho_4=\sum_{i=1}^{n-1}\frac{(-1)^{i-1}D_is^{D_i\nu}}{s^{D_i\nu}-b_i^{D_i}}.
\end{equation}

For $a\in\R$ and $r>0$, let $\I(a,r):=\{x\in\R:|x-a|<r\}$ be the interval centered at $a$ with length $2r$. Lemma \ref{lema-para-fixed-II} holds if the following Lemma \ref{zero-solution} has been proved.
\end{proof}

\begin{lema}\label{zero-solution}
The set of four equations \eqref{Q-n-equ-1}, \eqref{Q-n-equ-2}, \eqref{Q-n-equ-3} and \eqref{Q-n-equ-4} has a solution $(X_n,Y_n,Z_n,W_n)$ in
\begin{equation*}
\Theta:=\I(1+x_n s^\nu,s^{\nu+\beta_1})\times \I(y_n s^\nu,s^{\nu+\beta_2})
\times \I(1,s^{\nu+\beta_3})\times \I(w_n s^\nu,s^{\nu+\beta_4}),
\end{equation*}
where $0<\beta_1<\beta_2<\beta_3<\beta_4\leq \nu/\dm$ and
\begin{equation}\label{xn-yn-wn}
x_n=\frac{d_1(d_1-3)(d_n-1)}{(d_1-1)^2 d_n}, ~ y_n=\frac{2d_1(d_n-1)}{(d_1-1) d_n} \text{~and~} w_n=\frac{d_n-1}{d_n}.
\end{equation}
\end{lema}

\begin{proof}
In order to simplify the notations, we use $x,y,z$ and $w$ to denote $X_n,Y_n,Z_n$ and $W_n$ respectively. According to \eqref{Q-n-equ-1}, \eqref{Q-n-equ-2}, \eqref{Q-n-equ-3} and \eqref{Q-n-equ-4}, we define
\begin{equation}\label{4-equations}
\left\{
\begin{split}
f_1(x,y,z,w)&:=\frac{d_1\rho_1}{(d_1-1)x+y+z}+w-1,\\
f_2(x,y,z,w)&:=\frac{\rho_2}{(d_1-1)s^{d_1\nu}x+s^\nu y+z}+\frac{w}{s^\nu}-1,\\
f_3(x,y,z,w)&:=\frac{1}{1-w}-\rho_3-\frac{(d_1-1)y+d_1 z}{(d_1-1)x+y+z},\text{~and~}\\
f_4(x,y,z,w)&:=\frac{1}{1-w/s^\nu}-\rho_4-d_n+\frac{(d_1-1) d_1 s^{d_1\nu}x+s^\nu y}{(d_1-1)s^{d_1\nu}x+s^\nu y+z}.
\end{split}
\right.
\end{equation}
We want to prove the collection of four equations $f_i(x,y,z,w)=0$, where $1\leq i\leq 4$, has a solution in a small neighborhood of $(1,0,1,0)$ if $s>0$ is small enough.

Let $\beta_1,\beta_2,\beta_3,\beta_4>0$ be four positive numbers satisfying $0<\beta_1<\beta_2<\beta_3<\beta_4\leq \nu/\dm$. We denote $F:=(f_1,f_2,f_3,f_4)^\textup{T}$.
For $\Lambda=(x,y,z,w)\in\R^4$, the \emph{Newton's method} of $F$ is defined as
\begin{equation}\label{Newton's-mtd}
N_F(\Lambda):=\Lambda^\textup{T}-\Jac(\Lambda)^{-1}F(\Lambda),
\end{equation}
where $\Jac(\Lambda)$ is the Jacobi matrix of \eqref{4-equations} at $\Lambda$. If the sequence $\{\Lambda_k=N_F^{\circ k}(\Lambda)\}_{k\in\mathbb{N}}$ is convergent, then the limit of this sequence is a zero solution of \eqref{4-equations}. However, in our case, it is not easy to prove that $\{\Lambda_k\}_{k\in\mathbb{N}}$ is convergent in $\Theta$. Instead, we will prove $N_F(\Theta)\subset \Theta$ if $s>0$ is small enough. Then \eqref{4-equations} has a zero solution in $\Theta$ by Brouwer's Fixed Point Theorem.

Let $\Lambda_0 := (x_0, y_0, z_0,w_0)=(1+x_n s^\nu+t_1 s^{\nu+\beta_1},y_n s^\nu+t_2 s^{\nu+\beta_2},1+t_3 s^{\nu+\beta_3},w_n s^\nu+t_4s^{\nu+\beta_4})\in \Theta$ be an initial point, where $|t_i|\leq 1$ for $1\leq i\leq 4$. Comparing \eqref{range-s} with \eqref{range-s-fixed}, it follows that $a_i$ and $b_i$ has the same order in $s$ for all $1\leq i\leq n-1$. By the definition of $\rho_1$ in \eqref{Q-n-equ-1}, $\rho_3$ in \eqref{rho-3-4-defi} and Lemma \ref{para-fixed} (2), we have
\begin{equation}
|\rho_1-1|\preceq s^{1+2/\dm} \text{ and } |\rho_3|\preceq s^{1+2/\dm}.
\end{equation}
Similarly, by \eqref{rho_2}, Lemma \ref{lema-useful-fixed-2} and the definition of $\rho_4$ in \eqref{rho-3-4-defi}, we have
\begin{equation}
|\rho_2-1/d_n|\preceq s^{(1+2/\dm)\nu} \text{~and~} |\rho_4|\preceq s^{(1+2/\dm)\nu}.
\end{equation}
The inverse of the Jacobi matrix of \eqref{4-equations} at $\Lambda_0$ is equal to
\[
\Jac(\Lambda_0)^{-1}=
\begin{pmatrix}
\frac{\partial f_1}{\partial x} & \frac{\partial f_1}{\partial y} & \frac{\partial f_1}{\partial z} & \frac{\partial f_1}{\partial w}\\
\frac{\partial f_2}{\partial x} & \frac{\partial f_2}{\partial y} & \frac{\partial f_2}{\partial z} & \frac{\partial f_2}{\partial w}\\
\frac{\partial f_3}{\partial x} & \frac{\partial f_3}{\partial y} & \frac{\partial f_3}{\partial z} & \frac{\partial f_3}{\partial w}\\
\frac{\partial f_4}{\partial x} & \frac{\partial f_4}{\partial y} & \frac{\partial f_4}{\partial z} & \frac{\partial f_4}{\partial w}
\end{pmatrix}_{\Lambda_0}^{-1}
=
\begin{pmatrix}
-\frac{d_1(d_1-2)}{(d_1-1)^2} & -\frac{d_n}{(d_1-1)^2} & \frac{d_1}{(d_1-1)^2} & \frac{1}{(d_1-1)^2 d_n}\\
-\frac{d_1}{d_1-1} & \frac{d_1 d_n}{d_1-1} & -\frac{d_1}{d_1-1} & -\frac{d_1}{(d_1-1)d_n}\\
0 & -d_n & 0 & 1/d_n\\
0 & 0     & 0  & 0
\end{pmatrix}+\mathcal{O}(s^\nu),
\]
where $\mathcal{O}(s^\nu)$ means that every element in the matrix $\Jac(\Lambda_0)^{-1}$ is differ at most $Cs^\nu$ from the corresponding element in the matrix at the rightmost above and $C<+\infty$ is a positive constant depending only on $d_1,\cdots,d_n$.

Now we can estimate $F(\Lambda_0)$. By \eqref{xn-yn-wn}, we have
\begin{equation*}
\left\{
\begin{split}
f_1(\Lambda_0)&=-\frac{d_1-1}{d_1}t_1 s^{\nu+\beta_1}-\frac{1}{d_1}t_2 s^{\nu+\beta_2}
 -\frac{1}{d_1}t_3 s^{\nu+\beta_3}+t_4 s^{\nu+\beta_4}+\mathcal{O}(s^{\varsigma_1}),\\
f_2(\Lambda_0)&=-\frac{1}{d_n}t_3 s^{\nu+\beta_3}+t_4 s^{\beta_4}+\mathcal{O}(s^{\varsigma_2}),\\
f_3(\Lambda_0)&=\frac{d_1-1}{d_1}t_1 s^{\nu+\beta_1}-\frac{d_1-2}{d_1}t_2 s^{\nu+\beta_2}
 -\frac{d_1-1}{d_1}t_3 s^{\nu+\beta_3}+t_4 s^{\nu+\beta_4}+\mathcal{O}(s^{\varsigma_1}),\text{~and~}\\
f_4(\Lambda_0)&=d_n^2 t_4 s^{\beta_4}+\mathcal{O}(s^{\varsigma_2}).
\end{split}
\right.
\end{equation*}
where $\varsigma_1\geq \min\{2\nu,1+2/\dm\}$ and $\varsigma_2\geq (1+2/\dm)\nu$. Note that $\varsigma_1>\varsigma_2$. Then we have
\begin{equation*}
\Jac(\Lambda_0)^{-1}F(\Lambda_0)
= \Jac(\Lambda_0)^{-1}
\begin{pmatrix}
f_1(\Lambda_0)\\
f_2(\Lambda_0)\\
f_3(\Lambda_0)\\
f_4(\Lambda_0)
\end{pmatrix}
=
\begin{pmatrix}
t_1 s^{\nu+\beta_1}+\mathcal{O}(s^{v+\beta_4})\\
t_2 s^{\nu+\beta_2}+\mathcal{O}(s^{v+\beta_4})\\
t_3 s^{\nu+\beta_3}+\mathcal{O}(s^{\varsigma_2})\\
t_4 s^{\nu+\beta_4}+o(s^{2v})
\end{pmatrix}.
\end{equation*}

Let $\Lambda_1=N_F(\Lambda_0)=(x_1,y_1,z_1,w_1)$. By \eqref{Newton's-mtd}, we have
\begin{equation*}
|x_1-(1+x_n s^\nu)|\preceq s^{\nu+\beta_4}, |y_1-y_n s^\nu|\preceq s^{\nu+\beta_4}, |z_1-1|\preceq s^{\varsigma_2} \text{~and~} |w_1-w_n s^\nu|\preceq s^{2\nu}.
\end{equation*}
By the choice of $\beta_1,\beta_2,\beta_3$ and $\beta_4$, this means that $N_F(\Theta)\subset \Theta$ if $s>0$ is small enough.
According to Brouwer's Fixed Point Theorem, every continuous function $f$ defined from a convex compact subset $K$ of a Euclidean space to $K$ itself has a fixed point. Since $N_F$ is continuous, $\Theta$ is convex and compact, it follows that the set of equations \eqref{4-equations} has a zero solution in $\Theta$.
This ends the proof of Lemma \ref{zero-solution} and hence Lemma \ref{lema-para-fixed-II} holds.
\end{proof}

As a remark, we explain here how to find the constants $x_n,y_n$ and $w_n$ which are defined in \eqref{xn-yn-wn}.
By the expressions of $\rho_1,\rho_2,\rho_3$ and $\rho_4$, we have $(\rho_1,\rho_2,\rho_3,\rho_4)\to (1,1/d_n,0,0)$ as $s\to 0^+$. As a priori, it is naturally conjecture that the set of four equations \eqref{4-equations} has a zero solution near $(1,0,1,0)$ if $s$ is small enough.

By \eqref{Q-n-equ-2} and \eqref{rho_2}, we have $\lim_{s\to 0^+}W_n/s^\nu=(d_n-1)/d_n=w_n$. As a priori, let $X_n=1+x_n s^\nu+o(s^\nu)$, $Y_n=y_n s^\nu+o(s^\nu)$ and $Z_n=1+z_n s^{\nu+\varsigma}+o(s^{\nu+\varsigma})$ for some $\varsigma>0$. By \eqref{Q-n-equ-1} and \eqref{Q-n-equ-3}, if we compare the order $s^\nu$, we have
\begin{equation*}
((d_1-1)X_n+Y_n+Z_n)^2-d_1((d_1-1)Y_n+d_1 Z_n)=o(s^\nu),
\end{equation*}
which is equivalent to
\begin{equation}\label{xn-yn-1}
2(d_1-1)x_n=(d_1-3)y_n.
\end{equation}
Since $\lim_{s\to 0^+}W_n/s^\nu=w_n$, if we compare the order $s^\nu$ in the two sides of equation \eqref{Q-n-equ-1}, we have
\begin{equation}\label{xn-yn-2}
(d_1-1)x_n+y_n=d_1(d_n-1)/d_n.
\end{equation}
Combining \eqref{xn-yn-1} and \eqref{xn-yn-2}, we then have \eqref{xn-yn-wn}.

As a special case of Lemma \ref{lema-para-fixed-II}, if $d_i=K$ for $1\leq i\leq n$ and $K\geq n+1$, then $\nu=(n-1)/(K-1)$ and
\begin{equation}
\begin{split}
\lim_{s\to 0^+}\frac{X_n-1}{s^\nu}=\frac{K-3}{K-1},
&~~ \lim_{s\to 0^+}\frac{Y_n}{s^\nu}=2,\\
\lim_{s\to 0^+}\frac{Z_n-1}{s^{2\nu}}=-\frac{2(K+1)}{K}\text{~and}
&~~ \lim_{s\to 0^+}\frac{W_n}{s^\nu}=\frac{K-1}{K}.
\end{split}
\end{equation}

Let $\varphi(z)=s^\nu/z$, define
\begin{equation*}
\widehat{Q}_n:=\varphi\circ Q_n\circ\varphi^{-1}.
\end{equation*}
Note that the conjugacy $\varphi$ exchanges the two parabolic fixed points $1$ and $s^\nu$. The following Lemma \ref{Q_n_limit} indicates that the rational maps $Q_n$ and $\widehat{Q}_n$, respectively, can be served as a small perturbation of $h_{d_1}(z)= d_1 z^{d_1}/((d_1-1)z^{d_1}+1)$ and $h_{d_n}(z)= d_n z^{d_n}/((d_n-1)z^{d_n}+1)$, which are parabolic maps defined in \eqref{defi-h}.

\begin{lema}\label{Q_n_limit}
The rational maps $Q_n$ and $\widehat{Q}_n$, respectively, converge to $h_{d_1}$ and $h_{d_n}$ locally uniformly on $\EC\setminus\{0\}$ as the parameter $s>0$ tends to zero.
\end{lema}

\begin{proof}
Let $X_n$, $Y_n$, $Z_n$ and $W_n$ be the four numbers in \eqref{family-para-para-fixed}. Since $n\geq 3$ is odd, we have
\begin{equation*}\label{family-para-odd-Q_n}
Q_n(z)=\frac{d_1 z^{d_n}}{(d_1-1)X_n z^{d_1}+Y_n z+Z_n}\,\frac{z^{d_1+d_2}-b_1^{d_1+d_2}}{z^{d_2+d_3}-b_2^{d_2+d_3}}
\cdots\frac{z^{d_{n-2}+d_{n-1}}-b_{n-2}^{d_{n-2}+d_{n-1}}}{z^{d_{n-1}+d_n}-b_{n-1}^{d_{n-1}+d_n}}+W_n.
\end{equation*}
By Lemma \ref{lema-para-fixed-II}, it follows that $X_n$ and $Z_n$ tend to $1$, $b_1,\cdots,b_{n-1}$, $Y_n$ and $W_n$ tend to $0$ as $s>0$ tends to $0$. By the expressions of $Q_n$, it follows that $Q_n$ converges to $h_{d_1}$ locally uniformly on $\EC\setminus\{0\}$ as $s>0$ tends to zero.

Note that $n\geq 3$ is odd. By \eqref{b_n-prod-simplify} and making a straightforward calculation, we have
\begin{equation*}
1/\widehat{Q}_n(z) =
\frac{1/d_n}{(d_1-1)X_n s^{d_1\nu}z^{d_n-d_1}+Y_n s^\nu z^{d_n-1}+Z_n z^{d_n}}
\prod_{i=1}^{n-1}(1-\frac{s^{D_i\nu}}{b_i^{D_i}z^{D_i}})^{(-1)^{i-1}}+\frac{W_n}{s^\nu}.
\end{equation*}
By Lemma \ref{lema-para-fixed-II}, we have $\lim_{s\to 0^+}W_n/s^\nu=1-1/d_n$. For any $z\in\EC\setminus\{0\}$, we know that $(d_1-1)X_n s^{d_1\nu}z^{d_n-d_1}+Y_n s^\nu z^{d_n-1}+Z_n z^{d_n}$ tends to $z^{d_n}$ as $s\to 0^+$ since $(X_n,Y_n,Z_n)\to(1,0,1)$. Moreover, by Lemma \ref{lema-useful-fixed-2} or \eqref{rho_2}, we know that $\prod_{i=1}^{n-1}(1-{s^{D_i\nu}}/{(b_i z)^{D_i}})^{(-1)^{i-1}}$ tends to $1$ as $s\to 0^+$ if $z\neq 0$. This means that $1/\widehat{Q}_n(z)$ tends to $z\mapsto ((d_n-1)z^{d_n}+1)/(d_n z^{d_n})$. Equivalently, $\widehat{Q}_n$ tends to $h_{d_n}$ as $s>0$ tends to zero.
\end{proof}

\begin{thm}\label{parameter-parabolic-two-fixed-resta}
Let $b_1=(\dm^2 \tau s)^{1/d_1}$ and $b_i=(\tau s)^{1/d_i}\,b_{i-1}$ be the numbers defined in \eqref{range-s-fixed} for $2\leq i\leq n-1$. There exist suitable $X_n,Y_n,Z_n$ and $W_n$ such that if the parameter $s>0$ is small enough, then the Julia set of $Q_n$ is a Cantor set of circles with two parabolic fixed points $1$ and $s^\nu$.
\end{thm}

\begin{proof}
We will not give the very detailed proof of Theorem \ref{parameter-parabolic-two-fixed-resta} since the proof technique is almost the same as that in the last section. By Lemmas \ref{P_n_limit} and \ref{Q_n_limit}, the parabolic maps $P_n$ and $Q_n$ can be both seen as a small perturbation of $h_{d_1}$. Compare \eqref{range-s} and \eqref{range-s-fixed}, the parameters $a_i$ in $P_n$ and $b_i$ in $Q_n$ are chosen in the same order of $s$ and the only difference is the constant $\tau$. As in Lemma \ref{crit-close-Parameter}, besides $0$ and $\infty$, the rest $\sum_{i=1}^{n-1}D_i$ critical points $\bigcup_{i=1}^{n-1}\Crit_i$ of $Q_n$ are very `close' to the reference points $\bigcup_{i=1}^{n-1}\{r_i b_i e^{\pi \textup{i}(2j-1)/D_i}:1\leq j\leq D_i\}$, where $r_i:=\sqrt[D_i]{d_{i+1}/d_{i}}$ and $1\leq i\leq n-1$. Hence, as in Lemma \ref{lemma-want}, there exist a simply connected domain $U_\infty$ such that $Q_n(\overline{U}_\infty)\subset U_\infty\cup\{1\}$ and $n-1$ annuli $A_1,\cdots, A_{n-1}$ satisfying $A_{n-1}\prec \cdots\prec A_1$ and each $A_i$ contains $\mathbb{T}_{b_i}\cup \mathbb{T}_{r_i b_i} \cup \Crit_i$ compactly.

Similar to the argument of \eqref{P_n-est-1}, one can obtain $Q_n(\overline{A}_i)\subset U_\infty$ for even $i$. Now we prove the existence of $U_0$ such that $Q_n(\overline{A}_i)\subset U_0$ for odd $i$ and $U_0$ is contained in the immediate parabolic basin of $s^\nu$. By Lemma \ref{Q_n_limit}, the map  $\widehat{Q}_n(z)$ can be regarded as a small perturbation of the parabolic map $h_{d_n}(z)= d_n z^{d_n}/((d_n-1)z^{d_n}+1)$. Therefore, the immediate parabolic basin of $1$ of $\widehat{Q}_n$ `almost' contains the outside of the closed unit disk by Lemma \ref{lema-approx} (compare Corollary \ref{key-lemma-complex}). This means that the immediate parabolic basin of $s^\nu$ of $Q_n$ `almost' contains the round disk $\D(0,s^\nu)$ (see Figure \ref{Fig_C-C-F}). In particular, let $U_0:=\D(s^\nu/4,3s^\nu/4)$, then $U_0$ is contained in the immediate parabolic basin of $s^\nu$. By \eqref{P_n-est-0}, one can also obtain that $Q_n$ maps the annulus $A_i$ to a round disk with radius less than $Cs$ for odd $i$, where $C>0$ is a constant depending only on $d_1,\cdots,d_n$. Note that $\nu<1$, this means that $Q_n(\overline{A}_i)\subset U_0$ for odd $i$ if $s$ is small enough.

Up to now, we have obtained the four conclusions which are parallel to Lemma \ref{lemma-want} for the rational map $Q_n$. By applying a completely similar argument in the proof of Theorem \ref{parameter-parabolic-resta}, one can prove the Julia set of $Q_n$ is a Cantor set of circles. This ends the proof of Theorem \ref{parameter-parabolic-two-fixed-resta} and hence Theorem \ref{parameter-parabolic-two-fixed}.
\end{proof}

We now give a specific example such the Julia set of $Q_{d_1,\cdots,d_n}$ is a Cantor set of circles. Let $n=3$, $d_1=d_2=d_3=4$ and $s=10^{-8}$. By \eqref{tau}--\eqref{my-nu}, we have $\dm =4$, $\tau=4$, $\nu=2/3$, $b_1=(\dm^2 \tau s)^{1/d_1}=2\sqrt{2}\times 10^{-2}$ and $b_2=(\tau s)^{1/d_2}b_1=4\times 10^{-4}$. By \eqref{Q-n-equ-1}, \eqref{rho-2-defi}, \eqref{rho-3-4-defi} and a direct calculation, we have
\begin{equation*}
\rho_1\approx 1-4.096\times 10^{-13}, \rho_2\approx 0.25+10^{-16},
\rho_3\approx 3.2768\times 10^{-12}\text{~and~}\rho_4\approx 2.6\times 10^{-15}.
\end{equation*}
Solving the following set of equations which corresponds to \eqref{Q-n-equ-1}, \eqref{Q-n-equ-2}, \eqref{Q-n-equ-3} and \eqref{Q-n-equ-4}  with unknown numbers $X_3$, $Y_3$, $Z_3$ and $W_3$:
\begin{equation*}
\left\{                         
\begin{array}{l}               
4\rho_1=(1-W_3)(3X_3+Y_3+Z_3)          \\
\rho_2=(1-10^{16/3}W_3)(3\times 10^{-64/3}X_3+10^{-16/3}Y_3+Z_3) \\
1/(1-W_3)=\rho_3+(3Y_3+4 Z_3)/(3X_3+Y_3+Z_3) \\
\DF{1}{1-10^{16/3}W_3}=\rho_4+4-\DF{12\times 10^{-64/3}X_3+10^{-16/3}Y_3}{3\times 10^{-64/3}X_3+10^{-16/3}Y_3+Z_3},
\end{array}
\right.                         
\end{equation*}
we have
\begin{equation*}
\begin{split}
X_3\approx 1+1.5471913857\times 10^{-6}, ~ & ~ Y_3\approx 9.2832930409\times 10^{-6},\\
Z_3\approx 1-5.38605\times 10^{-11} \text{~and}~ & ~ W_3\approx 3.4811916252\times 10^{-6}.
\end{split}
\end{equation*}
The Julia set of
\begin{equation*}
Q_{4,4,4}(z)=\frac{4z^4(z^8-b_1^8)}{(3X_3z^4+Y_3 z+Z_3)(z^8-b_2^8)}+W_3,
\end{equation*}
is a Cantor set of circles with two parabolic fixed points $1$ and $s^\nu=10^{-16/3}\approx 4.64\times 10^{-6}$ (see Figure \ref{Fig_C-C-F} and compare Figure \ref{Fig_parabolic-polynomial}).

\begin{figure}[!htpb]
  \setlength{\unitlength}{1mm}
  \centering
  \includegraphics[width=70mm]{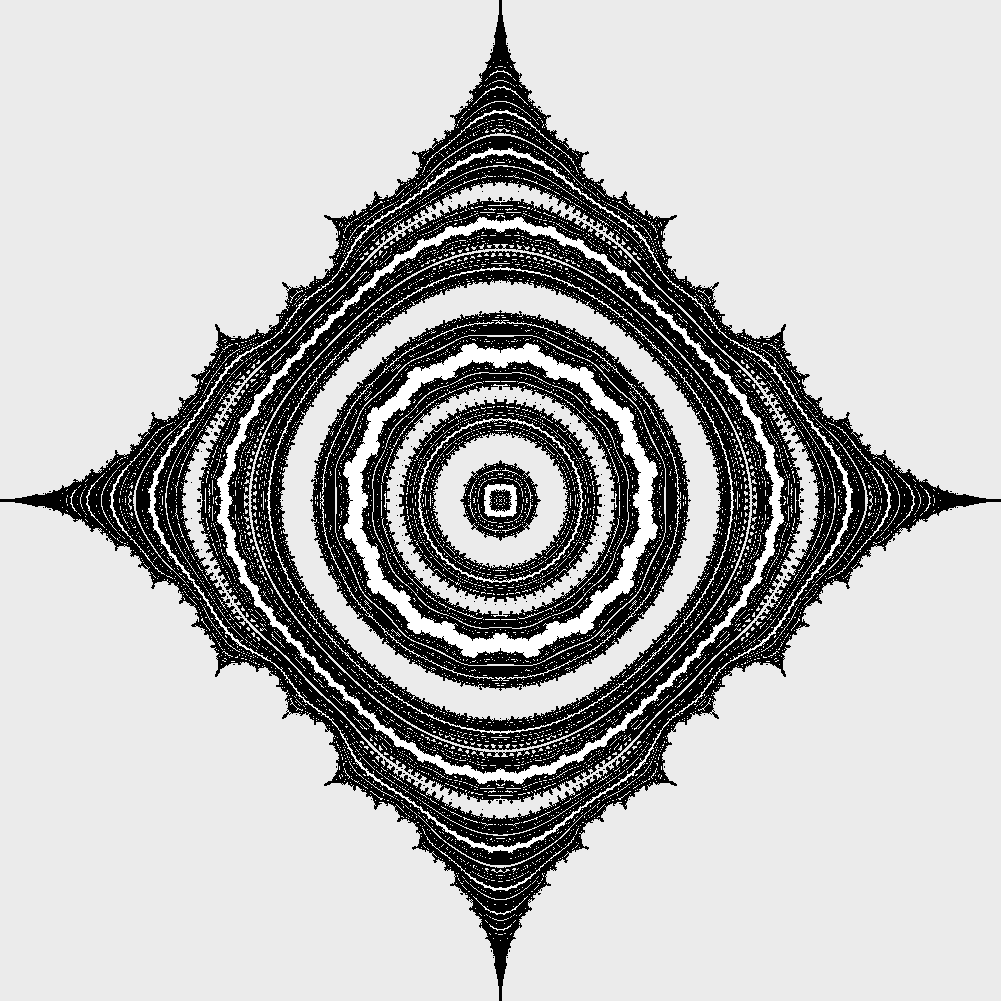}
  \includegraphics[width=70mm]{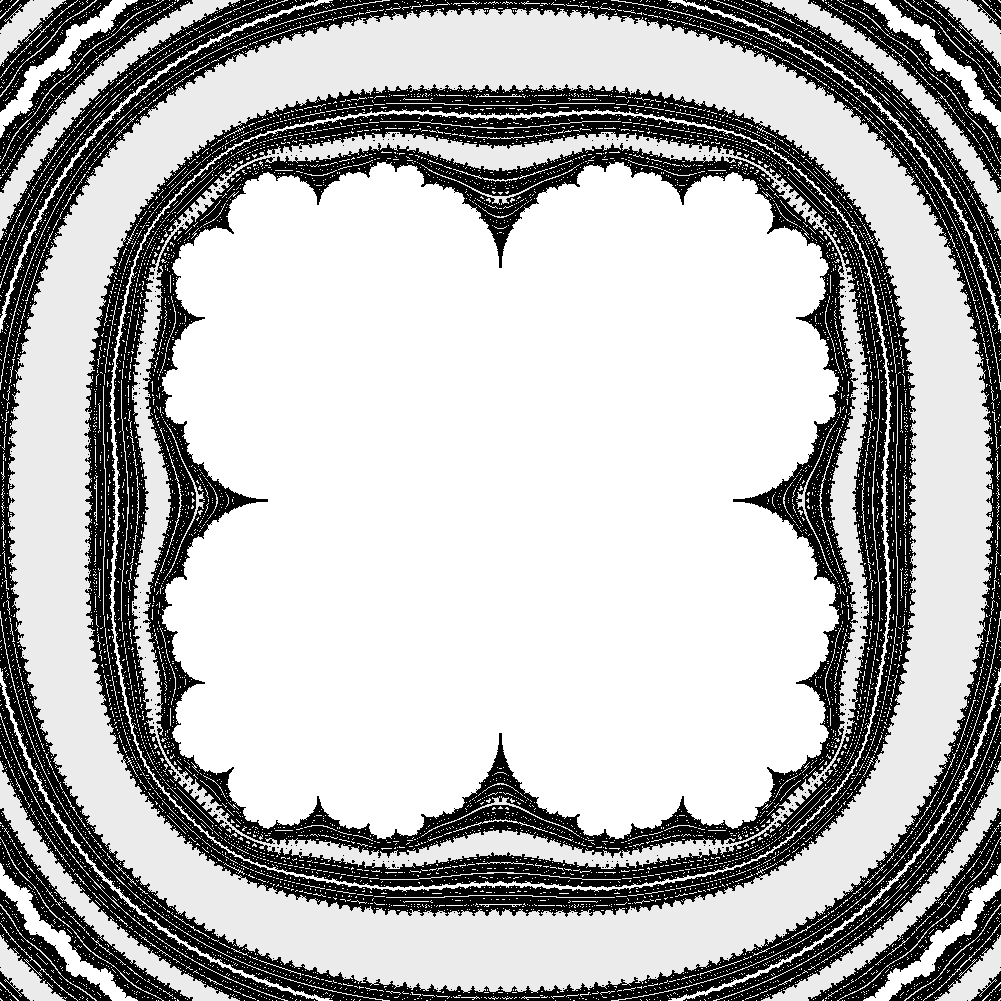}
  \caption{The Julia set of $Q_{4,4,4}$ and its zoom near the parabolic fixed point $s^\nu\approx 4.64\times 10^{-6}$ for suitable parameters, which is a Cantor set of circles. The white and gray parts in the Figure denote the Fatou components which are iterated to the parabolic fixed Fatou components whose boundaries containing $s^\nu$ and $1$ respectively. Figure ranges: $[-1,1]^2$ and $[-10^{-5},10^{-5}]^2$.}
  \label{Fig_C-C-F}
\end{figure}

To finish this section, we say some words about the expression of $Q_{d_1,\cdots,d_n}$ defined in \eqref{family-para-para-fixed}. One may ask why we don't choose the following expression for consideration:
\begin{equation*}\label{family-para-fixed-old}
\widetilde{Q}_{d_1,\cdots,d_n}(z)=\frac{d_1 (\widetilde{X}_n z^{d_n}+\widetilde{Y}_n z+\widetilde{Z}_n)}{(d_1-1)z^{d_1}+1}
\prod_{i=1}^{n-1}(z^{d_i+d_{i+1}}-b_i^{d_i+d_{i+1}})^{(-1)^{i-1}}+\widetilde{W}_n.
\end{equation*}
This map can be also regarded as a perturbation of the parabolic rational map $h_{d_1}(z)=d_1 z^{d_1}/((d_1-1)z^{d_1}+1)$ as $(\widetilde{X}_n,\widetilde{Y}_n,\widetilde{Z}_n,\widetilde{W}_n)\to(1,0,0,0)$ and $b_i\to 0$ for $1\leq i\leq n-1$. Actually, we did it originally. We try to find suitable $\widetilde{X}_n\approx 1,\widetilde{Y}_n\approx 0,\widetilde{Z}_n\approx 0$ and $\widetilde{W}_n\approx 0$ such that $\widetilde{Q}_{d_1,\cdots,d_n}$ has two parabolic fixed points and the corresponding Julia set is a Cantor set of circles. Unfortunately, it is easy to find $\widetilde{X}_n,\widetilde{Y}_n,\widetilde{Z}_n$ and $\widetilde{W}_n$ such that $\widetilde{Q}_{d_1,\cdots,d_n}$ has two parabolic fixed points $1$ and $s^\nu$ and both with multiplier $1$, but it cannot grantee the Julia set of $\widetilde{Q}_{d_1,\cdots,d_n}$ is a Cantor set of circles. An important reason is we cannot control the critical orbits near the parabolic fixed point $s^\nu$ since there are $d_n-1$ different critical points there. What we have done is to let the $d_n-1$ different critical points near $s^\nu$ to become only one, but with multiplicity $d_n-1$. This is the essential reason why we choose $Q_{d_1,\cdots,d_n}$ but not $\widetilde{Q}_{d_1,\cdots,d_n}$ eventually.

Finally, we say some words about the constant $\nu$ defined in \eqref{my-nu}. As the order of the parabolic fixed point $s^\nu$, this number was found by considering \eqref{bound-lower-in-disk-even}. Note that \eqref{bound-lower-in-disk-even} is an inequality. What we want to obtain is an equation $Q_n(s^\nu)=s^\nu$. Comparing the order of $s$ in both sides of $s^{d_n\nu}/b_{n-1}^{d_n}\asymp s^\nu$, we have $d_n\nu-\sum_{i=1}^{n-1}(d_n/d_i)=\nu$ by \eqref{range-s-fixed}. The solution of $\nu$ in this equation is just \eqref{my-nu}, as desired.

\section{Cantor circles with parabolic periodic points}\label{sec-para-period=2}

In this section, we will construct a family of non-hyperbolic rational maps whose Julia sets are Cantor circles such that each one of them has a parabolic periodic Fatou component with period $2$. The construction is the \textit{most} difficult case among all the constructions of the Cantor circles (including the constructions in \cite{QYY}) since we need to control the sizes of two parabolic periodic Fatou components and locate the positions of two parabolic periodic points which lie on the boundaries of these two parabolic basins.

As in last section, we recommend the reader to take a look at the right picture in Figure \ref{fig_parabolic_III-IV} in which a rough indication of the dynamics of such rational maps has been made. Let $n\geq 3$ be an odd number and $d_1,\cdots,d_n$ be $n$ positive integers such that $\sum_{i=1}^{n}(1/d_i)<1$. We define
\begin{equation}\label{family-para-R_n}
R_{d_1,\cdots,d_n}(z)=\frac{S_n}{z^{d_n}}
\prod_{i=1}^{n-1}(z^{d_i+d_{i+1}}-c_i^{d_i+d_{i+1}})^{(-1)^i}+T_n,
\end{equation}
where $c_1,\cdots,c_{n-1}$ are $n-1$ small \textit{real} numbers and $S_n,T_n$ are numbers depending only on $c_1,\cdots,c_{n-1}$. As before, let $\dm\geq 3$ be the maximal number among $d_1,\cdots,d_n$. Set
\begin{equation}\label{range-s-periodic}
c_1=(\dm^2 s)^{1/d_1} \text{~ and ~} c_i=s^{1/d_i}\,c_{i-1} \text{~for~} 2\leq i\leq n-1,
\end{equation}
where $s>0$ is the parameter which is small enough.

In the following, we show that there exist suitable $S_n$ and $T_n$ such the Julia set of $R_{d_1,\cdots,d_n}$ is a Cantor set of circles which contain two parabolic periodic points with period $2$. Actually, the largest trouble in the construction of $R_{d_1,\cdots,d_n}$ has been solved since we have already wrote the specific expression for $R_{d_1,\cdots,d_n}$. The next problem is to locate the positions of two parabolic periodic points. As in \S\ref{sec-para-1-fixed}, we can set $1$ as one of the parabolic periodic point. For another parabolic periodic point, we find it by solving a set of equations.

For simplicity, we use $R_n$ to denote $R_{d_1,\cdots,d_n}$ for the fixed integers $d_1,\cdots,d_n$ satisfying $\sum_{i=1}^{n}(1/d_i)<1$. As before, define $D_i:=d_i+d_{i+1}$, where $1\leq i\leq n-1$.

\begin{lema}\label{lema-para-periodic-2}
There exist suitable $S_n$ and $T_n$ such that $R_n$ has a parabolic periodic orbits $1\leftrightarrow z_0$ with multiplier $1$. In particular, $R_n(1)=z_0$, $R_n(z_0)=1$ and $R_n'(1)R_n'(z_0)=1$. Moreover,
\begin{equation}\label{S_n-T_n}
\lim_{s\to 0^+}S_n/s^\nu=\mu,~~ \lim_{s\to 0^+}T_n/s^\nu=(d_1 d_n -1)\mu \text{~and~} \lim_{s\to 0^+}z_0/s^\nu= d_1 d_n\mu,
\end{equation}
where $\nu$ is the constant depending only on $d_1,\cdots,d_n$ which was defined in \eqref{my-nu} and
\begin{equation}\label{my-mu}
\mu=(d_1 d_n)^{-\frac{d_n}{d_n-1}}\dm^{\frac{2(d_n-d_1)}{d_1(d_n-1)}}.
\end{equation}
\end{lema}

\begin{proof}
Let
\begin{equation}\label{S_n-T_n-z_0}
S_n=\mu s^\nu I_n,~T_n=(d_1 d_n-1) \mu s^\nu J_n \text{~~and~~} z_0= d_1 d_n \mu s^\nu z_1.
\end{equation}
We need to show that $I_n, J_n$ and $z_1$ all tend to $1$ as $s>0$ tends to $0$. By $R_n(1)=z_0$ and \eqref{family-para-R_n}, we have $z_0=R_n(1)=\kappa_1 S_n+T_n$, which means that
\begin{equation}\label{R-n-equ-1}
d_1 d_n z_1=\kappa_1 I_n+(d_1d_n-1)J_n, \text{~where~} \kappa_1=\prod_{i=1}^{n-1}(1-c_i^{D_i})^{(-1)^i}.
\end{equation}
Since $c_1=(\dm^2 s)^{1/d_1}$ and $c_i=s^{1/d_i}\,c_{i-1}$ for $2\leq i\leq n-1$ by \eqref{range-s-periodic}, we have
\begin{equation}\label{R-n-use-1}
\prod_{i=1}^{n-1}c_i^{(-1)^i D_i}=\frac{c_{n-1}^{d_n}}{\dm^2}
=\frac{s^{d_n((1/d_1)+\cdots+(1/d_{n-1}))}}{\dm^{\, 2(d_1-d_n)/d_1}}=\frac{s^{(d_n-1)\nu}}{\dm^{\, 2(d_1-d_n)/d_1}}.
\end{equation}
By the condition $R_n(z_0)=1$, \eqref{R-n-use-1} and the definitions of $\nu$ in \eqref{my-nu} and $\mu$ in \eqref{my-mu}, we have
\begin{equation}\label{R-n-equ-2}
1=R_n(z_0)=\kappa_2 I_n/z_1^{d_n}+ (d_1 d_n-1) \mu s^\nu J_n,
\end{equation}
where
\begin{equation}\label{defi-kappa-2}
\kappa_2=\prod_{i=1}^{n-1}\left(\frac{z_0^{D_i}}{c_i^{D_i}}-1\right)^{(-1)^i}
=\prod_{i=1}^{n-1}\left(\frac{(d_1 d_n \mu s^\nu)^{D_i} z_1^{D_i}}{c_i^{D_i}}-1\right)^{(-1)^i}.
\end{equation}

In order to calculate the derivation of $R_n$, we consider
\begin{equation*}
\frac{zR_n'(z)}{R_n(z)-T_n}= \sum_{i=1}^{n-1}\frac{(-1)^i D_iz^{D_i}}{z^{D_i}-c_i^{D_i}}-d_n.
\end{equation*}
Then we have following two equations:
\begin{equation}\label{solu-crit-parabolic-2}
\frac{R_n'(1)}{z_0-T_n}=\kappa_3-d_1 \text{~and~}
\frac{z_0R_n'(z_0)}{1-T_n}=\kappa_4-d_n,
\end{equation}
where
\begin{equation}\label{kappa-3-4}
\kappa_3=\sum_{i=1}^{n-1}\frac{(-1)^i D_i\,c_i^{D_i}}{1-c_i^{D_i}} \text{~~and~~}
\kappa_4=\sum_{i=1}^{n-1}\frac{(-1)^i D_i z_0^{D_i}}{z_0^{D_i}-c_i^{D_i}}
             =\sum_{i=1}^{n-1}\frac{(-1)^i D_i ( d_1 d_n\mu s^\nu)^{D_i} z_1^{D_i}}{( d_1 d_n\mu s^\nu)^{D_i} z_1^{D_i}-c_i^{D_i}}.
\end{equation}
By the condition $R_n'(1)R_n'(z_0)=1$ and \eqref{solu-crit-parabolic-2}, we have
\begin{equation}
\frac{z_0}{(z_0-T_n)(1-T_n)}=(\kappa_3-d_1)(\kappa_4-d_n).
\end{equation}
This equation is equivalent to
\begin{equation}\label{R-n-equ-3}
\frac{z_1}{(d_1 d_n z_1-(d_1 d_n -1) J_n)(1- (d_1 d_n-1)\mu s^\nu J_n)}=(1-\frac{\kappa_3}{d_1})(1-\frac{\kappa_4}{d_n}).
\end{equation}

By the formulas of $\kappa_1,\kappa_2,\kappa_3$ and $\kappa_4$ in \eqref{R-n-equ-1}, \eqref{defi-kappa-2} and \eqref{kappa-3-4}, we have $(\kappa_1,\kappa_2,\kappa_3,\kappa_4)\rightarrow (1,1,0,0)$ as $s\rightarrow 0^+$. On the other hand, if $s=0$ and $(\kappa_1,\kappa_2,\kappa_3,\kappa_4)= (1,1,0,0)$, then the set of three equations \eqref{R-n-equ-1}, \eqref{R-n-equ-2} and \eqref{R-n-equ-3} has a solution $(I_n,J_n,z_1)=(1,1,1)$. By using the Newton's method as in the proof of Lemma \ref{zero-solution}, one can also show that the solution $(I_n,J_n,z_1)\approx (1,1,1)$ satisfying \eqref{R-n-equ-1}, \eqref{R-n-equ-2} and \eqref{R-n-equ-3} exists. We omit the details here.
\end{proof}

Let $\psi(z)=z_0/z$, where $z_0\approx d_1d_n\mu s^\nu$ is the parabolic periodic point of $R_n$ appeared in Lemma \ref{lema-para-periodic-2}. Define
\begin{equation*}
\widehat{R}_n:=\psi\circ R_n\circ\psi^{-1}.
\end{equation*}
Note that the conjugacy $\psi$ exchanges the two parabolic periodic points $1$ and $z_0$ by \eqref{S_n-T_n}. The following Lemma \ref{R_n_limit} indicates that the rational maps $R_n^{\circ 2}$ and $\widehat{R}_n^{\circ 2}$ both can be served as a small perturbation of two parabolic rational maps.

\begin{lema}\label{R_n_limit}
The rational maps $R_n^{\circ 2}$ and $\widehat{R}_n^{\circ 2}$, respectively, converge to
\begin{equation}\label{new-parabolic}
h_{d_1,d_n}(z)=\frac{(d_1 d_n)^{d_n}z^{d_1 d_n}}{(1+(d_1 d_n-1)z^{d_1})^{d_n}} \text{~and~} h_{d_1 d_n}(z)=\frac{d_1 d_n z^{d_1 d_n}}{1+(d_1 d_n-1)z^{d_1 d_n}}
\end{equation}
locally uniformly on $\EC\setminus\{0\}$ as the parameter $s>0$ tends to zero.
\end{lema}

\begin{proof}
Let $S_n$ and $T_n$ be the two numbers in \eqref{family-para-R_n}. Since $n\geq 3$ is odd, we have
\begin{equation*}\label{family-para-odd-R_n}
R_n(z)=\frac{S_n}{z^{d_n}}\,\frac{z^{d_2+d_3}-c_2^{d_2+d_3}}{z^{d_1+d_2}-c_1^{d_1+d_2}}
\cdots\frac{z^{d_{n-1}+d_n}-c_{n-1}^{d_{n-1}+d_n}}{z^{d_{n-2}+d_{n-1}}-c_{n-2}^{d_{n-2}+d_{n-1}}}+T_n.
\end{equation*}
By \eqref{range-s-periodic}, it follows that $c_1,\cdots,c_{n-1}\to 0$ as $s\to 0^+$. This means that $R_n(z)/(S_n/z^{d_1}+T_n)$ converges to $1$ locally uniformly in $\EC\setminus\{0\}$ as $s\to 0^+$. By Lemma \ref{lema-para-periodic-2}, $\lim_{s\to 0^+}S_n/s^\nu=\mu$ and $\lim_{s\to 0^+}T_n/s^\nu=(d_1 d_n -1)\mu$. This means that $S_n/z^{d_1}+T_n\asymp s^\nu$ if $z\neq 0$. By Lemma \ref{lema-useful-fixed-2} and compare \eqref{range-s-fixed} and \eqref{range-s-periodic}, we have $\lim_{s\to 0^+}(S_n/z^{d_1}+T_n)^{D_i}/c_i^{D_i}=0$. Therefore, by \eqref{S_n-T_n}, \eqref{my-mu} and \eqref{R-n-use-1}, we have
\begin{equation}
\lim_{s\to 0^+}R_n(R_n(z))=\lim_{s\to 0^+}\frac{S_n}{(S_n/z^{d_1}+T_n)^{d_n}}\prod_{i=1}^{n-1}c_i^{(-1)^i D_i}=\frac{(d_1 d_n)^{d_n}z^{d_1 d_n}}{(1+(d_1 d_n-1)z^{d_1})^{d_n}}.
\end{equation}
This means that $R_n^{\circ 2}$ converges to $h_{d_1,d_n}$ locally uniformly on $\EC\setminus\{0\}$ as $s\to 0^+$.

Note that $n\geq 3$ is odd. After making a straightforward calculation, we have
\begin{equation*}
R_n\circ\psi(z) =R_n(z_0/z)=
\frac{S_n z^{d_n}}{z_0^{d_n}}\prod_{i=1}^{n-1}c_i^{(-1)^i D_i}
\prod_{i=1}^{n-1}(1-\frac{z_0^{D_i}}{c_i^{D_i}z^{D_i}})^{(-1)^i}+T_n.
\end{equation*}
By Lemma \ref{lema-useful-fixed-2} and compare \eqref{range-s-fixed} and \eqref{range-s-periodic}, we know that $\prod_{i=1}^{n-1}(1-{z_0^{D_i}}/{(c_i z)^{D_i}})^{(-1)^i}$ tends to $1$ as $s\to 0^+$ if $z\neq 0$. By \eqref{S_n-T_n}, \eqref{my-mu} and \eqref{R-n-use-1}, we know that $R_n\circ\psi(z)$ converges to $z^{d_n}$ locally uniformly on $\EC\setminus\{0\}$ as $s\to 0^+$. Hence, if $z\neq 0$, we have
\begin{equation}
\lim_{s\to 0^+}\psi\circ R_n(R_n\circ\psi(z))=\lim_{s\to 0^+}\frac{z_0}{S_n/z^{d_1d_n}+T_n}=\frac{d_1 d_n z^{d_1 d_n}}{1+(d_1 d_n-1)z^{d_1 d_n}}.
\end{equation}
This means that $\widehat{R}_n^{\circ 2}$ converges to $h_{d_1 d_n}$ locally uniformly on $\EC\setminus\{0\}$ as $s\to 0^+$.
\end{proof}

\begin{thm}\label{parameter-parabolic-two-periodic-resta}
Let $c_1=(\dm^2 s)^{1/d_1}$ and $c_i=s^{1/d_i}\,c_{i-1}$ be the numbers defined in \eqref{range-s-periodic} for $2\leq i\leq n-1$. There exist suitable $S_n$ and $T_n$ such that if the parameter $s>0$ is small enough, then the Julia set of $R_n$ is a Cantor set of circles containing a parabolic periodic orbit with period two.
\end{thm}

\begin{proof}
Similar to the proof of Theorem \ref{parameter-parabolic-two-fixed-resta}, we will not give the very detailed proof of Theorem \ref{parameter-parabolic-two-periodic-resta} since the proof technique is almost the same as that of Theorem \ref{parameter-parabolic}. Like the cases of $P_n$ and $Q_n$, besides $0$ and $\infty$, the rest $\sum_{i=1}^{n-1}D_i$ critical points $\bigcup_{i=1}^{n-1}\Crit_i$ of $R_n$ are very `close' to $\bigcup_{i=1}^{n-1}\{r_i c_i e^{\pi \textup{i}(2j-1)/D_i}:1\leq j\leq D_i\}$, where $r_i:=\sqrt[D_i]{d_{i+1}/d_{i}}$ and $1\leq i\leq n-1$. Although the specific expression of $R_n$ are very different from that of $P_n$ and $Q_n$, we only need to check the similar things in Lemma \ref{lemma-want} beginning with \eqref{Phi-n} since the parameters $c_i$ in $R_n$, $b_i$ in $Q_n$ and $a_i$ in $P_n$ have the same orders in $s$, where $1\leq i\leq n-1$ (see \eqref{range-s}, \eqref{range-s-fixed} and \eqref{range-s-periodic}).

Let $A_1,\cdots,A_{n-1}$ be the annuli defined as in Lemma \ref{lemma-want}. Similar to \eqref{esti-other}, if $z\in \overline{A}_i$ and $s$ is small enough, we have
\begin{equation}\label{esti-other-r-n}
|z/c_i|^{d_{i}}+|c_i/z|^{d_{i+1}}<\dm.
\end{equation}
By the definition of $R_n$ in \eqref{family-para-R_n}, define $\Psi_n(z): = (R_n(z)-T_n)/S_n$. By \eqref{esti-other-r-n}, if $z\in \overline{A}_i$ and $s$ is small enough, we have
\begin{equation}\label{abs-f-n-yang-R_n}
\begin{split}
& |\Psi_n(z)| =
\left\{                         
\begin{array}{ll}               
\dm^{-2}s^{-1}\, |(z/c_i)^{d_i}-(c_i/z)^{d_{i+1}}|^{-1} \, \Upsilon_i(z)>\dm^{-3}s^{-1}\Upsilon_i(z)       &~\text{if}~i~\text{is odd}, \\
\dm^{-2}\, |(z/c_i)^{d_i}-(c_i/z)^{d_{i+1}}| \, \Upsilon_i(z)<\dm^{-1}\Upsilon_i(z)   &~\text{if}~i~\text{is even},   
\end{array}                     
\right.                         
\end{split}
\end{equation}
where
\begin{equation*}\label{Q-i-yang-R_n}
\Upsilon_i(z)=\prod_{j=1}^{i-1}\left|1-({z}/{c_j})^{D_j}\right|^{(-1)^j}
       \prod_{j=i+1}^{n-1}\left|1-({c_j}/{z})^{D_j}\right|^{(-1)^j}.
\end{equation*}
If $z\in \overline{A}_i$ for $1\leq j<i$, by Lemma \ref{very-useful-est}(2), we have
\begin{equation*}\label{estim-1-r-n}
|z/c_j|^{D_j}\asymp |c_i/c_j|^{D_j}\asymp s^{\alpha_2}\leq s^{1+2/\dm},
\end{equation*}
where $\alpha_2$ is defined in \eqref{alpha_2}. Similarly, if $z\in \overline{A}_i$ for $i<j\leq n-1$, we have
\begin{equation*}\label{estim-2-r-n}
|c_j/z|^{D_j}\asymp |c_j/c_i|^{D_j}\asymp s^{\alpha_1}\leq s^{1+2/\dm},
\end{equation*}
where $\alpha_1$ is defined in \eqref{alpha_1}. This means that
\begin{equation}\label{Q-i-esti-1-yang-R_n}
|\Upsilon_i(z)-1|\preceq s^{1+2/\dm}.
\end{equation}

For odd $i$ and sufficiently small $s>0$, if $z\in \overline{A}_i$, by \eqref{S_n-T_n}, \eqref{abs-f-n-yang-R_n} and \eqref{Q-i-esti-1-yang-R_n}, we have
\begin{equation}\label{est-0-R_n}
 |R_n(z)|\geq  |S_n|\,|\Psi_n(z)|-|T_n| >\dm^{-4}\mu s^{\nu-1}-1\asymp s^{\nu-1}.
\end{equation}
On the other hand, let $z\in \overline{A}_i$ for even $i$. By \eqref{S_n-T_n}, \eqref{abs-f-n-yang-R_n} and \eqref{Q-i-esti-1-yang-R_n}, if $s>0$ is sufficiently small, we have
\begin{equation}\label{est-1-R_n}
 |R_n(z)|\leq  |S_n|\,|\Psi_n(z)|+|T_n| <(d_1 d_n-1/2)\mu s^\nu.
\end{equation}

By Lemma \ref{R_n_limit}, the map  $\widehat{R}_n^{\circ 2}$ can be regarded as a small perturbation of the parabolic map $h_{d_1 d_n}$ if $s$ is small enough. Therefore, the immediate parabolic basin of $1$ of $\widehat{R}_n^{\circ 2}$ `almost' contains the outside of the closed unit disk by Lemma \ref{lema-approx} (compare Corollary \ref{key-lemma-complex}). This means that the immediate parabolic basin of $z_0\approx d_1 d_n\mu s^\nu$ of $R_n$ `almost' contains the round disk $\D(0,|z_0|)$. Hence ,we know that $R_n(A_i)$ is contained in the immediate parabolic basin of $z_0\approx d_1 d_n\mu s^\nu$ by \eqref{est-1-R_n} if $i$ is even. On the other hand, $s^{\nu-1}$ tends to $\infty$ as $s\to 0^+$ since $\nu<1$. It follows that $R_n(A_i)$ is contained in the immediate parabolic basin of $1$ of $R_n$ by \eqref{est-0-R_n} with odd $i$. By applying a completely similar argument in the proof of Theorem \ref{parameter-parabolic-resta}, one can also prove the Julia set of $R_n$ is a Cantor set of circles. This ends the proof of Theorem \ref{parameter-parabolic-two-periodic-resta} and hence Theorem \ref{parameter-parabolic-two-periodic}.
\end{proof}

Now we give a specific example such the Julia set of $R_{d_1,\cdots,d_n}$ is a Cantor set of circles. Let $n=3$, $d_1=d_2=d_3=4$ and $s=2^8\times 10^{-12}$. By \eqref{my-nu}, \eqref{range-s-periodic} and \eqref{my-mu}, we have $\dm =4$, $\mu=2^{-16/3}$, $\nu=2/3$, $\mu s^\nu=10^{-8}$, $c_1=(\dm^2  s)^{1/d_1}=8\times 10^{-3}$ and $c_2=s^{1/d_2}c_1=3.2\times 10^{-5}$. By \eqref{R-n-equ-1}, \eqref{defi-kappa-2}, \eqref{kappa-3-4} and a straightforward calculation, we have
$\kappa_1\approx 1$ and $\kappa_3\approx -1.34217728\times 10^{-16}$. Moreover,
\begin{equation*}
\kappa_2=\frac{5^8\times 10^{-24}z_1^8-1}{2^8\times 10^{-40}z_1^8-1} \text{~and~}
\kappa_4=-\frac{2^{11}\times 10^{-40}z_1^8}{2^8\times 10^{-40}z_1^8-1}+\frac{5^5\times 10^{-21}z_1^8}{5^8\times 10^{-24}z_1^8-1}.
\end{equation*}
Solving the following set of equations which corresponds to \eqref{R-n-equ-1}, \eqref{R-n-equ-2} and \eqref{R-n-equ-3} with unknown numbers $I_3$, $J_3$ and $z_1$:
\begin{equation*}
\left\{                         
\begin{array}{l}               
16z_1=\kappa_1 I_3+15 J_3          \\
1=\kappa_2 I_3/z_1^4+15\times 10^{-8}J_3\\
z_1=(1-\kappa_3/4)(1-\kappa_4/4)(16z_1-15 J_3)(1-15\times 10^{-8}J_3),
\end{array}
\right.                         
\end{equation*}
we have
\begin{equation*}
I_3\approx 1+2.5\times 10^{-7}, ~ J_3\approx 1+9\times 10^{-8} \text{~and~} z_1\approx 1+ 10^{-7}.
\end{equation*}
Therefore, by \eqref{S_n-T_n-z_0}, we have
\begin{equation*}
\begin{split}
S_3 & =10^{-8}I_3\approx 10^{-8}, ~ T_3=15\times 10^{-8}J_3\approx 1.5\times 10^{-7} \text{~and~} \\
z_0 & =16\times 10^{-8}J_3\approx 1.6\times 10^{-7}.
\end{split}
\end{equation*}
The Julia set of
\begin{equation*}
R_{4,4,4}(z)=\frac{S_3(z^8-c_2^8)}{z^4(z^8-c_1^8)}+T_3,
\end{equation*}
is a Cantor set of circles with two parabolic periodic points $1$ and $z_0\approx 1.6\times 10^{-7}$ and the period are both two (see Figure \ref{Fig_C-C-F-periodic} and compare Figure \ref{Fig_parabolic-polynomial}).

\begin{figure}[!htpb]
  \setlength{\unitlength}{1mm}
  \centering
  \includegraphics[width=70mm]{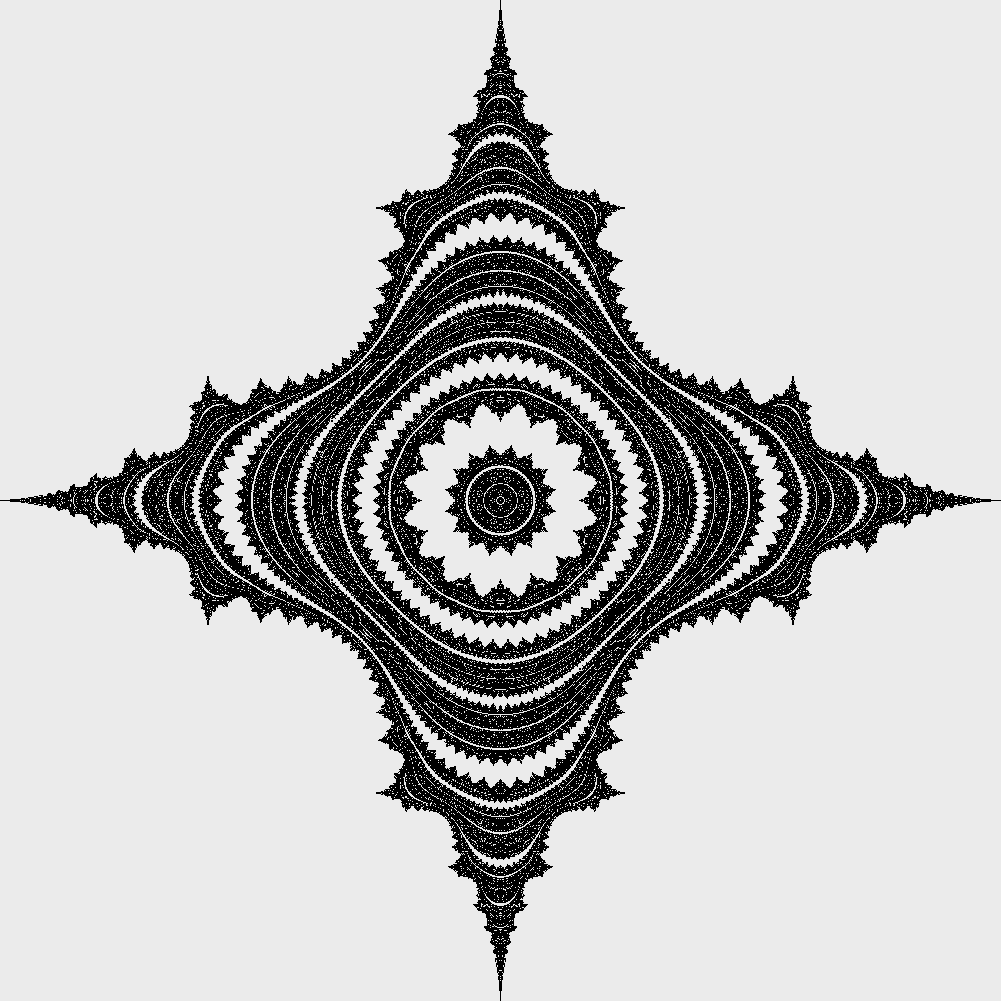}
  \includegraphics[width=70mm]{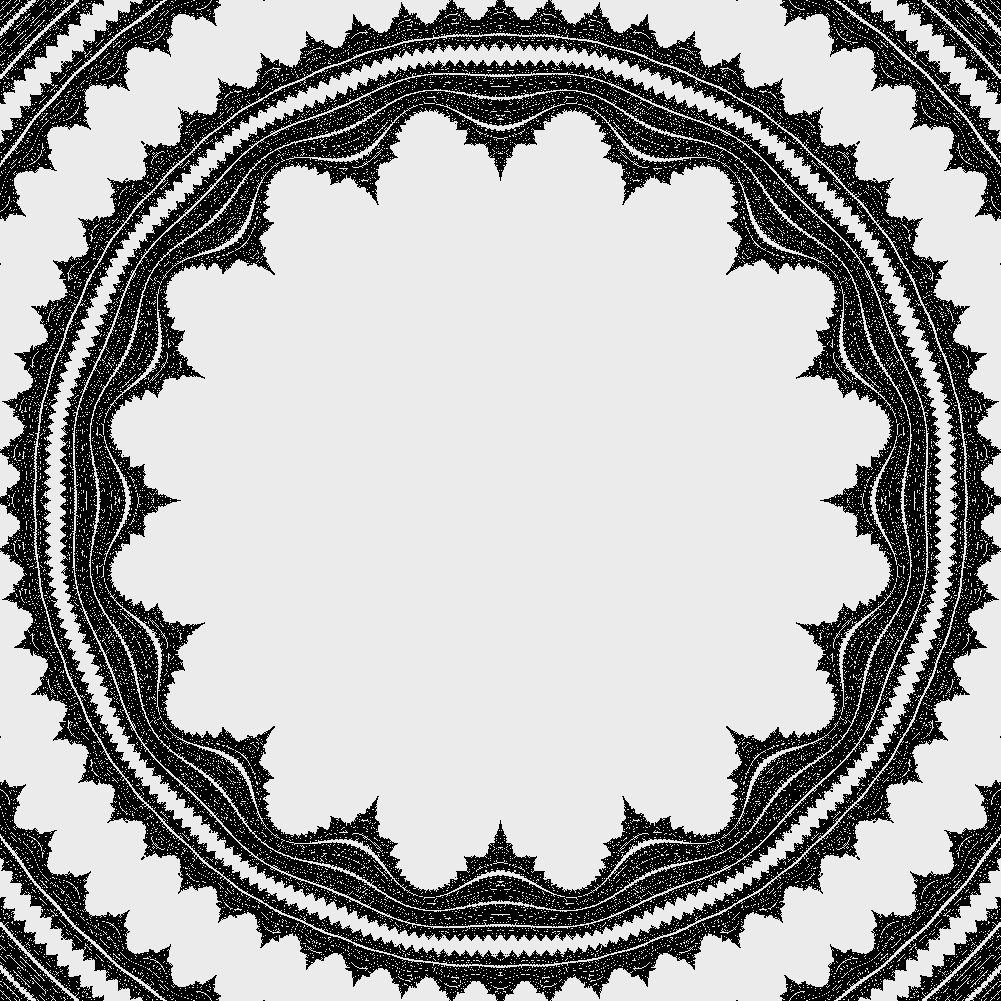}
  \caption{The Julia set of $R_{4,4,4}$ and its zoom near the parabolic periodic point $z_0\approx 1.6\times 10^{-7}$ for suitable parameters, which is a Cantor set of circles. The gray parts in the Figure denote the Fatou components which are iterated eventually to the parabolic periodic Fatou components whose boundaries containing $z_0$ and $1$. Figure ranges: $[-1,1]^2$ and $[-2.5\times 10^{-7},2.5\times 10^{-7}]^2$.}
  \label{Fig_C-C-F-periodic}
\end{figure}

\section{Quasisymmetric uniformization of the Cantor circles}\label{sec-quasi-unifor}

In this section, we give the quasisymmetric classification of the Cantor circles as the Julia sets of rational maps. The classification is divided into two broad categories: the hyperbolic case and the non-hyperbolic case.
The classification of topological conjugacy classes on the Cantor circles Julia sets for hyperbolic rational maps in \cite{QYY} is actually a quasiconformal classification. We state the main result in \cite{QYY} here for our further reference.

\begin{thm}\label{this-is-all-hyper}
Let $f$ be a hyperbolic rational map whose Julia set is a Cantor set of circles. Then there exist $p\in\{0,1\}$ and $n\geq 2$ positive integers $d_1,\cdots,d_n$ satisfying $\sum_{i=1}^{n}(1/d_i)<1$ such that $f$ is conjugate to $f_{p,d_1,\cdots,d_n}$ for suitable parameters on their corresponding Julia sets by a quasiconformal mapping.
\end{thm}

According to \cite[Theorem 11.14]{Hei}, an orientation preserving homeomorphism between the Riemann sphere to itself is quasisymmetric if and only if it is quasiconformal. This means that Theorem \ref{this-is-all-hyper} gives the quasisymmetric classification of the Cantor circles Julia sets for the hyperbolic rational maps. Moreover, the family $f_{d_1,\cdots,d_n}$ defined in \eqref{family-QYY} serves the quasisymmetric models for all hyperbolic Cantor circles Julia sets (compare Table \ref{Tab-classif}).

For non-hyperbolic rational maps with Cantor circles Julia sets, the family \eqref{family-QYY} only gives the topological models. In order to give a quasisymmetric classification of the non-hyperbolic Cantor circles Julia sets, one needs to give the corresponding quasisymmetric models. This is what we have done in last three sections. In the rest of this section, we will prove that every non-hyperbolic Cantor circles Julia sets is quasisymmetrically equivalent to the Julia set of one map in the three families $P_{d_1,\cdots,d_n}$, $Q_{d_1,\cdots,d_n}$ and $R_{d_1,\cdots,d_n}$.

\begin{thm}\label{thm_parabolic-conj-std}
Let $f$ be a parabolic polynomial with degree $n\geq 2$ whose Julia set $J(f)$ is a Jordan curve. Then $J(f)$ contains infinitely many cusps and there exists a quasiconformal mapping $\phi:\EC\rightarrow\EC$ such that $\phi\circ f=g_n\circ \phi$ on $J(f)$, where $g_n(z)=(z^n+n-1)/n$.
\end{thm}

\begin{proof}
Without loss of generality, suppose that the unique parabolic fixed point of $f$ is located at the origin and the local expansion of $f$ at the origin is $f(z)=z+z^2+\cdots$ since $f$ has exactly one parabolic petal. This means that the parabolic fixed point $0$ is a cusp on $J(f)$. Moreover, all $k$-th preimages of $0$ are cusps, where $k\geq 1$. Therefore, $J(f)$ contains infinitely many cusps. By the assumption, the Julia set $J(f)$ divides the complex plane into two components. The bounded one, denoted by $U$, is the fixed parabolic Fatou component which contains $n-1$ critical points of $f$. According to the dynamical behaviors of $f$ in the parabolic basin $U$, there exists a simply connected domain $U_1\subset U$ satisfying the following conditions:
\begin{itemize}
\item[(1)] all the critical points of $f$ in $U$ are contained in $U_1$;
\item[(2)] $\partial U_1\cap\partial U=\{0\}$, $f(\overline{U}_1)\subset U_1\cup\{0\}$; and
\item[(3)] $\partial U_1$ is real analytic.
\end{itemize}
Let $U_2=f^{-1}(U_1)$. Then $U_2$ is simply connected whose boundary is also real analytic and $f:U_2\to U_1$ is a branched covering with degree $n$.

Let $g(z)=((1+z)^n-1)/n$ be the parabolic polynomial which is conjugated to $g_n$ by the translation $z\mapsto z+1$. Then $g$ has a parabolic fixed point at the origin with multiplier $1$. Let $W_1=\D(-3/4,3/4)$ be the round disk centered at $-3/4$ with radius $3/4$ and define $W_2=g^{-1}(W_1)$, Then $\overline{W}_1\subset W_2\cup\{0\}$ since $g(\overline{W}_1)\subset W_1\cup\{0\}$ by Lemma \ref{lema-g-n-g-mn}. Moreover, $W_2$ is a simply connected domain whose boundary is real analytic and $g:W_2\to W_1$ is a branched covering with degree $n$. By Riemann Mapping Theorem, there is a conformal map $\varphi:U_1\to W_1$ such that $\varphi(0)=0$. This map can be extended naturally to a homeomorphism $\varphi:\partial U_1\to \partial W_1$ because the boundaries $\partial U_1$ and $\partial W_1$ are both real analytic. Since $f:\partial U_2\to\partial U_1$ and $g:\partial W_2\to \partial W_1$ are both covering maps with degree $n$, there exists a unique lift of $\varphi:\partial U_1\to \partial W_1$, denoted also by $\varphi:\partial U_2\to \partial W_2$, which is a homeomorphism such that $\varphi(0)=0$ and $\varphi\circ f=g\circ\varphi$ holds on $\partial U_2$.

We now extend $\varphi|_{\partial U_1\cup\partial U_2}$ to a quasiconformal map $\varphi: U_2\setminus \overline{U}_1 \to W_2\setminus \overline{W}_1$ as follows. Note that $\overline{U}_2\setminus U_1$ and $\overline{W}_2\setminus W_1$ are degenerated annulus. Hence the extension is not trivial. Let $w=\psi(z)= -1/z$ be a coordinate transformation and define $\mathcal{C}_1=\psi(U_2\setminus \overline{U}_1)$ and $\mathcal{C}_2=\psi(W_2\setminus \overline{W}_1)$. Then $\mathcal{C}_1$ and $\mathcal{C}_2$ are two strips with infinite length. Let $\mathcal{C}_i^-$ and $\mathcal{C}_i^+$ be the boundaries of $\mathcal{C}_i$ which lie on the left and right respectively, where $i=1,2$. Then $\varphi|_{\partial U_1\cup\partial U_2}$ induces a map $\widehat{\varphi}=\psi\circ\varphi\circ\psi^{-1}:\mathcal{C}_1^\pm\to\mathcal{C}_2^\pm$ such that $\widehat{\varphi}\circ F=G\circ \widehat{\varphi}$, where $F=\psi\circ f\circ\psi^{-1}$ and $G=\psi\circ g\circ\psi^{-1}$. By parabolic linearization theorem, there exists a large $R>0$ and conformal maps $\alpha_i$ defined on $\{w\in\mathcal{C}_i^-:|\textup{Im}(w)|>R\}$ such that $\alpha_1(F(w))=\alpha_1(w)+1$ and $\alpha_2(G(w))=\alpha_2(w)+1$. Moreover, $\alpha_i(w)\sim w$ as $\textup{Im}(w)\to\infty$ (see \cite[\S 10]{Mil}).

Take $w_1\in\mathcal{C}_1^-$ and $\widehat{\varphi}(w_1)\in\mathcal{C}_2^-$ such that $\textup{Im}(w_1)>R$ and $\textup{Im}(\widehat{\varphi}(w_1))>R$. Note that the linearization maps $\alpha_1$ and $\alpha_2$ are unique up to an additive constant. We could normalize $\alpha_2$ such that $\alpha_2(\widehat{\varphi}(w_1))=\xi_1+ i\eta_1=\alpha_1(w_1)$, where $\xi_1,\eta_1\in\R$. Define $S_1=\{\zeta:\xi_1\leq\textup{Re}(\zeta)\leq \xi_1+1,\textup{Im}(\zeta)\geq \eta_1\}$ and let $S_1^-$, $S_1^+$ be the left and right edges of $S_1$ respectively. Then $\widehat{\varphi}:\mathcal{C}_1^-\to \mathcal{C}_2^-$ induces a map $\widetilde{\varphi}:S_1^-\to S_1^-$ which is defined by
\begin{equation*}
\widetilde{\varphi}:\xi_1+i\eta \mapsto\alpha_2\circ\widehat{\varphi}\circ\alpha_1^{-1}(\xi_1+i\eta)=\xi_1+i v(\eta),
\end{equation*}
where $\eta\geq \eta_1$ and $v:[\eta_1,+\infty)\to[\eta_1,+\infty)$ is a homeomorphism, which means that $\eta$ is continuous and increasing. On the other hand, $\widehat{\varphi}:\mathcal{C}_1^+\to \mathcal{C}_2^+$ induces a map $\widetilde{\varphi}:S_1^+\to S_1^+$ which is defined by
\begin{equation*}
\widetilde{\varphi}:\xi_1+1+i\eta \mapsto\alpha_2\circ\widehat{\varphi}\circ\alpha_1^{-1}(\xi_1+1+i\eta)=\xi_1+1+i v_1(\eta).
\end{equation*}
One can easily check that $v(\eta)=v_1(\eta)$ for $\eta\geq \eta_1$. Hence, the map $\widetilde{\varphi}:S_1^-\cup S_1^+\to S_1^-\cup S_1^+$ has an extension $\widetilde{\varphi}:S_1\to S_1$ such that $\widetilde{\varphi}(\xi+i\eta )=\xi+iv(\eta)$, where $\zeta=\xi+i\eta\in S_1$.

Since the boundaries $\partial U_1$ and $\partial W_1$ are both real-analytic, the conformal map $\varphi:U_1\to W_1$ can be extended a larger domain $U_1'$ such that $\overline{U}_1\subset U_1'$ by Schwarz's Reflection Principle. This means that $\varphi'(0)$ exists. Hence,  $\lim_{w\to\infty}\widehat{\varphi}'(w)=\varphi'(0)$ when $w\in \mathcal{C}_1^-\cup\mathcal{C}_1^+$. By the definition of $\widetilde{\varphi}$, we have $v'(\eta)\to |\varphi'(0)|$ as $\eta\to\infty$ since $\alpha_i(w)\sim w$ as $\textup{Im}(w)\to\infty$ for $i=1,2$. Note that $\varphi'(z)\neq 0$ in a small punctured neighborhood of $0$. This means that $v'(\eta)>0$ if $\eta$ is large enough. Without loss of generality, we suppose $v'(\eta)>0$ for $\eta\geq \eta_1$ since we can choose the number $R>0$ large enough. By a direct calculation, we have $\partial \widetilde{\varphi}/\partial\xi=1$ and $\partial \widetilde{\varphi}/\partial\eta=iv'(\eta)$, it follows that
\begin{equation*}
\frac{\partial \widetilde{\varphi}}{\partial\overline{\zeta}}/\frac{\partial \widetilde{\varphi}}{\partial\zeta}
=\frac{\partial \widetilde{\varphi}/\partial\xi+i\, \partial \widetilde{\varphi}/\partial\eta}{\partial \widetilde{\varphi}/\partial\xi - i\, \partial \widetilde{\varphi}/\partial\eta}=\frac{1-v'(\eta)}{1+v'(\eta)}<1.
\end{equation*}
This means that $\widetilde{\varphi}:S_1\to S_1$ is a quasiconformal mapping in the interior of $S_1$.

Let $\mathcal{A}_i=\alpha_i^{-1}(S_1)\subset \mathcal{C}_i$ for $i=1,2$. For $w\in\mathcal{A}_1$, define $\widehat{\varphi}(w)=\alpha_2^{-1}\circ\widetilde{\varphi}\circ\alpha_1(w)$. Then $\widehat{\varphi}:\mathcal{A}_1\to \mathcal{A}_2$ is a quasiconformal mapping defined from the upper end of $\mathcal{C}_1$ to the upper end of $\mathcal{C}_2$. Similarly, one can define a quasiconformal map $\widehat{\varphi}:\mathcal{B}_1\to \mathcal{B}_2$ from the lower end of $\mathcal{C}_1$ to that of $\mathcal{C}_2$. The map $\widehat{\varphi}:\mathcal{C}_1^\pm\to\mathcal{C}_2^\pm$ can be extended to the middle part $\mathcal{C}_1\setminus(\mathcal{A}_1\cup \mathcal{B}_1)$ of $\mathcal{C}_1$ as a quasiconformal map $\widehat{\varphi}:\mathcal{C}_1\setminus(\mathcal{A}_1\cup \mathcal{B}_1)\to\mathcal{C}_2\setminus(\mathcal{A}_2\cup \mathcal{B}_2)$ is well known (since $\mathcal{C}_1^\pm$ and $\mathcal{C}_2^\pm$ are both real-analytic). Now we obtain a quasiconformal map $\widehat{\varphi}:\mathcal{C}_1\to \mathcal{C}_2$ which is the extension of $\widehat{\varphi}:\mathcal{C}_1^\pm\to\mathcal{C}_2^\pm$. Let $\varphi:=\psi\circ\widetilde{\varphi}\circ\psi^{-1}$. Then $\varphi: U_2\setminus \overline{U}_1 \to W_2\setminus \overline{W}_1$ is a quasiconformal extension of $\varphi|_{\partial U_1\cup\partial U_2}$.

Now we define a quasi-regular map as follows:
\[                
\widetilde{f}(z)=
\left\{                         
\begin{array}{ll}               
f(z)  &~~~~~~\text{if}~z\in \C\setminus U_2 \\             
\varphi^{-1}\circ g \circ \varphi(z) &~~~~~\text{if}~z\in U_2.    
\end{array}                     
\right.                         
\]
Let $\sigma_0$ be the standard complex structure in $\C$ and $\sigma$ the unique $\widetilde{f}$-invariant almost complex structure on $\C$ satisfying
\[                
\sigma(z)=
\left\{                         
\begin{array}{ll}               
\varphi^*\sigma_0  &~~~~~~~z\in U_2\setminus \overline{U}_1 \\             
\sigma_0 &~~~~~~z\in (\C \setminus \overline{U})\cup \overline{U}_1.    
\end{array}                     
\right.                         
\]

Let $\phi : (\C,\sigma) \rightarrow (\C,\sigma_0)$ be the integrating orientation preserving quasiconformal homeomorphism which is normalized by $\phi(0)=1$ and $\phi(\varphi^{-1}(-1))=0$. Then $\phi\circ \widetilde{f}\circ\phi^{-1}$ is an unicritical parabolic polynomial with parabolic fixed point at $1$ and critical points at the origin with multiplicity $n-1$. This means that $\phi\circ \widetilde{f}\circ\phi^{-1}=g_n$. By the definition of $\widetilde{f}$, we know that $\phi\circ f=g_n\circ \phi$ holds on $J(f)$. The proof is complete.
\end{proof}

Note that the Julia sets $J(f)$ and $J(g_n)$ are not quasicircles since they contain infinitely many cusps.
Actually, Theorem \ref{thm_parabolic-conj-std} is well known for experts (compare \cite[\S 6]{McM}).

\begin{cor}\label{cor_parabolic-conj-std}
Let $f_1,f_2$ be two parabolic polynomials with the same degree whose Julia set are both Jordan curves. Then there exists a quasiconformal map $\phi:\EC\rightarrow\EC$ such that $\phi\circ f_1=f_2\circ \phi$ on $J(f_1)$.
\end{cor}

Let $A$ be an annulus whose closure separates $0$ and $\infty$, we use $\partial_-A$ and $\partial_+A$ to denote the two components of the boundary of $A$ such that $\partial_-A\prec \partial_+A$.

\begin{thm}\label{this-is-all-resta}
Let $f$ be a non-hyperbolic rational map whose Julia set is a Cantor set of circles. Then there exist $n\geq 2$ positive integers $d_1,\cdots,d_n$ satisfying $\sum_{i=1}^{n}(1/d_i)<1$ such that $f$ is conjugate to $P_{d_1,\cdots,d_n}$, $Q_{d_1,\cdots,d_n}$ or $R_{d_1,\cdots,d_n}$ for suitable parameters on their corresponding Julia sets by a quasiconformal mapping.
\end{thm}

\begin{proof}
Let $J(f)$ be the Julia set of $f$ which is a Cantor set of circles. Assume that the two simply connected Fatou components of $f$, which are denoted by $D_0$ and $D_\infty$, contain two points $0$ and $\infty$, respectively. Let $\mathcal{D}=\{D_0,D_\infty\}$ and $\mathcal{A}$ be the collection of all simply and doubly connected Fatou components of $f$ respectively. Then $f(\mathcal{D})\subset\mathcal{D}$ and there exists an integer $k\geq 1$ such that $f^{\circ k}(A)\in\mathcal{D}$ for every $A\in \mathcal{A}$ by Riemann-Hurwitz's formula and Sullivan's eventually periodic theorem.

Note that every periodic Fatou component of $f$ must be attracting or parabolic because the Cantor circles Julia sets cannot contain any critical points. Since $f(\mathcal{D})\subset\mathcal{D}$, without loss of generality, we assume that $f(D_0)=D_\infty$, $f(D_\infty)=D_\infty$ and $D_\infty$ is a parabolic Fatou component since the rest three cases can be proved completely similarly:

(1) $f(D_0)=D_0$ and $f(D_\infty)=D_\infty$, where $D_0$ is an attracting Fatou component and $D_\infty$ is parabolic (The potential quasisymmetric model is $P_{d_1,\cdots,d_n}$ with odd $n\geq 3$);

(2) $f(D_0)=D_0$ and $f(D_\infty)=D_\infty$, where $D_0$ and $D_\infty$ are both parabolic Fatou components (The potential quasisymmetric model is $Q_{d_1,\cdots,d_n}$ with odd $n\geq 3$); and

(3) $f(D_0)=D_\infty$ and $f(D_\infty)=D_0$, where $D_0$ and $D_\infty$ are parabolic periodic Fatou components with period $2$ (The potential quasisymmetric model is $R_{d_1,\cdots,d_n}$ with odd $n\geq 3$).

Our aim is to prove that the dynamics on the Julia set of $f$ has a quasisymmetric model $P_{d_1,\cdots,d_n}$, where $n\geq 2$ is even.
The basic idea of the proof is as follows: Let $D_0'$ and $D_\infty'$ be the Fatou components of $P_{d_1,\cdots,d_n}$ containing the origin and infinity respectively, where $n\geq 2$ is even. We construct a sequence of quasiconformal maps $\phi_k:\EC\to\EC$ with uniform dilatation which conjugates $f$ to $P_{d_1,\cdots,d_n}$ on $f^{-k}(\partial D_0\cup\partial D_\infty)$. Then the limit $\phi_\infty:\EC\to\EC$ of $\phi_k:\EC\to\EC$ is also a quasiconformal map which conjugates $f$ to $P_{d_1,\cdots,d_n}$ on $\bigcup_{k\geq 0}f^{-k}(\partial D_0\cup\partial D_\infty)$, where $k\in\mathbb{N}$. By the continuity, $\phi_\infty:\EC\to\EC$ conjugates $f$ to $P_{d_1,\cdots,d_n}$ on $J(f)=\overline{\bigcup_{k\geq 0}f^{-k}(\partial D_0\cup\partial D_\infty)}$.

Now we give the detailed proof. Since $f(D_0)=D_\infty$, $f(D_\infty)=D_\infty$ and the Julia set $J(f)$ is a Cantor set of circles, then $f^{-1}(D_0)$ consists of finite many Fatou components which are all annular domains. In particular, we can write $f^{-1}(D_0)=\bigcup_{i=1}^m A_{2i-1}$, where $A_1,A_3,\cdots, A_{2m-1}$ are $m$ annular Fatou components separating $0$ from $\infty$ such that $A_{2i+1}\prec A_{2i-1}$ for every $1\leq i\leq m-1$. Suppose that $\deg (f:\partial_+A_{2i-1}\rightarrow \partial D_0)=d_{2i-1}$ and $\deg (f:\partial_-A_{2i-1}\rightarrow \partial D_0)=d_{2i}$ for $1\leq i\leq m$. Then, we have $\deg(f)=\sum_{j=1}^{2m}d_j$ (see Figure \ref{Fig_find-conj} for $m=3$).

\begin{figure}[!htpb]
  \setlength{\unitlength}{1mm}
  \centering
  \includegraphics[width=130mm]{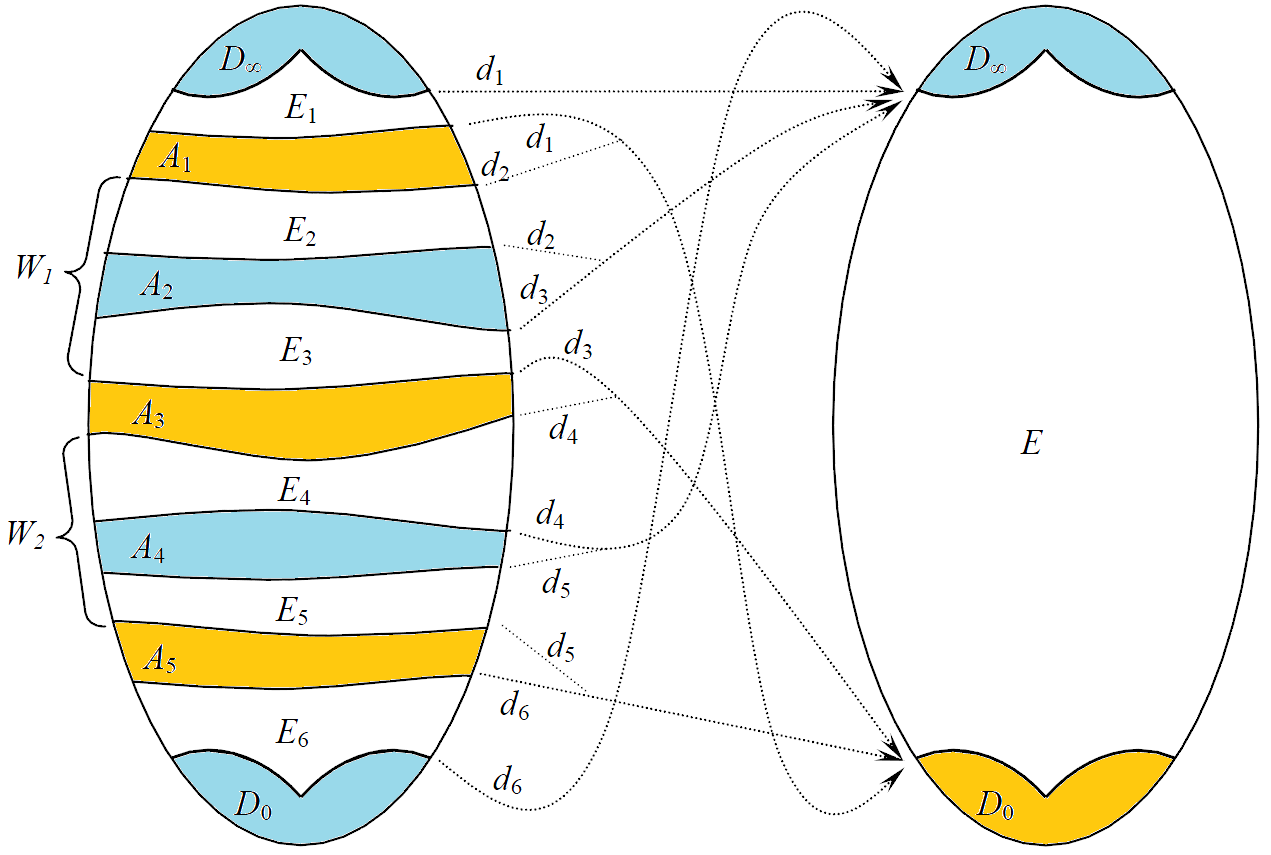}
  \caption{Sketch illustrating of the mapping relation of $f$. The numbers $d_1,d_2,\cdots, d_{2m}$ denote the degrees of the restriction of $f$ on the boundaries of Fatou components, where $m=3$.}
  \label{Fig_find-conj}
\end{figure}

Let $W_i$ be the annular domain between $\overline{A}_{2i-1}$ and $\overline{A}_{2i+1}$, where $1\leq i\leq m-1$. Then we have $f(W_i)=\EC\setminus \overline{D}_0$ and $\deg(f:W_i\rightarrow\EC\setminus \overline{D}_0)=d_{2i}+d_{2i+1}$. This means that there exists at least one Fatou component $A_{2i}\subsetneq W_i$ such that $f(A_{2i})=D_\infty$. Suppose that there exists another Fatou component (which is an annulus) $A_{2i}'\neq A_{2i}$ such that $A_{2i}'\subsetneq W_i$ and $f(A_{2i}')=D_\infty$. Then there must exist a component of $f^{-1}(D_0)$ contained in the region between $A_{2i}'$ and $A_{2i}$, which contradicts the assumption that $f^{-1}(D_0)=\bigcup_{i=1}^m A_{2i-1}$ is disjoint with $W_i$. Hence, there exists exactly one Fatou component $A_{2i}\subsetneq W_i$ such that $f(A_{2i})=D_\infty$ and $\deg(f:A_{2i}\rightarrow D_\infty)=d_{2i}+d_{2i+1}$. Similar argument shows that $D_0$ is the unique component of $f^{-1}(D_\infty)$ lying in the bounded component of $\EC\setminus \overline{A}_{2m-1}$, and $D_\infty$ is the unique component of $f^{-1}(D_\infty)$ lying in the unbounded component of $\EC\setminus \overline{A}_1$. Therefore, we have $f^{-1}(D_\infty)=D_0\cup D_\infty\cup\bigcup_{i=1}^{m-1}A_{2i}$. Moreover, $\deg (f:D_0\rightarrow D_\infty)=d_{2m}$ and $\deg (f:D_\infty\rightarrow D_\infty)=d_1$.

Define the closed annulus $E=\overline{\mathbb{C}}\setminus ({D_0\cup D_\infty})$. Then $f^{-1}(E)=\bigcup_{i=1}^{2m}E_i$, where $E_1$, $E_2$, $\cdots$, $E_{2m}$ are $2m$ annuli separating $0$ from $\infty$ such that $E_j\prec E_i$ for $1\leq i<j\leq 2m$. The map $f:E_i\rightarrow E$ is an unramified covering map with degree $d_i$, where $1\leq i\leq 2m$.

Define $n=2m$. By Gr\"{o}tzsch's modulus inequality, we have $\sum_{i=1}^{n}(1/d_i)<1$ since each $E_i$ is essentially contained in $E$ and $\text{mod} (E_i)=\text{mod} (E)/d_i$. In the following, we will construct a quasiconformal mapping $\phi:\EC\to\EC$ which conjugates the dynamics on the Julia set of $f$ to that of $P_{d_1,\cdots,d_n}$ with even $n$. For simplicity, we use $P$ to denote $P_{d_1,\cdots,d_n}$.

By Lemma \ref{lemma-want} and Theorem \ref{parameter-parabolic-resta}, there exists a simply connected parabolic Fatou component $D_\infty'\ni\infty$ of $P$ which satisfies $P^{-1}(D_\infty')=D_0'\cup D_\infty'$, where $D_0'$ is the Fatou component of $P$ containing the origin. For $1\leq i\leq n-1$, let $A_i'$ be the annular Fatou component of $P$ such that $A_i\subset A_i'$, where $A_i$ is the annulus defined in Lemma \ref{lemma-want}. By Lemma \ref{lemma-want} and the proof of Theorem \ref{parameter-parabolic-resta}, we have $P^{-1}(D_\infty')=D_0'\cup D_\infty'\cup\bigcup_{i=1}^{m-1} A_{2i}'$ and $P^{-1}(D_0')=\bigcup_{i=1}^{m}A_{2i-1}'$. Moreover, $\deg (P:D_\infty'\rightarrow D_\infty')=d_1$, $\deg (P:D_0'\rightarrow D_\infty')=d_{2m}$, $\deg(P:A_{2i-1}'\rightarrow D_0')=d_{2i-1}+d_{2i}$ and $\deg(P:A_{2i}'\rightarrow D_\infty')=d_{2i}+d_{2i+1}$.

Define the closed annulus $E':=\overline{\mathbb{C}}\setminus ({D_0'\cup D_\infty'})$. Then $P^{-1}(E')=\bigcup_{i=1}^{2m}E_i'$, where $E_1'$, $E_2'$, $\cdots$, $E_{2m}'$ are $2m$ annuli separating $0$ from $\infty$ such that $E_j'\prec E_i'$ for $1\leq i<j\leq 2m$. The map $P:E_i'\rightarrow E'$ is an unramified covering map with degree $d_i$, where $1\leq i\leq 2m$. Up to now, one can see that the rational maps $f$ and $P$ `almost' have the same dynamical behaviours. We now begin to define a sequence of quasiconformal mappings.

Let $U$ and $V$ be the interior of $E_1\cup D_\infty$ and $E\cup D_\infty$ respectively. Then $f:U\to V$ is a polynomial-like map with degree $d_1$. By Douady and Hubbard's Straighten Theorem \cite[p.\,296]{DH}, $f:U\to V$ is quasiconformally conjugate to a parabolic polynomial $\widehat{f}$ with degree $d_1$ in a neighborhood of the filled-in Julia set of $f:U\to V$. On the other hand, it is known that $P:(E_1'\cup D_\infty')^o\to (E'\cup D_\infty')^o$ is a polynomial-like map which is quasiconformally conjugate to the parabolic polynomial $g_{d_1}=(z^{d_1}+d_1-1)/d_1$. By Corollary \ref{cor_parabolic-conj-std}, there exists a quasiconformal mapping $\phi_0:\overline{\mathbb{C}}\rightarrow\overline{\mathbb{C}}$ such that $\phi_0(D_\infty)=D_\infty'$ and $\phi_0\circ f=P\circ\phi_0$ on $\partial D_\infty$. Since $f:D_0\to D_\infty$ and $P:D_0'\to D_\infty'$ are both holomorphic maps with degree $d_n$, the quasiconformal mapping $\phi_0:\EC\to\EC$ can be chosen such that it also satisfies $\phi_0\circ f=P\circ\phi_0$ on $\partial D_0$ and $\phi_0(D_0)=D_0'$. Hence the quasiconformal map $\phi_0:\EC\to\EC$ satisfies:
\begin{itemize}
\item[(1)] $\phi_0(D_0)=D_0'$ and $\phi_0(D_\infty)=D_\infty'$; and
\item[(2)] $\phi_0\circ f=P\circ\phi_0$ on $\partial D_0\cup\partial D_\infty$.
\end{itemize}

Next we define $\phi_1:\EC\to\EC$. Note that $f:E_1\rightarrow E$ and $P:E_1'\rightarrow E'$ are both covering maps with degree $d_1$. There exists a unique lift $\phi_{E_1}:E_1\rightarrow E_1'$ of $\phi_0:E\rightarrow E'$ such that $\phi_{E_1}=\phi_0$ on $\partial D_\infty$ and $\phi_0\circ f =P\circ \phi_{E_1}$ on $E_1$. In particular, $\phi_{E_1}:E_1\rightarrow E_1'$ is quasiconformal in the interior of $E_1$ since $\phi_0:E\rightarrow E'$ is quasiconformal and $f,P$ are both holomorphic. Similarly, there exists a unique quasiconformal mapping $\phi_{E_{2m}}:E_{2m}\rightarrow E_{2m}'$, which is the lift of $\phi_0:E\rightarrow E'$ such that $\phi_{E_{2m}}=\phi_0$ on $\partial D_0$ and $\phi_0\circ f =P\circ \phi_{E_{2m}}$ on $E_{2m}$. For $2\leq i\leq 2m-1$, there exists a lift (but not unique) of $\phi_0:E\rightarrow E'$ defined from $E_i$ to $E_i'$, which we denote by $\phi_{E_i}$ such that $\phi_0\circ f =P\circ \phi_{E_i}$ on $E_i$. For $1\leq i\leq 2m$, define $\phi_1=\phi_{E_i}$ on $E_i$ and define $\phi_1=\phi_0$ on $\overline{D}_0\cup\overline{D}_\infty$. Then we have $\phi_0\circ f =P\circ \phi_1$ on $\bigcup_{i=1}^{2m}E_i$ and $\phi_1\circ f =P\circ \phi_1$ on $\bigcup_{i=1}^{2m}\partial E_i$.

We need to define $\phi_1$ on $\bigcup_{j=1}^{2m-1}A_j$ (see Figure \ref{Fig_find-conj}). Note that for each $A_j$, its two boundary components $\partial_+ E_{i+1}$ and $\partial_- E_i$ are not quasicircles since $\partial D_\infty$ is not. Hence, we cannot take a quasiconformal extension directly. However, the lift $\phi_{E_i}:E_i\to E_i'$ defined in the last paragraph can be actually defined in a larger domain $U_i$, which is an open annular neighborhood  of $E_i$, such that the both boundary components of $U_i$ are quasicircles and they satisfy $U_{2m}\prec U_{2m-1}\prec\cdots\prec U_1$. For $1\leq i\leq 2m$, define $U_i'=\phi_{E_i}(U_i)$. For $1\leq j\leq 2m-1$, let $V_j\subset A_j$ and $V_j'\subset A_j'$, respectively, be the annular domains between $U_j$ and $U_{j+1}$, $U_j'$ and $U_{j+1}'$, respectively. It follows that the map $\phi_1|_{\partial_+ U_{j+1}\cup\partial_- U_j}$ has an extension $\phi_{V_j}:\overline{V_j}\rightarrow \overline{V_j'}$ such that $\phi_{V_j}:V_j\rightarrow V_j'$ is quasiconformal since $\phi_{E_j}|_{U_j}$ and $\phi_{E_{j+1}}|_{U_{j+1}}$ are both quasiconformal mappings and  $\partial_+ U_{j+1}$ and $\partial_- U_j$ are quasicircles. Define $\phi_1=\phi_{V_j}$ on $V_j$ for $1\leq j\leq 2m-1$. Then we obtain a quasiconformal mapping $\phi_1:\overline{\mathbb{C}}\rightarrow\overline{\mathbb{C}}$ which satisfies the following conditions:
\begin{itemize}
\item[(1)] $\phi_1=\phi_0$ on $D_0\cup D_\infty$;
\item[(2)] $\phi_0\circ f =P\circ \phi_1$ on $\bigcup_{i=1}^{2m}E_i$; and
\item[(3)] $\phi_1\circ f =P\circ \phi_1$ on $f^{-1}(\partial D_0\cup \partial D_\infty)$.
\end{itemize}

Now we define $\phi_2$. For $1\leq j\leq 2m-1$, define $\phi_2=\phi_1$ on $A_j$ and $\phi_2=\phi_1=\phi_0$ on $D_0\cup D_\infty$. For $1\leq i\leq 2m$, we lift $\phi_1:E\rightarrow E'$ to obtain $\phi_2:E_i\rightarrow E_i'$ such that $\phi_1\circ f=P\circ \phi_2$ on $E_2$. Now $\phi_2:\EC\to\EC$ has been defined but we need to check its continuity. There exists only one way to lift $\phi_1:E\rightarrow E'$ to obtain $\phi_2:E_1\rightarrow E_1'$ if one wants to guarantee the continuity of $\phi_2$ on $D_\infty\cup E_1$ since $\phi_2$ is required to equal to $\phi_1$ on $\partial_+E_1$. Note that $\phi_2=\phi_1$ on $A_1$, we need to check the continuity of $\phi_2$ at the boundary $\partial_- E_1$. In fact, $\phi_0|_E$ and $\phi_1|_E$ are homotopic to each other and $\phi_1|_{\partial E}=\phi_0|_{\partial E}$, it follows that $\phi_2=\phi_1$ on $\partial_- E_1$ since $\phi_2=\phi_1$ on $\partial_+ E_1$. This means that $\phi_2$ is continuous on $\partial_- E_1$. By taking a similar way, one can lift $\phi_1:E\rightarrow E'$ to obtain $\phi_2:E_i\rightarrow E_i'$ for $2\leq i\leq 2m$ and guarantee the continuity of $\phi_2$ on the corresponding boundaries.
Since $f$ and $P$ are both holomorphic, the dilatation of $K(\phi_2:E_i\rightarrow E_i')\leq K(\phi_1)$ for $1\leq i\leq 2m$. Note that $\phi_2=\phi_1$ on $A_j$ and $\phi_2=\phi_1=\phi_0$ on $D_0\cup D_\infty$, the map $\phi_2:\overline{\mathbb{C}}\rightarrow\overline{\mathbb{C}}$ has following properties:
\begin{itemize}
\item[(1)] $\phi_2=\phi_1$ on $f^{-1}(D_0\cup D_\infty)$;
\item[(2)] $\phi_1\circ f=P\circ\phi_2$ on $\bigcup_{i=1}^{2m}E_i$;
\item[(3)] $\phi_2\circ f=P\circ\phi_2$ on $f^{-2}(\partial D_0\cup \partial D_\infty)$; and
\item[(4)] The dilatation of $\phi_2$ satisfies $K(\phi_2)=K(\phi_1)$.
\end{itemize}

We now use the method of induction to prove the theorem. If we have obtained a quasiconformal mapping $\phi_k:\EC\rightarrow \EC$ for some $k\geq 1$ as above, then $\phi_{k+1}:\EC\rightarrow \EC$ can be defined completely similarly to the procedure of the derivation of $\phi_2$ from $\phi_1$. Inductively, we can obtain a sequence of quasiconformal mappings $\{\phi_k\}_{k\geq 0}$ with following properties:
\begin{itemize}
\item[(1)] $\phi_{k+1}(z)=\phi_{k}(z)$ for $z\in f^{-k}(D_0\cup D_\infty)$;
\item[(2)] $\phi_{k}\circ f=P\circ\phi_{k+1}$ on $\bigcup_{i=1}^{2m}E_i$;
\item[(3)] $\phi_{k+1}\circ f=P\circ\phi_{k+1}$ on $f^{-(k+1)}(\partial D_0\cup \partial D_\infty)$; and
\item[(4)] The dilatation satisfies $K(\phi_{k+1})=K(\phi_k)=K(\phi_1)$ for $k\geq 1$.
\end{itemize}

By \cite[Theorem 5.1]{LV}, the sequence $\{\phi_k:\EC\to\EC\}_{k\geq 0}$ is a normal family. Let $\phi_\infty:\EC\to\EC$ be the limit of any convergent subsequence in $\{\phi_k\}_{k\geq 0}$. Then $\phi_\infty$ is a quasiconformal mapping satisfying $\phi_\infty\circ f=P\circ\phi_\infty$ on $\bigcup_{k\geq 0} f^{-k}(\partial D_0\cup \partial D_\infty)=\bigcup_{k\geq 0}f^{-k}(\partial D_\infty)$. In particular, $\phi_\infty\circ f=P\circ\phi_\infty$ holds on $\overline{\bigcup_{k\geq 0}f^{-k}(\partial D_\infty)}$ since $\phi_\infty$ is continuous. Note that the Julia set $J(f)=\overline{\bigcup_{k\geq 0}f^{-k}(\partial D_\infty)}$. Therefore, $\phi_\infty:\EC\to\EC$ is a quasiconformal mapping which conjugates $f$ to $P$ on their corresponding Julia sets. The proof is complete.
\end{proof}

For each hyperbolic rational map whose Julia set is a Cantor set of circles, Theorem \ref{this-is-all-hyper} gives the corresponding quasisymmetric model of the Julia set. For the non-hyperbolic case, Theorem \ref{this-is-all-resta} gives the models.

\vskip 0.2cm
\noindent{\textit{Proof of Theorem \ref{this-is-all-intro}.}}
Combine Theorems \ref{this-is-all-hyper} and \ref{this-is-all-resta}.
\hfill $\square$

\begin{cor}\label{coro-cusp-CC}
Each parabolic Cantor circles Julia set has infinitely many cusps.
\end{cor}

\begin{proof}
For $P_{d_1,\cdots,d_n}$, $Q_{d_1,\cdots,d_n}$ or $R_{d_1,\cdots,d_n}$, one can easily get a polynomial-like map which is quasiconformally conjugate to a parabolic polynomial. Specifically, as in the proof in Theorem \ref{parameter-parabolic-resta}, $P_{d_1,\cdots,d_n}:(V_1\cup\widetilde{U}_\infty)^o\to(V\cup\widetilde{U}_\infty)^o$ is a polynomial-like map with degree $d_1$ which is quasiconformally conjugated to the parabolic polynomial $g_{d_1}=(z^{d_1}+d_1-1)/d_1$ (see Figure \ref{fig_P_n}). For $Q_{d_1,\cdots,d_n}$ and $R_{d_1,\cdots,d_n}$, one can refer Figure \ref{fig_parabolic_III-IV} to get two polynomial-like maps $Q_{d_1,\cdots,d_n}:(E_1\cup D_\infty)^o\to(E\cup D_\infty)^o$ and $R_{d_1,\cdots,d_n}^{\,\circ 2}:(\widehat{E}_1\cup D_\infty)^o\to(E\cup D_\infty)^o$ which are quasiconformally conjugated to the parabolic polynomials $g_{d_1}$ and $g_{d_1,d_n}$, where $\widehat{E}_1=E_1\cap R_{d_1,\cdots,d_n}^{-1}(E_3)$ and $g_{d_1,d_n}$ is defined in \eqref{g_m-g-mn}. By Theorem \ref{thm_parabolic-conj-std}, it follows that the boundary of the unbounded Fatou component of $P_{d_1,\cdots,d_n}$, $Q_{d_1,\cdots,d_n}$ and $R_{d_1,\cdots,d_n}$ contains infinitely many cusps. Then each parabolic Cantor circles Julia set has infinitely many cusps by Theorem \ref{this-is-all-resta}.
\end{proof}

\begin{rmk}
The hyperbolic Cantor circles Julia sets (three types) are not quasisymmetrically equivalent to the parabolic Cantor circles Julia sets (four types) since the former does not have cusps but the latter have (see Table \ref{Tab-classif}). However, we don't know whether there exist two different types hyperbolic Cantor circles Julia sets which are quasisymmetrically equivalent to each other or not.
\end{rmk}

\section{Regularity of the components of parabolic Cantor circles}\label{sec-regularity}

By definition (see  for example, \cite[p.\,100]{LV}), a Jordan curve $\gamma\subset \EC$ is called a \textit{quasicircle} if there exists a positive constant $C$ such that for any different points $x,y\in\gamma$, the \textit{turning} of $I$ about $x$ and $y$ satisfies
\begin{equation*}
\diam(I)/|x-y|\leq C,
\end{equation*}
where $I$ is the component of $\gamma \setminus\{x,y\}$ with smaller diameter.

If $f$ is a hyperbolic rational map whose Julia set is a Cantor set of circles, one can prove that all Julia components of $f$ are quasicircles by using the structure of the Cantor circles and a standard argument (see Remark after the proof of Theorem \ref{thm-regularity}). However, if $f$ is a parabolic rational map whose Julia set is a Cantor set of circles,  it needs to do more work if one wants to study the regularity of the Julia components.

Obviously, for parabolic Cantor circles, there are at least countably many Julia components which are not quasicircles since they contain infinitely many cusps (Corollary \ref{coro-cusp-CC}). The boundaries of the grand orbit of the parabolic periodic Fatou component are examples. On the other hand, surprisingly, there exist also many Julia components of parabolic Cantor circles which are quasicircles. We give a complete characterization for this dichotomy in the following theorem.

\begin{thm}\label{thm-regularity}
Let $f$ be a parabolic rational map whose Julia set is a Cantor set of circles and $J_0$ a Julia component of $f$. Then, $J_0$ is a quasicircle if and only if the closure of the forward orbit of $J_0$ is disjoint with the boundaries of the parabolic periodic Fatou components.
\end{thm}

\begin{proof}
Without loss of generality, we assume that the two points $0$ and $\infty$ are contained in $D_0$ and $D_\infty$ respectively, where $D_0$ and $D_\infty$ are the two simply connected Fatou components of $f$. Similar to the consideration in Theorem \ref{this-is-all-resta}, we assume that $f(D_0)=D_\infty$ and $f(D_\infty)=D_\infty$, where $D_\infty$ is a parabolic Fatou component. The other three cases stated in the proof of Theorem \ref{this-is-all-resta} can be proved completely similarly.

Let $\mathcal{J}$ be the closure of the union of the forward orbit $\{J_k:=f^{\circ k}(J_0)\}_{k\in\mathbb{N}}$ of the Julia component $J_0$.  We claim that $\mathcal{J}\subset J(f)$ is homeomorphic to the product of a closed subset of the standard middle third Cantor set $\mathcal{C}$ and the unit circle $\mathbb{S}^1$. Then each component of $\mathcal{J}$ is a Jordan curve. In fact, since $J(f)$ is a Cantor set of circles, there exists a homeomorphism $\phi:J(f)\to \mathcal{C}\times \mathbb{S}^1$ which maps the Julia set of $f$ to the standard Cantor set of circles. Let $z\in J(f)$ be a point which is accumulated by $\{J_k\}_{k\in\mathbb{N}}$. Then equivalently, $\phi(z)$ is an accumulation point of $\{\phi(J_k)\}_{k\in\mathbb{N}}$. Note that all components of $\mathcal{C}\times \mathbb{S}^1$ are round circles. Therefore, each $w'\in \phi(J_z)$ is accumulated by  $\{\phi(J_k)\}_{k\in\mathbb{N}}$, where $J_z$ is the Julia component containing $z$. This means that $\phi(\mathcal{J})$ is equal to the product  of a closed subset of $\mathcal{C}$ and $\mathbb{S}^1$. This ends the proof the claim.

If $\mathcal{J}$ is disjoint with the boundary $\partial D_\infty$, there exists a constant $\delta>0$ such that $\{z\in\C: \dist(z,\mathcal{J})\leq \delta\}$ is disjoint with the the closure of the critical orbit of $f$ since the image of the forward orbit of the critical points of $f$ is contained in $D_0\cup D_\infty$ and $\dist(\overline{D}_0\cup \overline{D}_\infty,\mathcal{J})>0$. Since $J(f)$ is a Cantor set of circles, we have $\inf_{k\in\mathbb{N}}\diam(J_k)>0$. We assume that $\delta>0$ is sufficiently small such that $\{z\in\C: \dist(z,\mathcal{J})\leq \delta\}$ is disjoint with the poles of $f$ and $\inf_{k\in\mathbb{N}}\diam(J_k)>\delta/2$.

For any two given points $x,y\in J_k=f^{\circ k}(J_0)$, let $L(x,y)$ be the component of $J_k\setminus\{x,y\}$ with the smaller diameter, where $k\in\mathbb{N}$. Let $\phi:J(f)\to \mathcal{C}\times \mathbb{S}^1$ be the homeomorphism defined above. Then $\phi$ is uniformly continuous since $J(f)$ is compact. This means that for each positive $\epsilon\ll \inf_{k\in\mathbb{N}}\diam(J_k)$, there exists a positive constant $\beta(\epsilon)<\epsilon$ depending only on $\epsilon$ such that for any $x,y\in \mathcal{J}$,
\begin{equation}\label{est-qwy}
\text{if~} |x-y|<\beta(\epsilon), \text{~then~} \diam (L(x,y))<\epsilon.
\end{equation}

Let $z_1,z_2\in J_0$ be any two different points. If $|z_1-z_2|\geq \beta(\delta/2)$, then
\begin{equation}\label{qc-case-1}
\frac{\diam(L(z_1,z_2))} {|z_1-z_2|} \leq \frac{M_1}{\beta(\delta/2)},
\end{equation}
where $M_1:=\sup_{k\in\mathbb{N}}\diam(J_k)<+\infty$. Otherwise, $|z_1-z_2|< \beta(\delta/2)$. Then $L(z_1,z_2)<\delta/2$ by \eqref{est-qwy}. Note that there exists a sufficiently large $N>0$ such that $f^{\circ m}(L(z_1,z_2))=J_m$ if $m\geq N$. There exists a smallest $k\geq 0$ such that
\begin{equation}\label{qwy-1}
\diam(f^{\circ k}(L(z_1,z_2)))<\delta/2 \text{~~and~~} \diam(f^{\circ k+1}(L(z_1,z_2)))\geq\delta/2.
\end{equation}
In particular, $f^{\circ k}:L(z_1,z_2)\to f^{\circ k}(L(z_1,z_2))\subset J_{k}$ is injective. Define $M_2:=\max\{|f'(z)|: \dist(z,\mathcal{J})\leq \delta/2\}$. There exist two points $w_1,w_2\in f^{\circ k}(L(z_1,z_2))\subset J_k$ such that
\begin{equation}\label{qwy-2}
\diam(f^{\circ k+1}(L(z_1,z_2)))= |f(w_1)-f(w_2)|\leq  M_2|w_1-w_2|\leq M_2\,\diam(f^{\circ k}(L(z_1,z_2))).
\end{equation}
By \eqref{qwy-1} and \eqref{qwy-2}, we have
\begin{equation*}
\delta/(2M_2)\leq  \diam(f^{\circ k}(L(z_1,z_2))) = \diam(L(f^{\circ k}(z_1),f^{\circ k}(z_2)))<\delta/2.
\end{equation*}
By \eqref{est-qwy}, there exists a constant $C_1(\delta,M_2)>0$ such that
\begin{equation*}
|f^{\circ k}(z_1)-f^{\circ k}(z_2)|\geq C_1(\delta,M_2).
\end{equation*}

Let $U$ and $V$ be the components of $f^{-k}(\D(f^{\circ k}(z_1),\delta/2))$ and $f^{-k}(\D(f^{\circ k}(z_1),\delta))$ both containing $z_1$, respectively. Consider the conformal mapping $f^{\circ k}:V\rightarrow \D(f^{\circ k}(z_1),\delta)$ and by Koebe's distortion theorem, we know that there is a constant $C_2>0$ such that
\begin{equation}\label{qc-case-2}
\frac{\diam(L(z_1,z_2))} {|z_1-z_2|} \leq C_2\frac{\diam(f^{\circ k}(L(z_1,z_2)))}{|f^{\circ k}(z_1)-f^{\circ k}(z_2)|}
\leq \frac{C_2\delta}{2C_1(\delta,M_2)}.
\end{equation}
Combining \eqref{qc-case-1} and \eqref{qc-case-2}, we know that $J_0$ is a quasicircle since it is bounded turning.

On the other hand, suppose the forward orbit of $J_0$ accumulates on the boundary $\partial D_\infty$. The idea of the proof is similar to that of \cite[Theorem 8.6]{PT}. There exists a subsequence $J_{n_k}\to \partial D_\infty$ in the Hausdorff topology as $k\to\infty$. Note that $\partial D_\infty$ is not a quasicircle since it contains infinitely many cusps (see Corollary \ref{coro-cusp-CC}). Choose one of the cusps, which is not the parabolic fixed point, and denote it by $\zeta$. There exists a sequence of small constants $\{\rho_m\}_{m\in\mathbb{N}}>0$ satisfying $1\gg \rho_0>\rho_1>\cdots>\rho_m>\cdots>0$ and $\lim_{m\to\infty}\rho_m=0$ such that the round disk $\D(\zeta,\rho_0)$ is disjoint with the critical orbit of $f$. Without loss of generality, for each $m\in\mathbb{N}$, assume that $J_{n_k}\cap\partial\D(\zeta,\rho_m/2)=\{x_{n_k}^{m},y_{n_k}^{m}\}$ and $L(x_{n_k}^m,y_{n_k}^m)\subset\overline{\D}(\zeta,\rho_m/2)$ for all $k\geq k_m$, where $L(x_{n_k}^m,y_{n_k}^m)$ is the component of $J_{n_k}\setminus\{x_{n_k}^m,y_{n_k}^m\}$ with smaller diameter and $k_m\geq 0$ is an integer depending only on $m$. Moreover, assume that $\partial D_\infty\cap\partial\D(\zeta,\rho_m/2)=\{x^m,y^m\}$ and $L(x^m,y^m)\subset\overline{\D}(\zeta,\rho_m/2)$, where $L(x^m,y^m)$ is the component of $\partial D_\infty\setminus\{x^m,y^m\}$ with smaller diameter.

Let $x_0^m,y_0^m\in J_0$, respectively, be the preimages of $x_{n_k}^m,y_{n_k}^m$ under $f^{\circ n_k}$ which are contained in a same component of  $f^{-n_k}(\D(\zeta,\rho_m))$. By Koebe's distortion theorem, there is a constant $C>0$ such that
\begin{equation}\label{est-my}
\frac{\diam(L(x_0^m,y_0^m))} {|x_0^m-y_0^m|} \geq C\frac{\diam(L(x_{n_k}^m,y_{n_k}^m))} {|x_{n_k}^m-y_{n_k}^m|}\geq C\frac{\diam(L(x^m,y^m))}{2|x^m-y^m|}
\end{equation}
for large $k\geq k_m$ since $J_{n_k}\to \partial D_\infty$ as $k\to\infty$, where $L(x_0^m,y_0^m)$ is the component of $J_0\setminus\{x_0^m,y_0^m\}$ with smaller diameter. Letting $m\to\infty$ and hence $\rho_m\to 0$, then the right most term in \eqref{est-my} tends to $\infty$ since $\zeta$ is a cusp. This means that $J_0$ is not bounded turning. Equivalently, $J_0$ is not a quasicircle.
This ends the proof of Theorems \ref{thm-regularity} and \ref{thm-regularity-intro}.
\end{proof}

\begin{rmk}\label{rmk}
By the completely similar arguments as Theorem \ref{thm-regularity}, one can prove that each Julia components of $f_{p,d_1,\cdots,d_n}$ in Theorem \ref{thm-QYY} is a quasicircle since $f_{p,d_1,\cdots,d_n}$ is hyperbolic if the parameters $a_1,\cdots,a_{n-1}$ are chosen suitably.

For the parabolic Cantor circles Julia sets, the quasicircle components of $f$ is dense in the space of the Julia components. On the other hand, the non quasicircle components form a residual set in Baire's category (see Figure \ref{Fig_Cantor-quasicircle}).
\end{rmk}

\begin{figure}[!htpb]
  \setlength{\unitlength}{1mm}
  \centering
  \includegraphics[width=70mm]{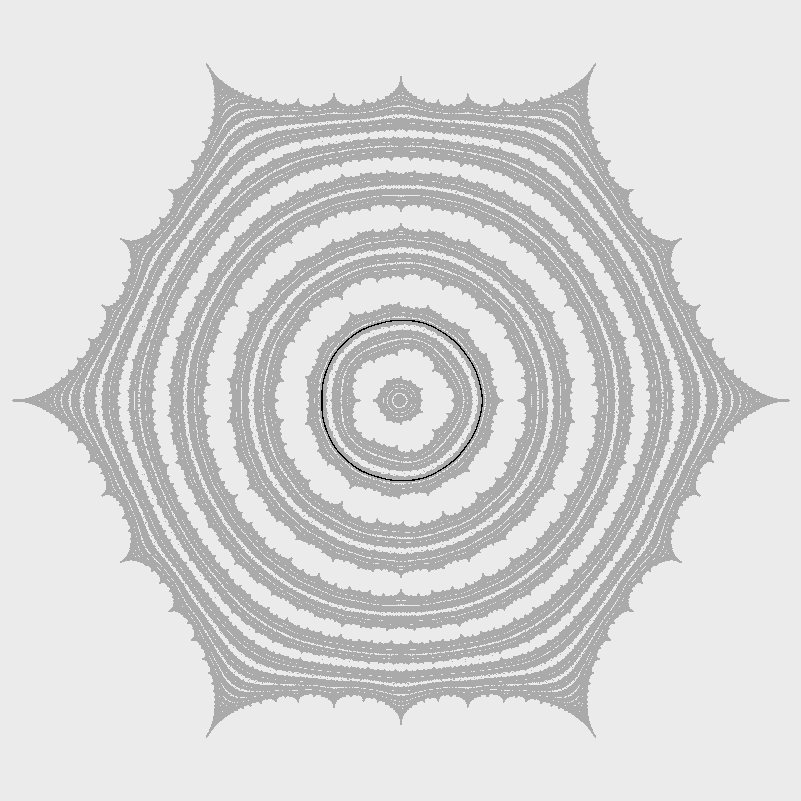}
  \includegraphics[width=70mm]{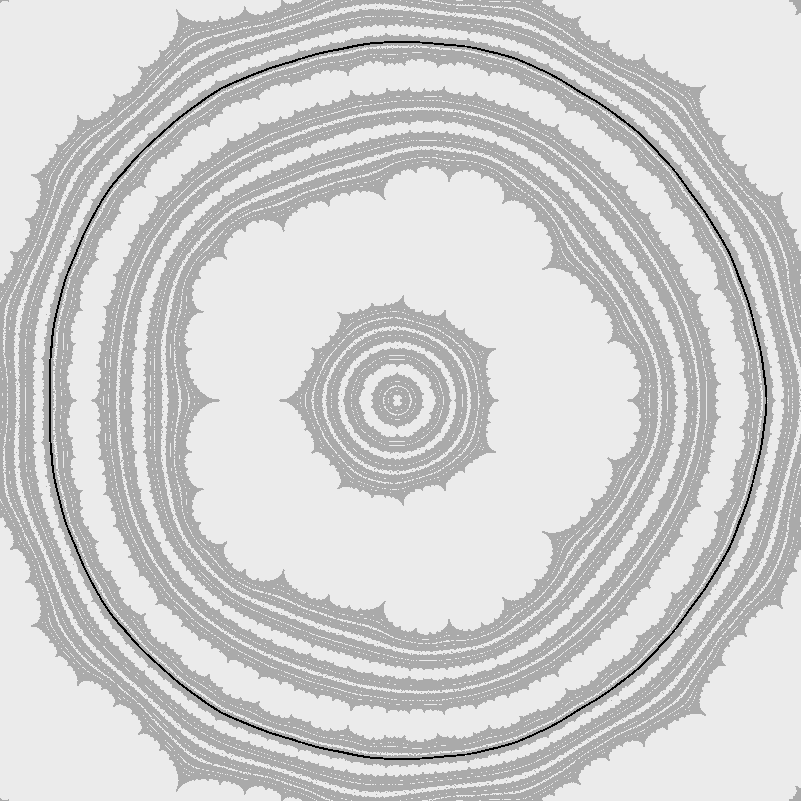}
  \caption{A quasicircle (black circle) in a parabolic Cantor circles Julia set (dark grey) and its partial enlargement.}
  \label{Fig_Cantor-quasicircle}
\end{figure}

In the rest of this section, we will use the language of symbolic dynamics to analyse the regularity of the Julia components, which can be served as a parallel result as Theorem \ref{thm-regularity}. The arguments will be divided into four cases.

For $n\geq 2$, let $\Sigma_n:=\{1,\cdots,n\}^{\mathbb{N}}=\{\s=s_0 s_1 s_2 \cdots|\ 1\leq s_k\leq n, \text{~and~} k\in\mathbb{N}\}$ be the space of $n$ symbols. The \emph{one-sided shift map} $\sigma:\Sigma_n\rightarrow\Sigma_n$ is defined by $\sigma(\s)=s_1 s_2 \cdots$, where $\s=s_0 s_1 s_2 \cdots$. We say the \textit{itinerary} $\s$ is \emph{periodic} if there exists an integer $p>0$ such that $s_{k+p}=s_k$ for all $k\geq 0$, and the minimal $p>0$ is called the \emph{periodic} of $\s$. The itinerary $\s$ is also denoted by $\overline{s_0 \cdots s_{p-1}}$ if it is periodic with period $p$.

Let $f$ be a rational map whose Julia set is Cantor set of circles. The dynamics on the set of Julia components of $P_n$ is conjugate to the one-sided shift on the space of $n$ symbols $\Sigma_n$. For example, let $J_0$ be a Julia component of $P_{d_1,\cdots,d_n}$. If $P_{d_1,\cdots,d_n}^{\circ k}(J_0)\subset E_{s_k}$ for $k\geq 0$, then $\s \circ P_{d_1,\cdots,d_n}(J_0)=\sigma\circ \s (J_0)$, where $\s (J_0)=s_0 s_1 s_2 \cdots\in\Sigma_n$ is the \textit{itinerary} of $J_0$ and $E_1,\cdots,E_n$ are the annuli appeared in Theorem \ref{this-is-all-resta}.

\subsection{Parabolic I.}
This subsection will deal with the first case: a Cantor circles parabolic rational map with only one parabolic fixed point and only one grand orbit of Fatou components. By Theorem \ref{thm_parabolic-conj-std}, there exists a quasisymmetric model $P_{d_1,\cdots,d_n}$ for this type, where $n\geq 2$ is even. This means that we only need to study the regularity of the Julia components of $P_{d_1,\cdots,d_n}$ if we want to study a Cantor circles parabolic rational map with only one parabolic fixed point and only one grand orbit of Fatou components.

For $P_{d_1,\cdots,d_n}$ with even $n\geq 2$, let $E:=\EC\setminus (D_0\cup D_\infty)$, where $D_0$ and $D_\infty$ are the Fatou components of $P_{d_1,\cdots,d_n}$ containing $0$ and $\infty$ respectively. Then $P_{d_1,\cdots,d_n}^{-1}(E)$ consists of $n$ components $E_1,\cdots,E_n$. Note that the Julia set $J(P_{d_1,\cdots,d_n})$ is contained in $E_1\cup\cdots\cup E_n$ (see left picture in Figure \ref{fig_parabolic_I-II}). For every Julia component $J_0$ of $P_{d_1,\cdots,d_n}$, if $P_{d_1,\cdots,d_n}^{\circ k}(J_0)\subset E_{s_k}$ for $k\geq 0$, then the \textit{itinerary} of $J_0$ is defined as $\s (J_0)=s_0 s_1 s_2 \cdots\in\Sigma_n$.

By Theorem \ref{thm-regularity} and the definition of $\s (J_0)$, we have following proposition \ref{thm-regularity-parabolic-I}.

\begin{prop}\label{thm-regularity-parabolic-I}
Let $n\geq 2$ be an even number. The Julia component $J_0$ of $P_{d_1,\cdots,d_n}$ is a quasicircle if and only if the sequence of the lengths of the maximal substrings with successive symbol $1$ in its itinerary $\s (J_0)$ is bounded.
\end{prop}

\begin{figure}[!htpb]
  \setlength{\unitlength}{1mm}
  \centering
  \includegraphics[width=80mm]{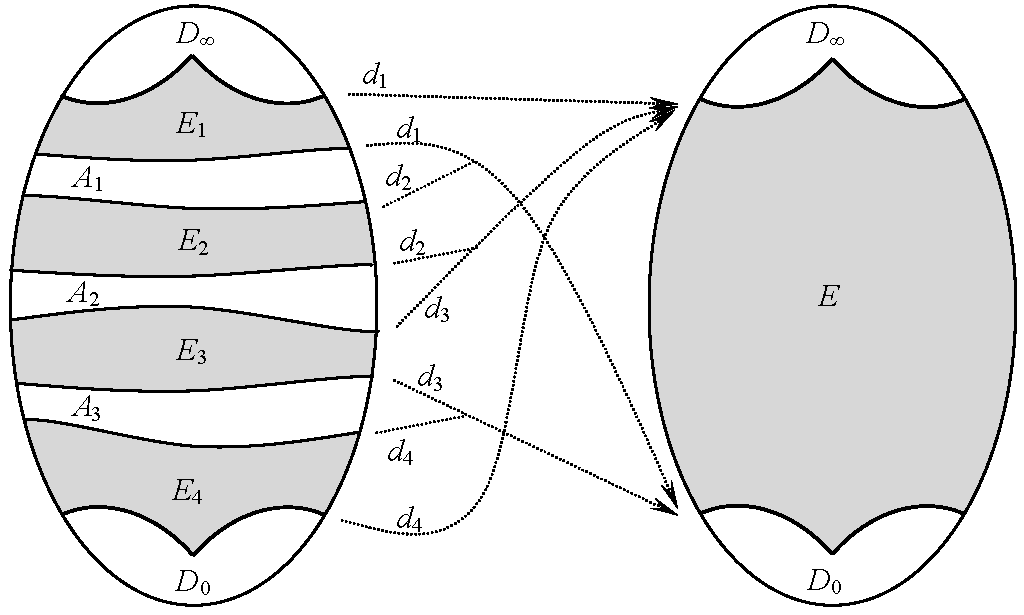}
  \includegraphics[width=80mm]{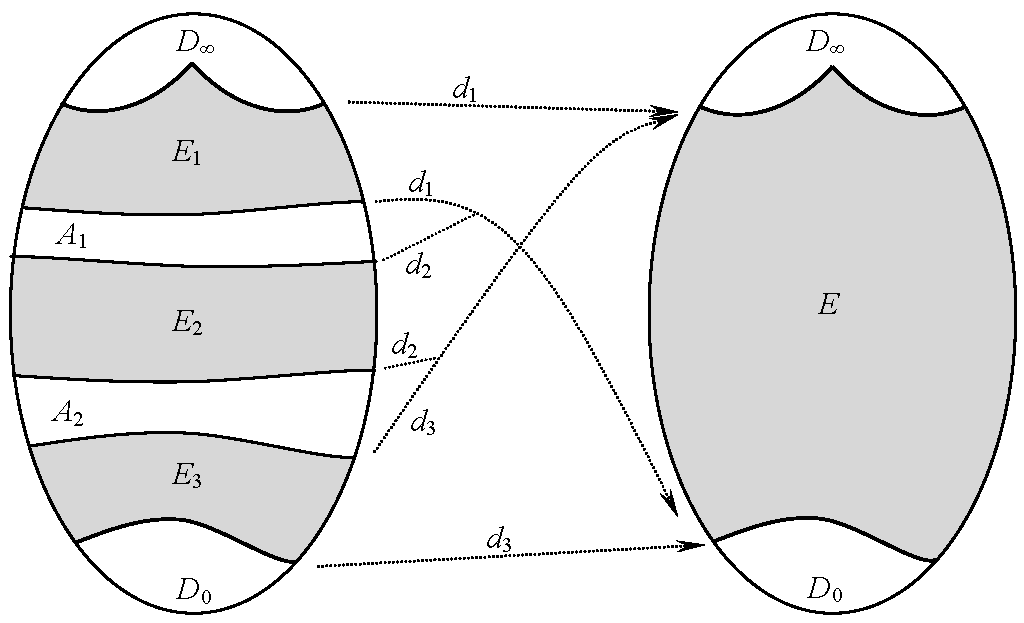}
  \caption{Sketch illustrating of the mapping relations of $P_{d_1,\cdots,d_n}$, where $n=4$ and $n=3$ (from left to right). Both Julia sets are contained in the gray parts.}
  \label{fig_parabolic_I-II}
\end{figure}

For example, let $J_0$ be a Julia component of $P_{d_1,d_2,d_3,d_4}$ such that $\s (J_0)=\overline{1133}$. Then the sequence of the lengths of the maximal substrings with successive symbol $1$ in $\overline{1133}$ is $2,2,2,\cdots$, hence bounded and the component $J_0$ is a quasicircle. If  $\s (J_0)=121121112\cdots$, the component $J_0$ is not a quasicircle since the the sequence of the lengths of the maximal substrings with successive symbol $1$ in $\s (J_0)$ is $1,2,3,\cdots$ and unbounded.

\subsection{Parabolic II.}
This subsection will deal with the second case: a Cantor circles parabolic rational map with one parabolic fixed point and one attracting fixed point. By Theorem \ref{thm_parabolic-conj-std}, there exists a quasisymmetric model $P_{d_1,\cdots,d_n}$ for this type, where $n\geq 3$ is odd.

For $P_{d_1,\cdots,d_n}$ with odd $n\geq 3$, the sets $E$ and $E_1,\cdots,E_n$ and the itinerary of of the Julia components can be defined similarly (see right picture in Figure \ref{fig_parabolic_I-II}). By Theorem \ref{thm-regularity} and the definition of $\s (J_0)$, we have following proposition \ref{thm-regularity-parabolic-II}.

\begin{prop}\label{thm-regularity-parabolic-II}
Let $n\geq 3$ be an odd number. The Julia component $J_0$ of $P_{d_1,\cdots,d_n}$ is a quasicircle if and only if the sequence of the lengths of the maximal substrings with successive symbol $1$ in its itinerary $\s (J_0)$ is bounded.
\end{prop}

Proposition \ref{thm-regularity-parabolic-II} is very similar to Proposition \ref{thm-regularity-parabolic-I} although $P_{d_1,\cdots,d_n}$ with even and odd $n$ have completely different dynamics. For example, let $J_0$ be a Julia component of $P_{d_1,d_2,d_3}$ such that $\s (J_0)=\overline{123}$. Then the sequence of the lengths of the maximal substrings with successive symbol $1$ in $\overline{123}$ is $1,1,1,\cdots$, hence bounded. Therefore, the component $J_0$ is a quasicircle.

\subsection{Parabolic III.}
We will consider the third case in this subsection: a Cantor circles parabolic rational map with two parabolic fixed points. By Theorem \ref{thm_parabolic-conj-std}, there exists a quasisymmetric model $Q_{d_1,\cdots,d_n}$ for this type.  We omit the definitions of the sets $E$ and $E_1,\cdots,E_n$, etc (see left picture in Figure \ref{fig_parabolic_III-IV}). By Theorem \ref{thm-regularity} and the definition of $\s (J_0)$, we have following proposition \ref{thm-regularity-parabolic-III}.

\begin{prop}\label{thm-regularity-parabolic-III}
Let $n\geq 3$ be an odd number. The Julia component $J_0$ of $Q_{d_1,\cdots,d_n}$ is a quasicircle if and only if the sequence of the lengths of the maximal substrings with successive symbol $1$ and $n$ in its itinerary $\s (J_0)$ is bounded.
\end{prop}

\begin{figure}[!htpb]
  \setlength{\unitlength}{1mm}
  \centering
  \includegraphics[width=80mm]{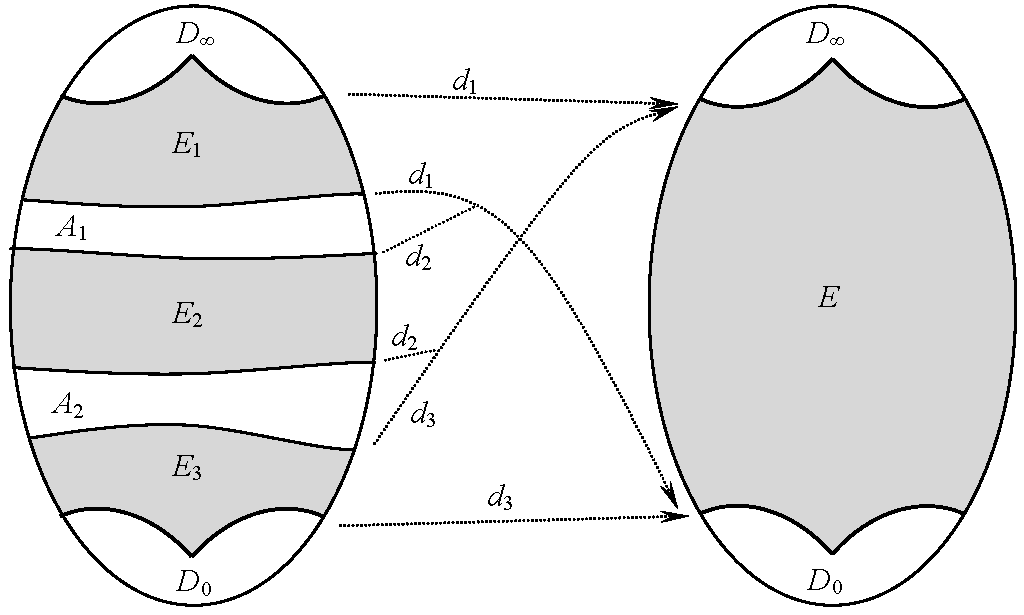}
  \includegraphics[width=80mm]{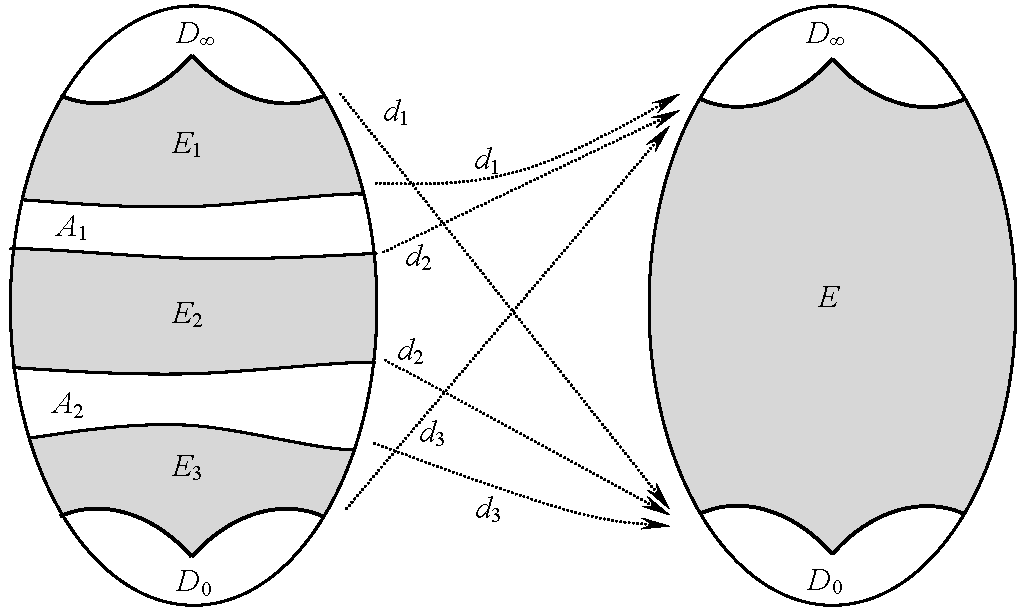}
  \caption{Sketch illustrating of the mapping relations of $Q_{d_1,\cdots,d_n}$ and $R_{d_1,\cdots,d_n}$ (from left to right), where $n=3$. Both Julia sets are contained in the gray parts.}
  \label{fig_parabolic_III-IV}
\end{figure}

For example, let $J_0$ be a Julia component of $Q_{d_1,d_2,d_3}$ such that $\s (J_0)=\overline{122}$. Then the sequence of the lengths of the maximal substrings with successive symbol $1$ in $\overline{122}$ is $1,1,1,\cdots$, hence bounded. Therefore, the Julia component $J_0$ is a quasicircle.

\subsection{Parabolic IV.}
Now we consider the last case: a Cantor circles parabolic rational map with a parabolic periodic orbit whose period is two. By Theorem \ref{thm_parabolic-conj-std}, there exists a quasisymmetric model $R_{d_1,\cdots,d_n}$ for this type. The definitions of the sets $E$ and $E_1,\cdots,E_n$, etc are also omitted here (see right picture in Figure \ref{fig_parabolic_III-IV}). By Theorem \ref{thm-regularity} and the definition of $\s (J_0)$, we have following proposition \ref{thm-regularity-parabolic-IV}.

\begin{prop}\label{thm-regularity-parabolic-IV}
Let $n\geq 3$ be an odd number. The Julia component $J_0$ of $R_{d_1,\cdots,d_n}$ is a quasicircle if and only if the sequence of the lengths of the maximal substrings with successive symbol $1$ and $n$ in its itinerary $\s (J_0)$ is bounded.
\end{prop}

For example, let $J_0$ be a Julia component of $R_{d_1,d_2,d_3}$ such that $\s (J_0)=\overline{112}$. Then the sequence of the lengths of the maximal substrings with successive symbol $1$ in $\overline{112}$ is $2,2,2,\cdots$, and hence the component $J_0$ is a quasicircle.


\end{document}